\documentclass[a4paper,11pt, final, BCOR=5mm,	pagesize,	pdftex]{scrartcl}
\usepackage[utf8]{inputenc}
\usepackage{a4wide}
\usepackage{csquotes}
\usepackage{verbatim}
\usepackage{float}
\usepackage{underscore}
\usepackage[margin=0.8cm]{caption}
\usepackage[section]{placeins}
\usepackage{subcaption}

\usepackage{amsmath}
\usepackage{amsfonts}
\usepackage{amssymb}
\usepackage{mathrsfs}
\usepackage{mathtools}

\usepackage{enumerate}

\usepackage{graphicx}
\usepackage{pgfplots}
\pgfplotsset{compat=1.8}
\usepackage{setspace,tabularx} 
\usepackage[draft=false,babel,tracking=true,kerning=true,spacing=true]{microtype} 

\usepackage[noadjust]{cite}
\usepackage{hyperref}
\usepackage{color}
\usepackage{dsfont}

\usepackage{fixltx2e} 
\usepackage[T1]{fontenc} 
\usepackage{lmodern} 
\usepackage{bigints} 

\numberwithin{equation}{section}

\providecommand{\keywords}[1]
{
  \small	
  \textbf{\textit{Keywords---}} #1
}

\providecommand{\subjclass}[2]{%
 \small	
  \textbf{\textit{2010 Mathematics subject classification---}} #1
}

\newcommand{\R}{\mathbb{R}}
\newcommand{\N}{\mathbb{N}}
\newcommand{\Z}{\mathbb{Z}}
\newcommand{\F}{\mathcal{F}^{(n)}}
\newcommand{\E}{\ensuremath\mathbb{E}}
\newcommand{\1}{\textup{$\mathds{1}$}}
\newcommand{\Pro}{\ensuremath\mathbb{P}}

\renewcommand{\leq}{\leqslant}
\renewcommand{\geq}{\geqslant}
\renewcommand{\bar}{\overline}
\renewcommand{\tilde}{\widetilde}

\newcommand{\tn}{\Delta t^{(n)}}

\newcommand{\xn}{\Delta x^{(n)}}
\newcommand{\pn}{\Delta p^{(n)}}
\newcommand{\vn}{\Delta v^{(n)}}
\newcommand{\pin}{\pi^{(n)}}
\newcommand{\ttn}{\lfloor t/\tn \rfloor}
\newcommand{\s}{S_{k-1}^{(n)}}

\newcommand{\Bn}{B^{(n)}}
\newcommand{\dBn}{\delta B^{(n)}}

\newcommand{\x}{x^{(n)}}
\renewcommand{\t}{t^{(n)}}

\usepackage{amsthm}

\newtheorem{The}{Theorem}[section]
\newtheorem*{The*}{Theorem}
\newtheorem{Exp}{Example}
\newtheorem{Ass}{Assumption}

\newtheorem{Lem}[The]{Lemma}
\newtheorem{Cor}[The]{Corollary}
\newtheorem{Rmk}[The]{Remark}
\newtheorem{Prop}[The]{Proposition}

\begin{document}

\title{Jump diffusion approximation for the price dynamics of a fully state dependent limit order book model}
\date{}
\author{D\"orte Kreher\footnote{Humboldt-Universit\"at zu Berlin, Germany. E-mail: \url{kreher@math.hu-berlin.de}} \qquad Cassandra Milbradt\footnote{Humboldt-Universit\"at zu Berlin, Germany. E-mail: \url{cassandra.milbradt@gmail.com}}}

\maketitle

\begin{abstract}
We study a microscopic limit order book model, in which the order dynamics depend on the current best bid and ask price and the current volume density functions, simultaneously, and derive its macroscopic high-frequency dynamics. As opposed to the existing literature on scaling limits for limit order book models, we include price changes which do not scale with the tick size in our model to account for large price movement, being for example triggered by highly unforeseen events. Our main result states that, when the size of an individual limit order and the tick size tend to zero while the order arrival rate tends to infinity,  the microscopic limit order book model dynamics converge to two one-dimensional jump diffusion processes describing the prices coupled with two infinite dimensional fluid processes describing the standing volumes at the buy and sell side.
\end{abstract}

 \hspace{-0.26cm} \keywords{limit order book,\, market microstructure,\, high-frequency limit,\, semimartingale convergence,\, jump diffusion}\newline

 \subjclass{60F17,\, 91B70}

\section{Introduction}

Electronic limit order books are widely used tools in economics to carry out transactions in financial markets. They are records, maintained by an exchange or specialist, of unexecuted orders awaiting execution. Motivated by the tremendous increase of trading activities, stochastic analysis provides powerful tools for understanding the dynamics of a limit order book via the description of suitable scaling (``high-frequency'') limits. Scaling limits allow for a tractable description of the macroscopic dynamics (prices and standing volumes) derived from the underlying microscopic dynamics (individual order arrivals). While modelling liquid stock markets, prices are typically approximated in the scaling limit by diffusion processes, cf. for example \cite{BHQ17, HorstKreherDiffusion, CL13, Rosenbaum}.  At the same time, there is a broad consensus that (large) jumps may occur as responses to highly unexpected news, cf. \cite{BC17, F02, JLG08}. In this paper, we specify reasonable microscopic dynamics of a limit order book which takes the possibility of such surprising news flows into account. Introducing suitable scaling constants, we derive the functional convergence of the microscopic model to a coupled system of two one-dimensional jump diffusions and two infinite dimensional stochastic fluid processes, describing the price and volume dynamics, respectively.\par
At any given point in time, a limit order book depicts the number of unexecuted buy and sell orders at different price levels,  cf. Figure \ref{Fig:orderBook}. The highest price a potential buyer is willing to pay is called the best bid price, whereas the best ask price is the smallest price of all placed sell orders. 
Incoming limit orders can be placed at many different price levels, while incoming market orders are matched against standing limit orders according to a set of priority rules.

\begin{figure}[H]
\centering
\includegraphics[scale=0.45]{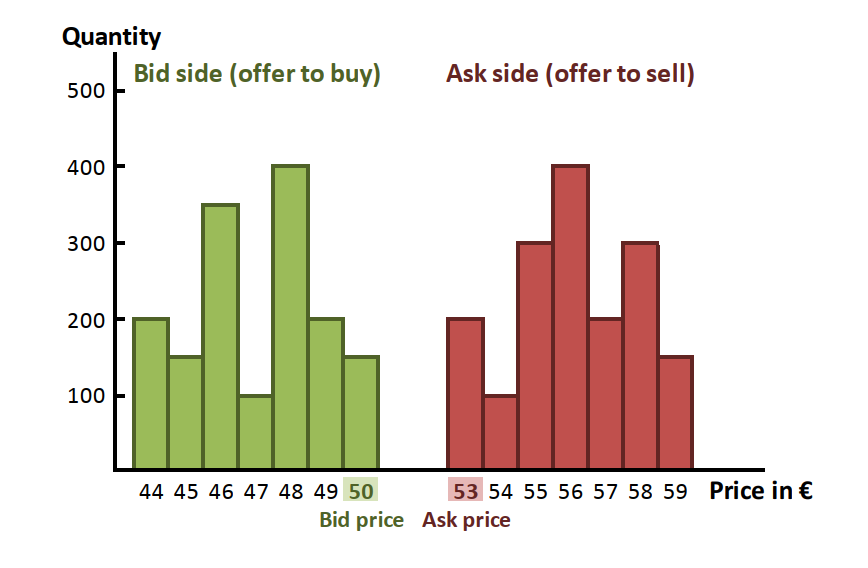}
\caption{\small Illustration of the state of a limit order book model.}
\label{Fig:orderBook}
\end{figure}

One popular approach to study limit order books is based on event-by-event descriptions of the order flow as done in e.g. \cite{L03, CST10, HP17, HorstKreherCLT, HorstKreherDiffusion, HorstKreherFluid, Gao, CL13, KY18}. The derived stochastic systems typically yield realistic models as they preserve the discrete nature of the dynamics at high frequencies, but turn out to be computationally challenging. They also give only little insight into the underlying structure of the order flow. To overcome the drawbacks of these models, some researchers deal with continuum approximations of the order book, describing it through its time-dependent density satisfying either certain partial differential equations as in \cite{LL07, CGGK09, CMP11, BCMW13} or certain stochastic partial differential equations as in \cite{MTW16, ContMueller}.\par
Combining these two approaches, one can introduce suitable scaling constants in the microscopic order book dynamics and study its scaling behaviour when the number of orders gets large while each of them is of negligible size. The scaling limit can then either be described through a system of (partial) differential equations (in the ``fluid'' limit, where random fluctuations vanish) or through a system of stochastic (partial) differential equations (in the ``diffusion'' limit, where random fluctuations dominate) or through a mixture thereof. Deriving a deterministic high-frequency limit for limit order book models guarantees that the scaling limit approximation stays tractable in view of practical applications. Such an approach is pursued by Horst and Paulsen \cite{HP17}, Horst and Kreher \cite{HorstKreherFluid}, and Gao and Deng \cite{Gao}. In contrast, a diffusion limit for the order book dynamics can be found in Cont and de Larrard \cite{CL13} or Horst and Kreher \cite{HorstKreherDiffusion}. While \cite{CL13} only analyses the volumes standing at top of the book, \cite{HorstKreherDiffusion} takes the whole standing volumes into account leading to both, a diffusive price and a diffusive volume approximation. Depending on the market and/or stock of interest either a fluid or a diffusive volume approximation seems to be appropriate. However, a diffusive infinite dimensional volume process makes practical applications more difficult as e.g.~the uniqueness of a solution to the infinite dimensional stochastic differential equation need not guaranteed, cf.~\cite{HorstKreherDiffusion}. On the other hand, absence of arbitrage considerations encourage price approximations by diffusion processes. \par

Our model is a further development of the model considered in \cite{HorstKreherFluid} and \cite{HorstKreherDiffusion}, where the order book dynamics are influenced by both, current bid and ask prices as well as standing volumes of the bid and sell sides. This is a reasonable starting point as there is considerable empirical evidence (see, e.g. \cite{BHS95, CH15, HH12}) that the state of the order book, especially the order imbalance at the top of the book, has a noticeable impact on order dynamics. In \cite{HorstKreherFluid} and \cite{HorstKreherDiffusion} the authors start from an event-by-event description of a limit order market based on the submission of market orders, limit orders, and the cancellation thereof. Their description allows them to write down the evolution of the bid and ask price and the buy and sell side volume density functions, which are both denoted in relative price coordinates. Denoting volumes in relative price coordinates is appealing from a modelling point of view as the empirical distribution of limit order placements at a given distance from the best price is almost stationary, cf.~\cite{BFL09,CST10}. An important simplifying assumption made in \cite{HorstKreherFluid,HorstKreherDiffusion} is that all price changes are assumed to be equal to the tick size and hence become infinitely small in the limit. This seems to be appropriate in
an efficient market setting with high liquidity, cf.~\cite{FGLMS04}. However, if the sizes of all price changes become negligible in the limit, there is no possibility to include jumps in the macroscopic price approximation. Even in highly liquid markets, price jumps occur with positive probability, cf.~\cite{BC17}. As discussed in several empirical studies (cf.~e.g.~\cite{F02, JLG08}) price jumps may be caused by exogenous news. However, most of them cannot be shown to be related to unforeseen news and are understood as endogenous shocks, cf.~\cite{BC17}. No matter what causes large price movements, there is a need for an approximation which takes price jumps into account. In addition, such a model also allows for a reasonable approximation of intraday electricity markets with a continuous trading mechanism (such as the SIDC market), where extreme price spikes during a single trading day can be observed: Figure \ref{Fig:OM18} depicts the EPEX SPOT intraday prices from one single day. Comparing the minimum and maximum intraday prices paid for one single delivery time (products for different delivery times are carried out through different order books), huge differences occur and prices can even become negative. 

\begin{figure}[H]
\centering
\includegraphics[scale=0.45]{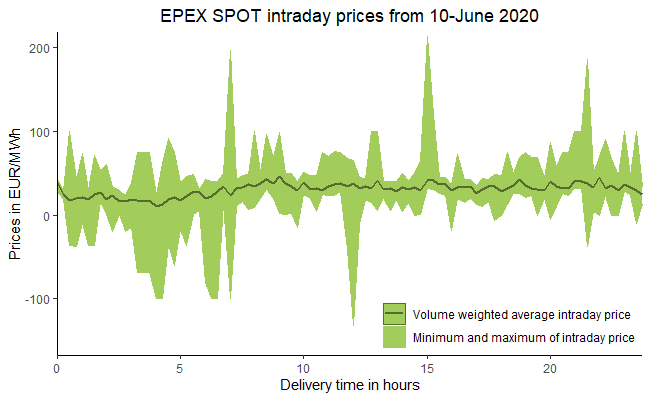}
\caption{\small Extreme price differences occur between minimum and maximum intraday prices in the German intraday electricity market.}
\label{Fig:OM18}
\end{figure}

Motivated by these facts, we extend the results of \cite{HorstKreherFluid,HorstKreherDiffusion} in two ways. First, we take their event-by-event description of a limit order book model and include price changes which do not scale to zero. While doing so, we allow the jump sizes to vary across different states of the LOB, but we require the jump intensities to be approximately the same. This indeed allows for a jump diffusion approximation of the price dynamics in the scaling limit, as the driver becomes independent of the order book dynamics. Second, we combine the diffusive price approximation from \cite{HorstKreherDiffusion} with the fluid approximation for the relative volume density process from \cite{HorstKreherFluid}. Therefore, we end up with two one-dimensional stochastic differential equations describing the bid and ask price dynamics coupled with two infinite dimensional fluid processes approximating the relative volume dynamics of the bid and ask side in the scaling limit. Conditionally on the price movements, the latter behave like deterministic PDEs, since random fluctuations of the queue sizes vanish. However, they are still random because their coefficients depend on the whole limit order book dynamics, including prices. To give the reader an intuition for our model, we perform a simulation study of the full order book dynamics in Section \ref{sec:simStudy}, in which we allow the probabilities of different price movements as well as the limit order sizes and cancellations to depend on the spread and standing volumes, especially on order imbalances. 
\par 
Mathematically, we derive a limit theorem which goes beyond the standard theory of finite-dimensional, diffusive limit processes with continuous sample paths. First, our state process takes values in an infinite dimensional Hilbert space, which requires to apply results from Kurtz and Protter \cite{KurtzProtter2}. Second, while conditions to derive limit theorems towards diffusion processes (in finite dimensions) are generally well-understood, to the best of our knowledge this is not true for a jump diffusion limit. Therefore, one of our main mathematical achievements is to formulate the correct assumptions, which allow to derive such a limit process. Furthermore, to deal with the possible discontinuities of the limit process we need to apply many not so well-known extensions of otherwise well-known results about the convergence of probability measures in the Skorokhod space, which can be found in \cite{W80,W02}. Last but not least, while the correct scaling assumptions to obtain a diffusive-fluid system can readily be deduced from \cite{HorstKreherFluid,HorstKreherDiffusion}, the mixture of the two results requires a totally new idea of the proof.\par

To this end, we first approximate the sequence of discrete-time limit order book models $S^{(n)}$ by a sequence of discrete-time processes $\tilde{S}^{(n)}$, in which we replace the random fluctuations in the volume dynamics by their conditional expectations. This simplifies the subsequent analysis as the new infinite dimensional system $\tilde{S}^{(n)}$ is driven by two independent one-dimensional noise processes. However, the original volume processes $v_b^{(n)}$ and $v_a^{(n)}$ describe the volume dynamics relative to the best bid and ask price, respectively, and therefore do not only depend on the previous order book dynamics, but also on the current prices. This prevents us from directly applying convergence results for infinite dimensional semimartingales.\footnote{Indeed, for this reason in \cite{HorstKreherDiffusion} the absolute volume function is taken as part of the state variable. However, as argued above from a modelling point of view taking the relative volume function is more sensible.} To bypass this problem, we construct approximate order book dynamics with respect to absolute volume functions, denoted $\tilde{S}^{(n),abs}$, through a random shift in the location variable. This allows us to apply results of Kurtz and Protter \cite{KurtzProtter2} on the weak convergence of stochastic integrals in infinite dimensions and to prove that $\tilde{S}^{(n),abs}$ converges weakly in the Skorokhod topology to the unique solution of a coupled diffusion-fluid system. Finally, exploiting the properties of the Skorokhod topology, we can conclude the weak convergence of our original discrete order book dynamics $S^{(n)}$ to $S$ being the unique solution of a coupled diffusion-fluid system. 

\subsection{Jump diffusion approximation: empirical motivation}

Additionally to the references about the occurrence of jumps in equity prices cited above, we provide in this subsection some empirical motivation for our model based on order book data from the European intraday electricity market SIDC. On SIDC electricity contracts with different durations (hour, half-hour, quarter-hour) and delivery times are traded through different order books. In the following we analyse data from the German market area for the hourly product with delivery time $1\,$pm from March 5, 2020.\footnote{The data is publicly available at \url{https://www.epexspot.com/}} Considering the evolution of the best bid and ask prices during the last five hours before closing, we observe the occurrence of an extreme price increase over a short time period shortly after $12\,$pm.  

\begin{figure}[H]
    \centering
    \includegraphics[scale = 0.35]{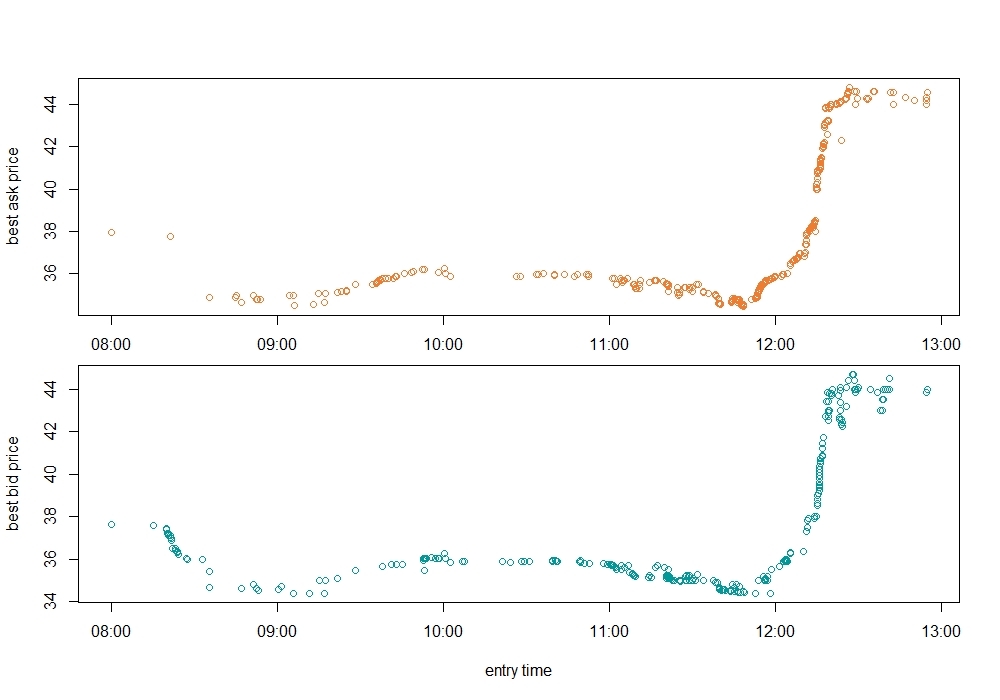}
    \caption{Evolution of the best bid and ask price of the hour product with delivery time $1\,$pm from the German SIDC market area on March 5, 2020.}
    \label{fig:priceEvSIDC}
\end{figure}

To statistically validate the observation of a large price jump in the considered data set, we use the statistical test introduced in Barndoff-Nielsen and Shephard \cite{BNS06}, which is based on bi-power variations. Given data points $X_{0}, \cdots, X_{n}$ at time points $0 = t_0, \cdots, t_n = T$ with fixed interval time length $|t_i - t_{i-1}| = \delta$ for $i = 1, \cdots, n$, the test statistic is given by 
\begin{equation*}
    Z_n := \sqrt{\frac{\delta^{-1}}{\vartheta QPV_n}}\left(RV_n - BPV_n\right)\quad \text{with}\quad \vartheta := (\pi^2/4) + \pi - 5,
\end{equation*}
where $RV_n$ denotes the realized volatility, $BPV_n$ the bi-power variation, and $QPV_n$ the quad-power variation. In \cite{BNS06} it is shown that $Z_n$ is asymptotically standard normally distributed under the null-hypothesis that the data points $X_{0}, \cdots, X_{n}$ are generated by a continuous process $X$. As we consider high-frequency limits as approximations of our price dynamics, we need to analyse the price evolution in a time scale much larger than the interarrival time between price changes. In our data set, the mean interarrival time between price changes is $44$ seconds at the bid and $35$ seconds at the ask price. Therefore, we choose $\delta = 20$ minutes, for which we calculate the values of the test statistic to be $Z_n=15.84$ (bid) and $Z_n=14.27$ (ask). Similar values can be obtained for other durations and delivery times. These values are considerably high and we reject the null-hypothesis that the price data is generated by a continuous process at any reasonable significance level. This provides some first empirical evidence for the occurrence of price jumps in intraday electricity markets and suggests that any reasonable model of intraday electricity price dynamics should take (large) price jumps into account. 

\subsection{Outline of the paper}

The remainder of the paper is structured as follows: Section 2 describes a microscopic, stochastic model for a two-sided limit order book. Moreover, we introduce assumptions under which we are able to establish a scaling limit for the model dynamics and state our main result. In Section 3 we present a simulation study of the full order book dynamics. As our assumption for the large price jumps (cf. Assumption \ref{ass:ExQ} in Section \ref{sec:setup}) might be rather technical at first sight, we provide two examples of jump behaviours in Section 4 which are supported by our model. In Section 5 we state a proof sketch of our main theorem, whereas the technical details are presented in Section 6. The appendix contains some auxiliary results. \medskip

\textbf{Notation.} In the following, $\lambda$ denotes Lebesgue measure and $\varepsilon_x$ denotes Dirac measure at $x \in \R,$ i.e.~for any $A \in \mathcal{B}(\R)$ we have $\varepsilon_x(A) = 1$ if $x \in A$ and $\varepsilon_x(A) = 0$ otherwise. For a discrete-time stochastic process $X := \{X_k: k\in \N_0\}$ let $\delta X_k := X_k - X_{k-1},\ k\in\N,$ denote the $k$-th increment of $X.$ For any continuous-time stochastic process $Y,$ let $\Delta Y(t) := Y(t) - Y(t-)$ denote the jump of $Y$ at time $t>0.$ Last, for any $\R$-valued function $f: E \rightarrow \R,$ we denote by $f^+(x) := \max\{f(x), 0\}$ and $f^-(x) := -\min\{f(x), 0\}$ the positive and negative part of the function $f,$ respectively.

\section{Setup and main result} \label{sec:setup}

In what follows, we fix some finite time horizon $T > 0$ and introduce the Hilbert space
\begin{displaymath}
    E := \R \times L^2(\R) \times \R \times L^2(\R) \times [0,T],\hspace{0.3cm} \|(b,v,a,w,t,y)\|^2_E := |b|^2 + \|v\|^2_{L^2} + |a|^2 + \|w\|^2_{L^2} + |t|^2.
\end{displaymath}
We describe the random evolution of a sequence of limit order book models through a sequence of $E$-valued stochastic processes $S^{(n)} = (B^{(n)}, v^{(n)}_b, A^{(n)}, v^{(n)}_a, \tau^{(n)}),$ where for each $n \in \N$ the $\R$-valued processes $B^{(n)},$ $A^{(n)}$ specify the dynamics of the best bid and ask prices, the $L^2(\R)$-valued processes $v^{(n)}_b,$ $v^{(n)}_a$ specify the dynamics of the buy and ask side volume density functions \textit{relative} to the best bid and ask price and the $[0,T]$-valued process $\tau^{(n)}$ describes the dynamics of the order arrival times. For each $n\in\N$, $S^{(n)}$ is defined on a probability space $(\Omega^{(n)}, \F, \Pro^{(n)}),$ where we will write $\E$ and $\Pro$ instead of $\E^{(n)}$ and $\Pro^{(n)}.$ 

\begin{Rmk}\label{rmk:extInfl}
In order to simplify the subsequent analysis, we include the order arrival times $\tau^{(n)}$ in the state dynamics of our LOB model. In the same manner, one can also include additional exogenous  factor processes: let $Y^{(n)}$ be an $\R^d$-valued stochastic process with almost surely càdlàg sample paths. Then, the state space of the LOB model $(S^{(n)}, Y^{(n)})$ is given by 
\begin{displaymath}
E':= E \times \R^d,\quad \|\cdot\|^2_{E'} := \|\cdot\|^2_{E} + \|\cdot\|^2_{\R^d}.
\end{displaymath}
To derive a high-frequency approximation under this more general setting, additional conditions on the convergence of $Y^{(n)}$  as $n\rightarrow\infty$ have to be satisfied. The factor process $Y^{(n)}$ can be used to model external influences on the LOB dynamics, such as the infeed of renewables in intraday electricity markets, the performance of a stock index in equity markets, or political influences on general market conditions.
\end{Rmk} 

The order book changes due to arriving market and limit orders and due to cancellations, where we differentiate between so-called passive limit orders that are placed on top of standing volumes and aggressive limit orders that are placed inside the spread. In the $n$-th model, the $k$-th order event occurs at a random point in time $\tau^{(n)}_k.$  Throughout, we assume that $\tau^{(n)}_0 = 0$ for all $n\in \N.$ The time between two consecutive order events will tend to zero as $n \rightarrow \infty.$ Furthermore, we introduce the tick size $\xn$ and the average size of a passive limit order placement $\vn$, which are both assumed to tend to zero as $n\rightarrow\infty$. We put $\x_j := j \xn$ for $j \in \Z$ and define the interval $I^{(n)}(x)$ as
\begin{equation}
    I^{(n)}(x) := \left[\x_{j}, \x_{j+1}\right) \quad \text{ for } \x_j \leq x < \x_{j+1}.
\end{equation}
Further, we denote by $\xn \Z := \{\x_j: j \in \Z\}$ the $\xn$-grid. In the following, we specify the sequence of order book models for which we establish a scaling limit when the tick size and the average limit order size tend to zero while the number of order events tends to infinity.\par

In order to model placements of limit orders inside the spread, the relative volume density functions $v^{(n)}_{b}$ and $v^{(n)}_{a}$ are defined on the whole real line. We refer to the volumes standing at negative distances from the best bid/ask price as the shadow book, cf. Figure \ref{fig:shBook}. The idea of the shadow book is taken from \cite{HP17} and was subsequently also used in \cite{BHQ17, HorstKreherFluid, HorstKreherDiffusion, HorstKreherCLT}. It has to be understood as a technical tool to model the conditional distribution of the size of limit order placements inside the spread. Each current volume density function of the visible book is extended in a sufficiently ``smooth'' way to the left to obtain a well-defined scaling limit for the volume functions. The shadow book follows the same dynamics as the volumes of the visible book and becomes part of the visible book through price changes (cf. Example \ref{exp:shadowBook} below). Needless to say, the shadow book cannot be observed in real world markets, but this does not play a role for our analysis, cf.~also Remark \ref{rmk:filtrations} below. 

\begin{figure}[H]
\centering
\includegraphics[scale=0.4]{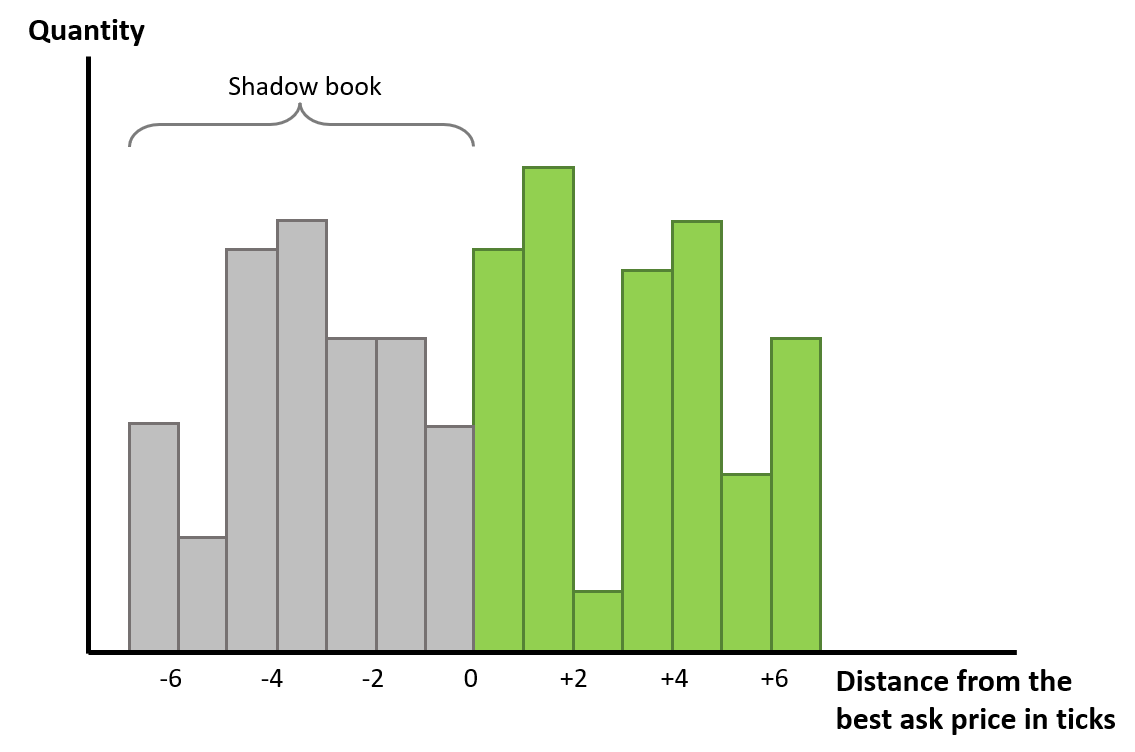}
\caption{Ask-side volume density function in relative coordinates; green: standing volume; grey: shadow book.}
\label{fig:shBook}
\end{figure}

\subsection{The initial state}

In the $n$-th model, the initial state of the limit order book is given by a (positive) best bid price $B^{(n)}_{0}\in \xn\Z$, a (positive) best ask price $A^{(n)}_{0} \in \xn\Z$ satisfying $B_0^{(n)}<A^{(n)}_0$, and non-negative buy and ask side volume density functions $v^{(n)}_{b,0}, v^{(n)}_{a,0} \in L^2(\R)$, which are given \textit{relative} to the best bid and ask price. Here, $v_{b,0}^{(n)}$ and $v_{a,0}^{(n)}$ are supposed to be deterministic step functions which only jump at points in $\xn \Z$. To be precise,
\begin{displaymath}
\int_{\x_j}^{\x_{j+1}} v^{(n)}_{a,0}(x) dx
\end{displaymath}
represents the liquidity available at time $t=0$ for buying (sell side of the book) at a price which is $j \in \N_0$ ticks \textit{above} the best ask price. Similarly,
\begin{displaymath}
\int_{\x_l}^{\x_{l+1}} v^{(n)}_{b,0}(x) dx
\end{displaymath}
gives the liquidity available at time $t=0$ for selling (buy side of the book) at a price which is $l \in \N_0$ ticks \textit{below} the best bid price.\footnote{Note that we choose to model the standing volumes at the bid side through the positive instead of the negative half-line so that we can treat both volume density functions in the same manner.}\par 

At time $t = 0$ the state of the limit order book is deterministic for all $n \in \N$ and is denoted by
\begin{displaymath}
S^{(n)}_0 := \left(B^{(n)}_{0},\, v^{(n)}_{b,0},\, A^{(n)}_{0},\, v^{(n)}_{a,0},\, 0\right) \in E.
\end{displaymath}
In order to prove a convergence result for the microscopic order book sequence to a high-frequency limit, we need to state convergence assumptions on the initial values.

\begin{Ass}[Convergence of the initial states]\label{ass:IV}
There exist a constant $L > 0$ and non-negative functions $v_{b,0}, v_{a,0} \in L^2(\R)$ such that for any $x, \tilde{x} \in \R$ and $I = b,a,$
 \begin{equation}\label{ass:Lcond}
  \left\|v_{I,0}(\cdot + x) - v_{I,0}(\cdot + \tilde{x})\right\|_{L^2} \leq L|x-\tilde{x}|
 \end{equation}
 and 
 \begin{equation}\label{ass:Lconv}
\left\|v_{I,0}^{(n)} - v_{I,0}\right\|_{L^2} \longrightarrow 0.
 \end{equation}
Also, there exist $B_{0}, A_{0} \in \R_+$ with $B_{0} \leq A_{0}$ such that $B^{(n)}_{0} \rightarrow B_{0}$ and $A^{(n)}_{0} \rightarrow A_{0}$. We denote $S_0:= (B_{0},\, v_{b,0},\, A_{0},\, v_{a,0},\, 0,\, Y_0) \in E$.
\end{Ass}

If $v_{I,0} \in L^2(\R),$ $I = b,a,$ is Lipschitz-continuous with constant $L> 0$ and has compact support in $[-M, M]$ for $M>0$, then \eqref{ass:Lcond} is satisfied for the constant $2L\sqrt{M}>0$, i.e. for $x, \tilde{x} \in \R$, we have
\begin{displaymath}
\left\|v_{I,0}(\cdot + x) - v_{I,0}(\cdot + \tilde{x})\right\|_{L^2} \leq L\left\|\1_{[-2M,2M]}(\cdot)|x - \tilde{x}|\, \right\|_{L^2} \leq 2L\sqrt{M} |x-\tilde{x}|.
\end{displaymath}
Furthermore, \eqref{ass:Lcond} and \eqref{ass:Lconv} together imply a similar Lipschitz condition for the $v_{I,0}^{(n)},\ n\in\N$, up to a deterministic sequence $(a_n)_{n\in \N}$ converging to zero: for any $x, \tilde{x} \in \R$, it holds
\begin{displaymath}
\left\|v_{I,0}^{(n)}(\cdot + x)- v_{I,0}^{(n)}(\cdot + \tilde{x})\right\|_{L^2} \leq a_n +L|x - \tilde{x}|.
\end{displaymath}

\subsection{Arrival times and event types}

The order book changes due to incoming order events.
In the $n$-th model, the consecutive times of incoming order events are described by
\begin{equation}\label{dynTimes}
\tau^{(n)}_k = \tau^{(n)}_{k-1} + \varphi^{(n)}_k\tn,\quad k\in\N,
\end{equation}
where $(\varphi^{(n)}_k)_{k\geq 1}$ is a sequence of positive random variables that (scaled by $\tn$) specify the duration between two consecutive order events and $\tn$ goes to zero as $n\rightarrow\infty$. Further, we denote by $N^{(n)}_t$ the random number of incoming order events in $[0,t]$ for any $t\leq T$.

 In the following, we will differentiate between four types of events that may change the state of the book at each time $\tau_k^{(n)}$: 
\begin{enumerate}
    \item[\textbf{A.}] Either a market sell order of size equal to the first $\xi^{(n)}_k$ queues of the bid volumes arrives, which forces the best bid price to decrease by $\xi^{(n)}_k$ ticks and the relative bid side volume density function to shift $\xi^{(n)}_k$ ticks to the left, or an aggressive buy limit order is placed inside the spread, which forces the best bid price to increase by $\xi^{(n)}_k$ ticks and the relative bid side volume density function to shift $\xi^{(n)}_k$ ticks to the right.
    \item[\textbf{B.}] A passive buy limit order placement of size $\frac{\vn}{\xn}\omega^{(n)}_{k}$ at price level $\pi^{(n)}_{k}$ relative to the best bid price occurs. If $\omega^{(n)}_{k} < 0$, this correspond to a cancellation of volume.
    \item[\textbf{C.}] Either a market buy order of size equal to the first $\xi^{(n)}_k$ queues of the ask volumes arrives, which forces the best ask price to increase by $\xi^{(n)}_k$ ticks and the relative ask side volume density function to shift $\xi^{(n)}_k$ ticks to the left, or an aggressive sell limit order is placed inside the spread, which forces the best ask price to decrease by $\xi^{(n)}_k$ ticks and the relative ask side volume density function to shift $\xi^{(n)}_k$ ticks to the right.
    \item[\textbf{D.}] A passive sell limit order placement of size $\frac{\vn}{\xn}\omega^{(n)}_{k}$ at price level $\pi^{(n)}_{k}$ relative to the best ask price occurs. If $\omega^{(n)}_{k} < 0,$ this correspond to a cancellation of volume.
\end{enumerate}\smallskip

The event types $A$ and $C$ lead to price changes of the best bid respectively ask price and will be called active order events. Here, the $\Z$-valued random variable $\xi^{(n)}_k$ specifies the number of ticks the respective price process changes. In contrast, the event types $B$ and $D$ do not lead to changes in the best bid and ask price and will be referred to as passive order events. Here, the $\R$-valued random variables $\omega^{(n)}_k$ and $\pi^{(n)}_k$ specify the size and the location of a placement/cancellation, which does not result in a price change. In the following, event types are determined by a field of random variables $(\phi^{(n)}_k)_{k,n \in \N}$ taking values in the set $\{A, B, C, D\}.$ \par

\subsection{State dynamics of the order book}

The price dynamics of the LOB-models can be described as follows: for each $k,n \in \N$,
\begin{equation}\label{dynPrice}
 B^{(n)}_{k} = B^{(n)}_{k-1} + \1\left(\phi^{(n)}_k = A\right) \xn\xi^{(n)}_{k}, \qquad 
A^{(n)}_{k} = A^{(n)}_{k-1} + \1\left(\phi^{(n)}_k = C \right)\xn\xi^{(n)}_{k}.
\end{equation}
Next, we define the placement/cancellation operator in the following way:
\begin{align}
\begin{aligned}
 M^{(n)}_{b,k}(x):=\frac{\omega^{(n)}_{k}}{\xn}\1\left(x \in I^{(n)}(\pi^{(n)}_{k})\right)\1\left(\phi^{(n)}_k = B\right),\\
 M^{(n)}_{a,k}(x) := \frac{\omega^{(n)}_{k}}{\xn}\1\left(x \in I^{(n)}(\pi^{(n)}_{k})\right) \1\left(\phi^{(n)}_k = D\right).
 \end{aligned}
\end{align}
Then, the dynamics of the volume density function \textit{relative} to the best bid and ask price, respectively, are described by
\begin{align}\label{dynVol}
\begin{aligned}
 v^{(n)}_{b,k}(x) = v_{b,k-1}^{(n)}\left(x - \delta B_{k}^{(n)}\right) + \vn M^{(n)}_{b,k}(x),\\
 v^{(n)}_{a,k}(x) = v_{a,k-1}^{(n)}\left(x + \delta A_{k}^{(n)}\right) + \vn M^{(n)}_{a,k}(x),
 \end{aligned}
\end{align}
for all $x \in \R$. On a first sight, this approach could potentially lead to negative volumes. However, this can be avoided by imposing additional assumptions on the joint conditional distribution of the random variables $\phi_k^{(n)},\omega_k^{(n)},\pi_k^{(n)}$, cf.~Assumption \ref{ass:restLOB} below. For example, volume cancellations can be modelled to be proportional to the standing volume as done in our simulation study in Section \ref{sec:simStudy}.
\par
Note that volume changes take place in the visible or shadow book, depending on the sign of $\pi^{(n)}_k.$ If $\pi^{(n)}_k \geq 0$, then the visible book changes; if $\pi^{(n)}_k <0,$ then the placement/cancellation takes place in the shadow book. The shadow book interacts with the visible book through price changes, which shift the relative volume density functions. The following example illustrates the working of the shadow book.

\begin{Exp}[The shadow book]\label{exp:shadowBook}
    Suppose that the $k$-th incoming order event is a limit order placement into the shadow book one tick above the current best bid price, i.e.
\begin{displaymath}
    \phi^{(n)}_k = B, \quad \pi^{(n)}_k = - \xn, \quad \text{ and } \quad \omega^{(n)}_k > 0.
\end{displaymath}
    Furthermore, suppose that the $(k+1)$-th event is an aggressive buy limit order placement in the spread up to two ticks above the best bid price, i.e. $\phi^{(n)}_{k+1} = A$ and $\xi^{(n)}_ {k+1} = 2.$ Then,
\begin{displaymath}
    B^{(n)}_{k+1} = B^{(n)}_{k} + 2\xn = B^{(n)}_{k-1} + 2\xn
\end{displaymath}
    and for all $x \in \left[\xn, 2\xn\right)$ corresponding to standing volumes one tick above the current best bid price,
 \begin{displaymath}
    v^{(n)}_{b,k+1}(x) = v^{(n)}_{b,k}\left(x - 2\xn\right) = v^{(n)}_{b,k-1}\left(x - 2 \xn\right) + \frac{\vn}{\xn}\omega^{(n)}_k,
 \end{displaymath}
    while for all $x \notin \left[\xn, 2\xn\right),$
\begin{displaymath}
    v^{(n)}_{b,k+1}(x) = v^{(n)}_{b,k}\left(x - 2\xn\right) = v^{(n)}_{b,k-1}\left(x - 2\xn\right).
\end{displaymath}
\end{Exp}

For each $n \in \N$, the microscopic limit order book dynamics are defined through the piecewise constant interpolation
\begin{equation*}
    S^{(n)}(t) := S^{(n)}_k \quad \text{ for } t\in \left[\tau^{(n)}_k, \tau^{(n)}_{k+1}\right)\cap[0,T]
\end{equation*}
of the $E$-valued random variables 
\[S^{(n)}_k := \left(B^{(n)}_k,\, v^{(n)}_{b,k},\, A^{(n)}_k,\, v^{(n)}_{a,k},\, \tau^{(n)}_k\right),\quad k\in\N_0.\]
Finally, for each $k,n \in \N$, we introduce the $\sigma$-field $\F_k:= \sigma(\varphi^{(n)}_j, \phi^{(n)}_j, \omega^{(n)}_j, \pi^{(n)}_j,\xi^{(n)}_j: j \leq k)$ and assume that $S_k^{(n)}$ is $\F_k$-measurable. We denote $\mathcal{B}^{(n)}:= (\Omega^{(n)}, \F, (\F_k)_{k\geq 1}, \Pro^{(n)})$ for all $n\in\N$.

\begin{Rmk}
     The state and time dynamics of our model are driven by the random variables $(\varphi^{(n)}_k)$ (event times), $(\phi^{(n)}_k)$ (event types), $(\omega^{(n)}_{k})$ (sizes of passive orders), $(\pi^{(n)}_{k})$ (relative price levels of passive orders), and $(\xi^{(n)}_k)$ (sizes of price changes). In particular, the process $S^{(n)}$, $n \in \N,$ is adapted to the filtration $\mathbb{G}^{(n)} = (\mathcal{G}^{(n)}_t)_{t \in [0,T]}$, where $\mathcal{G}^{(n)}_t :=\F_{N_t^{(n)}}$ for $t\in[0,T]$.
\end{Rmk}

\subsection{The main result: High-frequency approximation of the microscopic model}

In this section we state our assumptions regarding the distributional properties of the arrival times, price changes, and order placement/cancellations as well as an assumption on the relation between the scaling parameters $\tn, \xn$, and $\vn$, which will allow us to derive a high-frequency limit. We then present our main result. \newline

 First, we assume that the second moment of the unscaled interarrival times is uniformly bounded, which ensures that the random fluctuations of the order arrival time dynamics will vanish in the high-frequency limit. Moreover, we assume that the conditional expectation of each interarrival time only depends on the current state of the order book. 

\begin{Ass}[Conditions on interarrival times]\label{ass:randTimes}
$ $
\begin{enumerate}
    \item[i)] It holds $\sup_{k,n\in \N} \E\left[ \left(\varphi^{(n)}_k\right)^2\right] <\infty.$
    \item[ii)] Moreover, there exist measurable, bounded functions $\varphi^{(n)}: E \rightarrow (0, \infty),\ n\in\N,$ such that
\begin{displaymath}
    \E\left[\left.\varphi^{(n)}_k\right|\, \F_{k-1}\right] = \varphi^{(n)}\left(S^{(n)}_{k-1}\right)\quad a.s.
\end{displaymath}
    \item[iii)] Furthermore, there exists a Lipschitz continuous function $\varphi: E \rightarrow (0, 1]$ with Lipschitz constant $L > 0$ such that
 \begin{displaymath}
 	\sup_{s \in E} \left|\varphi^{(n)}(s) - \varphi(s)\right| \rightarrow 0 \quad \text{ as } n \rightarrow \infty.
 	\end{displaymath}
\end{enumerate}
\end{Ass}

\begin{Rmk}\label{rmk:filtrations}
We note that by definition not only the visible book, but also the shadow book is adapted to the filtration $(\F_k)_{k\in\N_0}$. Therefore, the filtration $(\F_k)_{k\in\N_0}$ (resp.~$\mathbb{G}^{(n)}$) must not be misunderstood as the market filtration. Especially, while from a mathematical point of view Assumption \ref{ass:randTimes} ii) is sufficient for the derivation of a Markovian high-frequency limit process (cf.~Theorem \ref{res:mainTheorem} below), in applications the function $\varphi^{(n)}$ will only depend on the visible book as it is the case in our simulation study, cf.~Section \ref{sec:simStudy}. Moreover, Assumption \ref{ass:randTimes} ii) should be interpreted in the correct way: the conditional expectation of the interarrival times does neither depend on future spread placements through the shadow book (which would be absurd, anyway) nor on the past evolution of the order book, but only on its current state. In the same way Assumptions \ref{ass:probVol} ii) and \ref{ass:Prob} should be understood. 
\end{Rmk}

 Next, we present our assumption on the conditional expectations of the placement/cancellation operator of the volume dynamics. It is of the same flavour as the one for the interarrival times and ensures that also the random fluctuations generated by the volume dynamics will vanish in the high-frequency limit. 

\begin{Ass}[Conditions on placements/cancellations]\label{ass:probVol} $ $
  \begin{enumerate}[i)]
   \item It holds $\sup_{k,n \in \N}  \E\left[\left(\omega^{(n)}_{k}\right)^2\right] < \infty.$
   \item There exist measurable functions $f_b^{(n)}, f^{(n)}_a: E \rightarrow L^2(\R)$, $n \in \N,$ such that for all $k \in \N$,
	 \begin{align*}
	      f_b^{(n)}\left[S^{(n)}_{k-1}\right](\cdot) &= \frac{1}{\xn} \E\left[\left. \omega_{k}^{(n)}  \1\left(\phi^{(n)}_k = B\right) \1\left(\cdot \in I^{(n)}(\pin_{k})\right)\right| \F_{k-1}\right] \quad \text{a.s.}\\
	      f_a^{(n)}\left[S^{(n)}_{k-1}\right](\cdot) &= \frac{1}{\xn} \E\left[\left. \omega_{k}^{(n)}  \1\left(\phi^{(n)}_k = D\right) \1\left(\cdot \in I^{(n)}(\pin_{k})\right)\right| \F_{k-1}\right] \quad \text{a.s.}
	 \end{align*}
   \item There exist bounded, Lipschitz continuous functions $f_b, f_a: E \rightarrow L^2(\R)$ with Lipschitz constant $L>0$ such that
	 \begin{displaymath}
	      \sup_{s\in E} \left\{\left\|f_b^{(n)}[s] - f_b[s]\right\|_{L^2} + \left\|f_a^{(n)}[s] - f_a[s]\right\|_{L^2}\right\} \rightarrow 0
	 \end{displaymath}
	 and for any $x, \tilde{x} \in \R$, $I = b,a,$
	 \begin{displaymath}
	 \sup_{s\in E} \left\|f_I[s](\cdot + x) - f_I[s](\cdot + \tilde{x})\right\|_{L^2} \leq L|x - \tilde{x}|\quad \text{and}\quad
	 \sup_{s\in E}\|f_I[s](\cdot)\1_{[r,\infty)}(|\cdot|)\|_{L^2}\stackrel{r\rightarrow\infty}{\longrightarrow}0.
	 \end{displaymath}
  \end{enumerate}
\end{Ass}

As we aim to derive a jump-diffusion-type limit for the price dynamics, the assumption for the price changes will be of a different form. We differentiate between so-called small and large price changes. Small price changes are assumed to become negligible as the number of orders gets large. This framework is analysed in \cite{HP17,HorstKreherFluid, HorstKreherDiffusion, HorstKreherCLT}, where all price changes are assumed to be equal to $\pm\xn$ and $\xn \rightarrow 0.$ In order to take more extreme price movements into account, we include large price changes in our model, which do not converge to zero as $n \rightarrow \infty.$ The following assumption introduces the scaling of the conditional first and second moments of the small price jumps and the scaling of the conditional probabilities of the large price jumps.

\begin{Ass}[Conditions on the price changes]\label{ass:Prob}
Let $M>0$ and $(\delta_n)_{n \in \N} \subset \R_+$ be a null sequence that satisfies $\xn\leq\delta_n$ for all $n\in\N$.\smallskip

\begin{enumerate}[i)]
 \item There exist bounded, measurable functions $p^{(n)}_b, p^{(n)}_a: E \rightarrow \R$ and $r^{(n)}_b, r^{(n)}_a: E \rightarrow \R_+$ such that for all $k \in \N$,
\begin{align*}
\E\left[\left(\xi^{(n)}_{k}\right)^2\1\left(\phi_k^{(n)}=A\right)\1\left(0 < \xn\left|\xi^{(n)}_{k}\right| \leq \delta_n \right) \Big| \, \F_{k-1}\right] &= \frac{\tn}{(\xn)^2} \left(r^{(n)}_b(S^{(n)}_{k-1})\right)^2\quad \text{a.s.}\\
\E\left[\left(\xi^{(n)}_{k}\right)^2\1\left(\phi_k^{(n)}=C\right)\1\left(0 < \xn\left|\xi^{(n)}_{k}\right| \leq \delta_n \right)\Big| \, \F_{k-1}\right] &= \frac{\tn}{(\xn)^2} \left(r^{(n)}_a(S^{(n)}_{k-1})\right)^2\quad \text{a.s.}
\end{align*}
and 
\begin{align*}
\begin{aligned}
\E\left[\xi^{(n)}_{k} \1\left(\phi_k^{(n)}=A\right)\1\left(0 < \xn\left|\xi^{(n)}_{k}\right| \leq \delta_n \right)  \Big|\, \F_{k-1}\right] &= \frac{\tn}{\xn} p^{(n)}_b(S^{(n)}_{k-1})\quad \text{a.s.}\\
\E\left[\xi^{(n)}_{k} \1\left(\phi_k^{(n)}=C\right)\1\left(0 < \xn\left|\xi^{(n)}_{k}\right| \leq \delta_n \right)  \Big|\, \F_{k-1}\right] &= \frac{\tn}{\xn} p^{(n)}_a(S^{(n)}_{k-1})\quad \text{a.s.}
\end{aligned}
\end{align*}
Further, there exists another null sequence $(\eta_n)_{n\in \N} \subset \R_+$ with $\delta_n/ \eta_n \rightarrow 0$ such that for all $n \in \N$, $s\in E,$ we have
\begin{equation*}
\min\left\{r^{(n)}_b(s), r^{(n)}_a(s)\right\} > \eta_n.
\end{equation*}
  \item For all $n\in\N$ and $j \in \Z,$ there exist measurable, bounded functions $k^{(n)}_{b,j}, k^{(n)}_{a,j}: E \rightarrow \R_+$ with $k^{(n)}_{b,j} \equiv k^{(n)}_{a,j} \equiv 0$ whenever $|\x_j|\leq\delta_n$, such that
\begin{align*}
\begin{aligned}
\Pro\left[\xn \xi^{(n)}_{k} =\x_j,\ \phi_k^{(n)}=A \big|\ \F_{k-1}\right] &= \tn k^{(n)}_{b,j}(S^{(n)}_{k-1}),\\
\Pro\left[\xn \xi^{(n)}_{k} =\x_j,\ \phi_k^{(n)}=C \big|\ \F_{k-1}\right] &= \tn k^{(n)}_{a,j}(S^{(n)}_{k-1}).\\
\end{aligned}
\end{align*}
\end{enumerate}
\end{Ass}

The null sequence $\delta_n$ separates the price changes into two regimes: first, we have the regime of small prices changes becoming negligible in the limit. The second one describes those which do not scale to zero. The null sequence $\eta_n$ is introduced for technical reasons only to guarantee that the diffusion component does not vanish in the pre-limit.

\begin{Rmk} \label{rmk:priceScaling}
	We note that Assumption \ref{ass:Prob} i) is a generalization of the assumptions made in \cite{HP17,HorstKreherFluid, HorstKreherDiffusion, HorstKreherCLT}, where a further scaling parameter $\pn = o(1)$ is introduced that controls the proportion of price changes among all events. Ensuring that this proportion relates to the other scaling parameters as $(\xn)^2 \pn \approx \tn$, Assumption 2.1 in \cite{HorstKreherDiffusion} implies indeed a scaling of order $\tn$ for the second moments of price changes as we demand in the first two equations in Assumption \ref{ass:Prob} i). Note however, that Assumption \ref{ass:Prob} does not necessarily imply that the proportion of price changing events converges to zero. Indeed, suppose that all four events happen with equal probability independently of anything else and that $\delta_n=\xn$. Then Assumption \ref{ass:Prob} i) is satisfied with $\tn=(\xn)^2$. Furthermore, our small price changes can be of order larger than $\xn$ if the probability of price changing events goes to zero. To see this, suppose that $A$ and $C$ events occur with equal probability $\tn/\xn$ (which goes to zero by Assumption \ref{ass:Scaling} below) and that $|\xi^{(n)}_k|\approx(\xn)^{-1/2}$, in which case Assumption \ref{ass:Prob} i) also holds true.
\end{Rmk}

The next assumption guarantees that the coefficient functions $p^{(n)}_b$, $p^{(n)}_a,$ $r^{(n)}_b$, and $r^{(n)}_a$ satisfy the right limiting behaviour.

\begin{Ass}[Convergence assumptions corresponding to the small jumps]\label{ass:limitFct}
    There exist bounded, Lipschitz continuous functions $p_b, p_a: E \rightarrow \R$ and $r_b, r_a: E \rightarrow \R_+$ with Lipschitz constant $L>0$ such that as $n\rightarrow \infty,$ it holds that
\begin{displaymath}
    \sup_{s\in E} \left\{\left|p^{(n)}_b(s) - p_b(s)\right| + \left|p^{(n)}_a(s) - p_a(s)\right| + \left|r^{(n)}_b(s) - r_b(s)\right| + \left|r^{(n)}_a(s) - r_a(s)\right|\right\} \rightarrow 0.
\end{displaymath}
\end{Ass}

 Next we need to specify assumptions that guarantee the convergence of the large jumps. To this end, we first construct kernels $K^{(n)}_b, K^{(n)}_a: E \times \R \rightarrow \R_+$ representing the conditional distributions of the large price changes by setting
\begin{equation}\label{def:coefFct-lP}
	K^{(n)}_b(s, A) := \sum_{j\in \Z} \1_A(\x_j) k_{b,j}^{(n)}(s), \quad K^{(n)}_a(s, A) := \sum_{j\in \Z} \1_A(\x_j) k_{a,j}^{(n)}(s)
\end{equation}
for $A \in \mathcal{B}(\R)$ and $s\in E$. In particular, for $s\in E$ and $I = b,a,$ we have $K^{(n)}_I(s, I^{(n)}(x)) = k_{I,j}^{(n)}(s)$ if $\x_{j} \leq x < \x_{j+1}$. The following assumption guarantees that in the limit the driving jump measures do not depend on the order book dynamics, which is necessary to derive a jump diffusion for the prices as opposed to more general (and more complex) semimartingale dynamics in the limit. We will assume that the driving jump measures have compact support in $[-M,M]$. To define their discrete approximations later on, we introduce for all $n\in\N$ the set $\Z^{(n)}_M:=\{j\in \Z:\ -M\leq x_j^{(n)}\leq M\}.$

\begin{Ass}[Convergence assumptions corresponding to the large jumps]\label{ass:ExQ} There exist kernels $K_b, K_a: E\times\R\rightarrow\R_+$ satisfying $K_b(s,\{0\}) = K_a(s, \{0\}) =0$ for all $s\in E$ as well as finite measures $Q_b,$ $Q_a$ on $\mathcal{B}(\R)$ with compact support in $[-M,M]$ satisfying $Q_b(\{0\}) = Q_a(\{0\})=0$ and measurable, bounded functions $\theta_b, \theta_a: E\times[-M,M]\rightarrow\R$ such that for $I = b,a,$\smallskip

\begin{enumerate}[i)]
    \item the family $\theta_I(s,\cdot),\ s\in E,$ is uniformly equicontinuous in $x$
 and for every $y\in [-M,M]$ either $\theta_I(s,y)=0$ or $x\mapsto\theta_I(s,x)$ is strictly increasing in an open neighbourhood of $y$\footnote{In fact, strictly decreasing would work as well. For the ease of exposition we consider the increasing case.};
    \item for all $s\in E$ and all $A\in \mathcal{B}([-M,M])$,
\begin{equation}\label{eq:Q}
Q_I(A)=K_I(s,\theta_I(s,A))+Q_I(\{x\in A:\theta_I (s,x)=0\}),
\end{equation}
where $\theta_I(s,A):=\{\theta_I(s,x): x\in A\}$\footnote{For notational simplicity, we write $\theta_I(s, [\x_{j}, \x_{j+1})), j \in \Z$, and always think of the well-defined sets $\theta_I(s, [\x_{j}, \x_{j+1}) \cap [-M,M]), j \in \Z$, for all $s\in E$.}.
    \item \[\sup_{s\in E}\sum_j \left|K_I^{(n)}\left(s,\theta_I\Big(s,\left[x_{j}^{(n)},x_{j+1}^{(n)}\right)\Big)\right)-K_I\left(s,\theta_I\Big(s,\left[x_{j}^{(n)},x_{j+1}^{(n)}\right)\Big)\right)\right| \longrightarrow 0.\]
    \item There exists a constant $L> 0$ such that for all $s, \tilde{s} \in E$ and $y\in[-M,M]$,
\[\left|\theta_I(s,y) - \theta_I(\tilde{s}, y)  \right| \leq L\|s-\tilde{s}\|_E.\]    
\end{enumerate}
Moreover, setting for all $n\in\N$,
\begin{equation}\label{def:thetaN1}
    \theta_I^{(n)}(s,x_j^{(n)}):=\left\lceil\frac{\theta_I(s,x_j^{(n)})}{\xn}\right\rceil\cdot\xn,\quad s\in E,\ j\in\Z^{(n)}_M,
\end{equation}
we suppose that
\begin{enumerate}[v)]
    \item for all $s\in E$ and $i\in\N$ with $K_I^{(n)}(s,\{x_i^{(n)}\})>0$ there exists a unique $j\in\Z^{(n)}_M$ such that 
    \[x_i^{(n)}=\theta_I^{(n)}(s,x_j^{(n)});\]
\item[vi)] \[\sup_{s\in E}\sum_{j}\int_{[x_{j}^{(n)},x_{j+1}^{(n)})} \left|\1\Big(\theta_I^{(n)}(s,x_j^{(n)})=0\Big)-\1\Big(\theta_I(s,x)=0\Big)\right|Q_I(dx)\longrightarrow 0.\]
\end{enumerate}
\end{Ass}

 Together, part ii) and part iii) of Assumption \ref{ass:ExQ} require the distribution of some transformation of the large jumps to converge to a limit that is independent of the state $s \in E$. This should be compared to Assumption \ref{ass:Prob} i), where we indirectly require that the distribution of the standardized small price changes converges to a standard Gaussian law.

\begin{Rmk}\label{rmk:ExQ}$ $
\begin{enumerate}[i)]
 \item We require $Q_b$ and $Q_a$ to be compactly supported on $[-M,M]$, so that we can take the separable space $C_b([-M,M])$ of bounded, continuous functions on $[-M,M]$ as test functions, which will be important to be able to apply the results from Kurtz and Protter \cite{KurtzProtter2}.
 \item For $I = b,a$, Assumption \ref{ass:ExQ} v) asks for bijectivity of the discretized coefficient function $\theta_I^{(n)}$ on the support of the measure $K_I^{(n)}(s,\cdot)$. Of course, we need surjectivity to get a representation of the sum of large jumps as a discrete stochastic integral. Injectivity will allow us to map the large jumps of the respective price process uniquely to the jumps of the corresponding integrator defined below.
\end{enumerate}
\end{Rmk}

The advantage in working with $Q_b$ and $Q_a$ instead of $K_b$ and $K_a$ is that they are independent of the order book dynamics. The key requirement in Assumption \ref{ass:ExQ} is the validity of equation \eqref{eq:Q}, which may look a little bit mysterious in the beginning. It says that the jump sizes of $K_b,$ $K_a$ may vary across different states $s\in E$, but that the jump intensities stay the same, modulo the modification of jump sizes, as long as the jump size does not vanish to zero. In Section \ref{sec:jumpBe}, we provide explicit examples of different jump behaviours satisfying Assumption \ref{ass:ExQ}. \par
For later use, we will extend the definition of $\theta^{(n)}_I$, $I = b,a$, to the whole interval $[-M,M]$ by linear interpolation, i.e.~we set for all $s \in E$, $\x_j \leq x < \x_{j+1}$,
\begin{equation}\label{def:thetaN2}
    \theta^{(n)}_I(s,x):=\theta^{(n)}_I(s,x_j^{(n)})+\frac{x-x_j^{(n)}}{\xn}\left(\theta^{(n)}_I(s,x_{j+1}^{(n)})-\theta^{(n)}_I(s,x_j^{(n)})\right).
\end{equation}

The next assumption introduces the crucial relation between the different scaling parameters $\tn, \xn$, and $\vn$. It is a mixture of the scaling assumption in \cite{HorstKreherFluid} for the parameters $\tn$ and $\vn$ and the one in \cite{HorstKreherDiffusion} for the parameters $\tn$, $\xn$, and $\pn$; the latter, however, occurs only implicitly in Assumption \ref{ass:Prob}, cf.~Remark \ref{rmk:priceScaling}.
Because of this assumption, an approximation of the price dynamics by a jump diffusion together with an approximation of the volume dynamics by a fluid process can be obtained in the high-frequency limit. 

\begin{Ass}[Relation between the scaling parameters]\label{ass:Scaling}
We have
\[\lim_{n\rightarrow \infty} \frac{\tn}{\xn} = 0\quad\text{and} \quad \lim_{n\rightarrow \infty }\frac{\tn}{\vn} = 1.\]
\end{Ass}

\begin{The}[Main result]\label{res:mainTheorem}
 Under Assumptions \ref{ass:IV}--\ref{ass:Scaling} the microscopic LOB-dynamics $S^{(n)}$ converges weakly in the Skorokhod topology to $S= \eta \circ \zeta$, where $\zeta(t) := \inf\{s > 0: \tau^{\eta}(s) > t\},\ t\in[0,T]$, is a random time change and $\eta := (B^{\eta}, v^{\eta}_b, A^{\eta}, v^{\eta}_a, \tau^{\eta})$ is the unique strong solution of the coupled diffusion-fluid system
 \begin{align}\label{eq:eta}
 \begin{aligned}
  B^{\eta}(t) &= B_{0} + \int_0^t p_b(\eta(u))du + \int_0^t r_b(\eta(u)) dZ_{b}(u)\\
  &\hspace{4.5cm}+ \int_0^t \int_{[-M,M]} \theta_b(\eta(u-),y) \, \mu^{Q}_b(du, dy),\\
  v^{\eta}_b(t,x) &= v_{b,0}(x-(B^{\eta}(t)-B_{0})) + \int_0^t f_b[\eta(u)](x-(B^{\eta}(t) -B^{\eta}(u))) du,\\
  A^{\eta}(t) &= A_{0} + \int_0^t p_a(\eta(u))du + \int_0^t r_a(\eta(u)) dZ_a(u)\\
  &\hspace{4.5cm} + \int_0^t \int_{[-M,M]} \theta_a(\eta(u-),y) \, \mu^{Q}_a(du, dy),\\
  v^{\eta}_a(t,x) &= v_{a,0}(x+A^{\eta}(t)-A_{0}) + \int_0^t f_a[\eta(u)](x+A^{\eta}(t) -A^{\eta}(u)) du,\\
  \tau^{\eta}(t) &= \int_0^t \varphi(\eta(u))du,
 \end{aligned}
 \end{align}
for all $t \in [0,T]$, $x \in \R$, where $Z_b,$ $Z_a$ are independent standard Brownian motions and $\mu^{Q}_b,$ $\mu^{Q}_a$ are independent homogeneous Poisson random measures with intensity measures $\lambda \otimes Q_b$ and $\lambda \otimes Q_a$, independent of $Z_b, Z_a$.
\end{The}

\begin{Rmk}
Let the assumptions of Theorem  \ref{res:mainTheorem} be satisfied and suppose that there exist functions $h_b, h_a:E\rightarrow\R$ such that $\theta_I(s,x)=h_I(s)x$, $I = b,a,$ for all $s\in E$ and $x\in[-M,M]$. Then the dynamics of the prices simplifies to
\begin{align*}
    B^{\eta}(t) &= B_{0} + \int_0^t p_b(\eta(u))du + \int_0^t r_b(\eta(u))dZ_b(u) + \int_0^t h_b(\eta(u-))dL_b(u),\\
    A^{\eta}(t) &= A_0 + \int_0^t p_a(\eta(u))du + \int_0^t r_a(\eta(u))dZ_a(u) + \int_0^t h_a(\eta(u-))dL_a(u)
\end{align*}
for $t\in [0,T],$ where $L_b, L_a$ are one-dimensional, independent L\'evy processes with jumps in $[-M,M]$. 
\end{Rmk}

\begin{Cor}\label{cor:mainTheorem}
 Let the assumptions of Theorem \ref{res:mainTheorem} be satisfied. Further assume that for $I = b,a$ the functions $v_{I,0}: \R \rightarrow \R_+$ and $f_I[s]: \R \rightarrow \R_+,\ s\in E,$ are twice continuously differentiable. Then the microscopic LOB-dynamics $S^{(n)}$ converges weakly in the Skorokhod topology to $S = (B, v_b, A, v_a, \tau),$ starting in $S_0,$ and being the unique solution of the following coupled SDE--SPDE system: for $(t,x) \in [0,T] \times \R$,
 \begin{align}\label{eq:sde-spde}
     \begin{aligned}
      dB(t) &= p_b(S(t)) (\varphi(S(t)))^{-1}dt + r_b(S(t)) \zeta^{1/2}(t) d\tilde{Z}_b(t) + \int_{[-M,M]} \theta_b(S(t-), y) \tilde{\mu}^{Q}_b(dt, dy),\\
      dv_b(t, x) &= \left(-\frac{\partial v_b}{\partial x}(t, x) p_b(S(t)) + \frac{1}{2}\frac{\partial^2 v_b}{\partial x^2}(t,x) (r_b(S(t)))^2 + f_b[S(t)](x)\right) (\varphi(S(t)))^{-1}dt\\
      &\hspace{2cm} -\frac{\partial v_b}{\partial x}(t, x) r_b(S(t)) \zeta^{1/2}(t) d\tilde{Z}_b(t) + \left(v_b(t-,x-\Delta B(t)) - v_b(t-,x)\right),\\
      dA(t) &= p_a(S(t)) (\varphi(S(t)))^{-1}dt + r_a(S(t)) \zeta^{1/2}(t) d\tilde{Z}_a(t) + \int_{[-M,M]} \theta_a(S(t-), y) \tilde{\mu}^{Q}_a(dt,dy),\\
       dv_a(t, x) &= \left(\frac{\partial v_a}{\partial x}(t, x) p_a(S(t)) + \frac{1}{2}\frac{\partial^2 v_a}{\partial x^2}(t,x) (r_a(S(t)))^2 + f_a[S(t)](x)\right) (\varphi(S(t)))^{-1}dt\\
      &\hspace{2cm} + \frac{\partial v_a}{\partial x}(t, x) r_a(S(t)) \zeta^{1/2}(t) d\tilde{Z}_a(t) + \left(v_a(t-,x+\Delta A(t)) - v_a(t-,x)\right),\\
      d\tau(t)&=dt,
     \end{aligned}
 \end{align}
 where $\tilde{Z}_I,\ I=b,a,$ are independent Brownian motions and $\tilde{\mu}^{Q}_b$ and $\tilde{\mu}^{Q}_a$ are independent, integer-valued random jump measures with compensators $\tilde{\nu}^{Q}_I(dt, dy)= (\varphi(S(t)))^{-1}dt \times Q_I(dy)$ for $I = b,a,$ and $\zeta(t) = \int_0^t(\varphi(S(u)))^{-1}du.$
\end{Cor}

Note that the above SPDEs for the volume processes are degenerate. If we condition the volume dynamics on the price movements, they behave like deterministic PDEs, since random fluctuations of the queue sizes vanish in the high-frequency limit. 
\par 

In order to guarantee that the bid and ask price, the spread, and the volume density functions do not become negative, certain conditions on the joint distribution of the driving variables have to be satisfied, which are specified in the following assumption.

\begin{Ass}[Conditions to guarantee non-negative prices, spread, and volumes]\label{ass:restLOB}
$ $
\begin{enumerate}
    \item[i)] For all $k,n\in\N$ it holds
    \begin{align*}
\Pro\left[\left.\xi_k^{(n)}\geq \frac{A_{k-1}^{(n)}-B_{k-1}^{(n)}}{\xn},\ \phi_k^{(n)}=A\, \right|\F_{k-1}\right] = \Pro\left[\left.\xi_k^{(n)}\leq \frac{B_{k-1}^{(n)}-A_{k-1}^{(n)}}{\xn},\ \phi_k^{(n)}=C\, \right|\F_{k-1} \right] = 0.
\end{align*}
    \item[ii)] For all $k,n\in\N$ it holds that
\[\Pro\left[\left.\xi_k^{(n)}\leq \frac{B_{k-1}^{(n)}}{\xn},\ \phi_k^{(n)}=A\, \right|\F_{k-1}\right] = 0.    \]
    \item[iii)] For all $k,n\in\N$ it holds that
    \begin{align*}
\Pro\left[\left.v_{b,k-1}^{(n)}(\pi^{(n)}_{k}) \leq - \omega^{(n)}_{k},\ \phi^{(n)}_k = B\,\right|\F_{k-1}\right] = 
\Pro\left[\left.v_{a,k-1}^{(n)}(\pi^{(n)}_{k}) \leq -\omega^{(n)}_{k}, \ \phi^{(n)}_k = D\, \right|\F_{k-1}\right] = 0.
\end{align*}
\end{enumerate}
\end{Ass}

The following corollary of Theorem \ref{res:mainTheorem} is a direct consequence of the weak convergence result $S^{(n)}\Rightarrow S$ and the characterization of the limit $S$. 

\begin{Cor}\label{cor:nonnegative}
Let the assumptions of Theorem \ref{res:mainTheorem} be satisfied. Then:
\begin{enumerate}
    \item[i)] Under Assumption \ref{ass:restLOB} i), we have
        \[r_a(s)=r_b(s)=0,\quad p_a(s)\geq0\geq p_b(s)\quad\forall\, s=(a,v,a,w,t)\in E,\]
    and the spread stays non-negative, i.e.~for all $t\in[0,T]$, $A(t)\geq B(t)$ a.s. If in addition also $B^{(n)}_0, A^{(n)}_0 \geq 0$ for all $n\in \N$ and Assumption \ref{ass:restLOB} ii) holds, we have 
    \[r_b(s) = 0, \quad p_b(s)\geq0\quad\forall\, s = (0,v,a,w,t) \in E\]
    and the bid and ask prices stay non-negative, i.e.~for all $t\in[0,T]$, $B(t), A(t) \geq 0$ a.s.
    \item[ii)] Under Assumption \ref{ass:restLOB} iii), we have
            \[\|f_b^-[s](\cdot)\1(v(\cdot)=0)\|_{L^2}=0,\quad \|f_a^-[s](\cdot)\1(w(\cdot)=0)\|_{L^2}=0\quad\forall\, s=(b,v,a,w,t)\in E\]
              and both volume density functions are non-negative, i.e.~for all $t\in[0,T]$,  
              \[\|v_a^-(t)\|_{L^2}=\|v_b^-(t)\|_{L^2}=0\quad\text{a.s.}\]
\end{enumerate}
\end{Cor}

\section{Simulation study} \label{sec:simStudy}

In this section, we present a simulation study of the order book dynamics introduced in the previous section. It demonstrates the usefulness of the general dependence structure, where all coefficient functions are allowed to depend on current prices and volumes. Such dependencies are plausible according to the observations in \cite{BHS95, CH15, HH12}. Among others, the simulation study shows the impact of endogenously and exogenously triggered large jumps in the price dynamics.\newline

Let $\pn = o(1)$ denote a scaling parameter that controls the proportion of active order events among all events (cf.~also Remark \ref{rmk:priceScaling}). Further, let us fix some $h> 0.$ For each $s = (b, v_b, a, v_a, t) \in E$, we define the spread $Sp(s)$ and an order imbalance factor $Im(s)$ via
\[Sp(s) := a - b \quad\text{and}\quad Im(s) := \frac{VolBid(s)}{VolBid(s) + VolAsk(s)}\]
with 
\[VolBid(s) := \int_0^h v_b(x) dx \quad \text{ and } \quad VolAsk(s) := \int_0^h v_a(x) dx.\]
In the following, we denote 
$$VolBid^{(n)}_k := VolBid(S^{(n)}_k),\ VolAsk^{(n)}_k := VolAsk(S^{(n)}_k),\ Im^{(n)}_k := Im(S^{(n)}_k),\ Sp^{(n)}_k := Sp(S^{(n)}_k).$$ 
For simplicity, we assume that all small price changes are of size $\pm \xn$, while the sizes of the large price changes may depend on the current state of the book through the spread and the cumulative volumes at the top of the book.\par 
Let us set $\gamma_n(x) := \gamma_1 (x-\xn)$ for some $\gamma_1 > 0$. We allow the probabilities of the small price changes to depend on the current imbalance factor and the current spread as
\begin{align*}
    \Pro\left[\xi^{(n)}_{k+1} = 1,\phi_{k+1}^{(n)}=A \,\big|\, \F_{k}\right] &=  \pn \left\{ \xn Im^{(n)}_{k}\Big(1-e^{-\gamma_n\left( Sp^{(n)}_{k}\right)}\Big)  + \Big(1-e^{-\gamma_n\left(Sp^{(n)}_{k}\right)}\Big)\right\},\\
    \Pro\left[\xi^{(n)}_{k+1} = 1,\phi_{k+1}^{(n)}=C \, \big|\, \F_k\right] &= \pn \left\{\xn Im^{(n)}_k e^{-\gamma_n\left(Sp^{(n)}_k\right)} + \Big(1- e^{-\gamma_n\left(Sp^{(n)}_k\right)}\Big)\right\},
\end{align*}
and
\begin{align*}
    \Pro\left[\xi^{(n)}_{k+1} = -1,\phi_{k+1}^{(n)}=A\,\big| \, \F_{k}\right] &= \pn \left\{ \xn \left(1-Im^{(n)}_{k}\right)e^{-\gamma_n\left( Sp^{(n)}_{k}\right)}+ \Big(1-e^{-\gamma_n\left(Sp^{(n)}_{k}\right)}\Big)\right\}\\
    \Pro\left[\xi^{(n)}_{k+1} = -1,\phi_{k+1}^{(n)}=C \, \big| \, \F_k\right] &= \pn \left\{\xn \left(1-Im^{(n)}_k\right) \Big(1-e^{-\gamma_n\left(Sp^{(n)}_k\right)}\Big) + \Big(1-e^{-\gamma_n\left(Sp^{(n)}_k\right)}\Big)\right\}.
\end{align*}
The above four probabilities sum up to $4\pn \Big(1-e^{-\gamma_n\left(Sp^{(n)}_k\right)}\Big) + \xn \pn$. This choice of the conditional distribution for the occurrence of small price changes guarantees that the bid and ask price do not cross. If the spread is small, the first and second term of these probabilities are approximately of the same size.  
Moreover, the probability of an upward price change is increasing with the imbalance factor, i.e.~a price increase is more likely to occur if the standing volume at the top of buy side is significantly higher than the standing volume at the top of the sell side. This behaviour of the price is motivated by the empirical observations in e.g. \cite{CHW09, YZ15}. Moreover, if the spread is equal to $\xn$ the probabilities of observing an increase in the best bid price or a decrease in the best ask price are zero. For large spreads, the conditional probabilities are all dominated by their second term and hence are all of similar size. Here, the parameter $\gamma_1 >0$ controls the influence of the spread on the order book dynamics. In particular, we obtain the following feedback functions:
\[ p^{(n)}_b\left(S^{(n)}_k\right) = -\left(1-Im^{(n)}_k\right) +\Big(1- e^{-\gamma_n\left(Sp^{(n)}_k\right)}\Big), \hspace{0.3cm} p^{(n)}_a\left(S^{(n)}_k\right) = Im^{(n)}_k - \Big(1- e^{-\gamma_n\left(Sp^{(n)}_k\right)}\Big)\]
and
\[ \left(r^{(n)}_b(S^{(n)}_k)\right)^2 = 2\Big(1-e^{-\gamma_n\left(Sp^{(n)}_k\right)}\Big) + \mathcal{O}(\xn), \hspace{0.2cm} \left(r^{(n)}_a(S^{(n)}_k)\right)^2 =  2\Big(1-e^{-\gamma_n\left(Sp^{(n)}_k\right)}\Big) + \mathcal{O}(\xn).\]

Note that for all $s\in E$ with $Sp(s) = \xn$, the diffusion coefficients vanish, while the drift of the bid price is negative and the drift of the ask price is positive. This guarantees that the prices move apart if the spread equals $\xn$.\par
Next, let us turn to the conditional probabilities of observing large price changes. In our setting, it is important that the jump intensities are independent of the order book dynamics. Nevertheless, the jump sizes are allowed to vary across different states of the book. In order to ensure that the bid and ask prices do not cross, the jump sizes of the bid and ask must depend on the current spread. Moreover, small standing volumes at the top of the bid or ask side increase the size of a large jump. As noted in Remark \ref{rmk:extInfl} the jump behaviour might also be influenced by external factors. To model such external influences, we take a discretized Poisson process $(Y^{(n)}_k)_{k \geq 1}$ with intensity parameter $\sigma > 0$, which only jumps at times $\{\t_k : k =1,\cdots, N^{(n)}_T\}$ and fix some threshold level $\kappa>0$. If $Y^{(n)}_k$ crosses the threshold level, jump sizes increase significantly by a factor $1+\eta_1\geq 1$. Altogether, the  sizes of the large price jumps depending on the current state $s\in E$ and external influence $y \in \N_0$ are modelled as follows: Take $\eta_2>0$ and $j^+_b, j^+_a, j^-_b, j^-_a \in \xn\Z$ with $j^+_b, j^+_a > \xn$ and $j^-_b, j^-_a < -\xn$. Then, we define
\begin{align*}
    J^{+}_b(s,y) &= \min\left\{\rho(y) j^+_b, \ Sp(s)-\xn\right\}, \quad &  J^{+}_a(s,y) &=\left\lfloor\frac{ \rho(y)\eta_2j_a^+}{VolAsk(s)\xn}\right\rfloor\xn,\\
    J^{-}_b(s,y) &= \left\lfloor\frac{\rho(y)\eta_2 j^-_b}{VolBid(s)\xn}\right\rfloor\xn,\quad & J^{-}_a(s,y) &= \max\left\{\rho(y) j^-_a, \ -Sp(s)+\xn \right\},
\end{align*}
where $\rho(y) := 1 + \eta_1 \1(y > \kappa)$ and the parameters $\eta_1, \eta_2$ control the impact of the external factor and the cumulative standing volumes, respectively, on the size of the large jumps.

Now, for non-negative $\lambda^+_b, \lambda^-_b, \lambda^+_a, \lambda^-_a  \in [0,1]$ with $\lambda^+_b + \lambda^-_b + \lambda^+_a + \lambda^-_a = 1$, we set
\begin{align*}
    \Pro\left[\xi^{(n)}_{k+1} = j,\phi_{k+1}^{(n)}=A \, \big| \, \F_{k}\right] &=
    \tn \left\{\lambda^+_b \1\left(\x_j = J^+_b(S^{(n)}_k, Y^{(n)}_k)\right) +  \lambda^-_b \1\left(\x_j = J^-_b(S^{(n)}_k, Y^{(n)}_k) \right)\right\},\\
    \Pro\left[\xi^{(n)}_{k+1} = j,\phi_{k+1}^{(n)}=C \, \big| \, \F_k\right] &= \tn\left\{ \lambda^+_a \1\left(\x_j = J^+_a(S^{(n)}_k, Y^{(n)}_k)\right) +  \lambda^-_a \1\left(\x_j = J^-_a(S^{(n)}_k,Y^{(n)}_k)\right)\right\}.
\end{align*}
Hence, with probability $\tn$ a large price change occurs. These choices of probabilities for the large price jumps yield the following feedback functions: for $I = b,a,$ we have
\[K_I\left(S^{(n)}_{k}, Y^{(n)}_k, dx\right) = \lambda^-_I\varepsilon_{J^-_I(S^{(n)}_k,Y^{(n)}_k)}(dx) + \lambda^+_I \varepsilon_{J^+_I(S^{(n)}_k,Y^{(n)}_k)}(dx), \hspace{0.3cm} Q_I(dx) = \lambda^-_I\varepsilon_{-1}(dx) + \lambda^+_I\varepsilon_1(dx)\]
and
\[\theta_I\left(S^{(n)}_k, Y^{(n)}_k, x\right) = \begin{cases} J^-_I(S^{(n)}_k,Y^{(n)}_k) \quad &: x \in (-\infty, -1]\\
    J^-_I(S^{(n)}_k,Y^{(n)}_k) + \frac{x+1}{2}\{J^+_I(S^{(n)}_k,Y^{(n)}_k) - J^-_I(S^{(n)}_k,Y^{(n)}_k)\} \quad &: x \in (-1,1)\\
    J^+_I(S^{(n)}_k,Y^{(n)}_k) \quad &: x \in [1, \infty)  \end{cases}.\]
In the following, we denote by $p^{(n)}_{k+1} := \Pro[\phi^{(n)}_{k+1} \in \{A,C\}|\F_{k}]$ the conditional probability of a price changing event at time $\t_{k+1}$ which is uniquely determined by the previous equations. Now, let us turn to the limit order placements. For simplicity, we assume that they are always of size $10$ and are normally distributed around the best bid respectively ask price, i.e.
\begin{align*}
    \Pro\left[\phi^{(n)}_{k+1} = B,\ \omega^{(n)}_{k+1} = 10,\ \pi^{(n)}_{k+1} \in dy\, \big| \, \F_{k}\right] &= \left(1-p^{(n)}_{k+1}\right) \left(1-Im^{(n)}_k\right)\frac{1}{2\pi}e^{-y^2} dy,\\
\Pro\left[\phi^{(n)}_{k+1} = D,\ \omega^{(n)}_{k+1} = 10,\ \pi^{(n)}_{k+1} \in dy\, \big| \, \F_{k}\right] &= \left(1-p^{(n)}_{k+1}\right) Im^{(n)}_k \frac{1}{2\pi}e^{-y^2}dy.
\end{align*}
Moreover, cancellation of volume is supposed to be proportional to the current volume: for all $x \leq 0,$
\begin{align*}
    &\Pro\left[\phi^{(n)}_{k+1} = B,\ \omega^{(n)}_{k+1} \in dx,\ \pi^{(n)}_{k+1} \in dy \, \big| \, \F_k\right] = \left(1-p^{(n)}_{k+1}\right) \frac{Im^{(n)}_k}{v^{(n)}_{b,k}(y)} \1_{\left[-v^{(n)}_{b,k}(y), 0\right]}(x) \frac{1}{2\pi} e^{-y^2}dx\, dy,\\
    &\Pro\left[\phi^{(n)}_{k+1} = D,\ \omega^{(n)}_{k+1} \in dx,\ \pi^{(n)}_{k+1} \in dy \, \big| \, \F_k \right] =\left(1-p^{(n)}_{k+1}\right)  \frac{1-Im^{(n)}_k}{v^{(n)}_{a,k}(y)}\1_{\left[-v^{(n)}_{a,k}(y), 0\right]}(x) \frac{1}{2\pi} e^{-y^2} dx\, dy.
\end{align*}
We have chosen the limit order placements and cancellations in such a way that a high imbalance factor results in more order placements at the ask side and more order cancellations at the bid side, while a small imbalance factor leads to more order placements at the bid side and more order cancellations at the ask side. This induces an equalizing effect. Supposing $\vn=\tn$, the coefficient functions of the relative volume densities are given by
\begin{align*}
    f^{(n)}_b\left(S^{(n)}_k, x\right) &= \frac{1-p^{(n)}_{k+1}}{\xn}\frac{1}{2\pi}  \int_{I^{(n)}(x)} \Big\{ 10\cdot  \left(1-Im^{(n)}_k\right) - \frac{v^{(n)}_{b,k}(y)}{2} Im^{(n)}_k\Big\} e^{-y^2} dy,\\
    f^{(n)}_a\left(S^{(n)}_k, x\right) &= \frac{1-p^{(n)}_{k+1}}{\xn}\frac{1}{2\pi} \int_{I^{(n)}(x)} \left\{ 10 \cdot Im^{(n)}_k - \frac{v_{a,k}(y)}{2}\left(1-Im^{(n)}_k\right)\right\} e^{-y^2}dy.
\end{align*}

We run two different simulations of the above specified model. For both, we suppose that $\xn=n^{-1}$, $\pn=n^{-1/2}$, $\tn = (\xn)^2\pn=n^{-5/2}$ and choose the parameters
\[n = 100, \quad h = 0.55, \quad \gamma_1 = 1, \quad \eta_2 = 100, \quad T = 2.\]
In a first simulation we further choose $\eta_1=0$ and $\lambda_b^- = 1,$ $\lambda_b^+= \lambda_a^- = \lambda_a^+ = 0,$ i.e.~only downward jumps at the best bid price are possible and there is no external factor. Moreover, we start with a bid price $B_0 = 6.9$, an ask price $A_0 = 7$, and with a limit order book that has a severe imbalance at time $t = 0:$ standing volumes at the bid side are much higher than standing volumes at the ask side, i.e. we choose
\[v_{0,b}(x) = 0.0075(x-4)^2(x+4)^2 \1_{[-4,4]}(x), \quad v_{a,0}(x) = 0.0025 (x-4)^2(x+4)^2 \1_{[-4,4]}(x).\]
For these parameter values, Figure \ref{fig:SimFullLOBPrice} shows the evolution of the best bid and ask prices over time, while Figure \ref{fig:SimFullLOBAbVol} depicts the evolution of the absolute volume density functions of the visible book, i.e.
\[u_b(x) = v_b(-x+B)\1(x\leq B), \quad u_a(x) = v_a(x-A)\1(x\geq A),\]
over time.
The evolution of the prices is influenced by the spread as well as the imbalance factor. We start with a quite small spread and a large imbalance factor implying very small price volatilities, a slightly negative drift of the best bid price, and a positive drift of the best ask price. Hence, the prices move apart from each other. At $t \approx 0.4,$ we observe a price drop in the best bid price which heavily increases the spread between bid and ask. Therefore, both prices become more volatile, but imbalances are still significant. After a second price drop in the best bid price at $t \approx 0.7$, the huge spread now dominates the price evolution and the spread decreases. In the last quarter, we observe a similar price evolution of the best bid and ask price, which is caused by the fact that the imbalance factor and the spread stabilize around $0.5$ and $\ln(2)$, respectively. 

\begin{figure}[H]
    \centering
    \includegraphics[scale = 0.5]{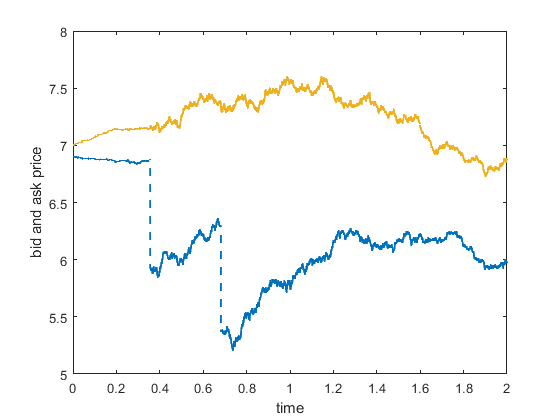}
    \caption{The evolution of the best bid price (blue) and the best ask price (yellow).}
    \label{fig:SimFullLOBPrice}
\end{figure}

In Figure \ref{fig:SimFullLOBAbVol} we observe that the cumulative volumes at the top of both sides of the limit order book converge. In this simulation study the sell side volume approaches the bid side volume because we have chosen the size of the order placements much greater than the size of average cancellations for the initial volume density functions. If placements would be of smaller size, the opposite effect could be observed, i.e. the buy side volumes at the top of the book would decrease to approach the sell side volumes at the top of the book. Moreover, the price drops in the bid price lead to a significant decrease of order volumes at the top of the bid side and hence to a decrease of the volume imbalance factor, which subsequently forces the spread to narrow again.

\begin{figure}[H]
    \centering
    \includegraphics[scale = 0.5]{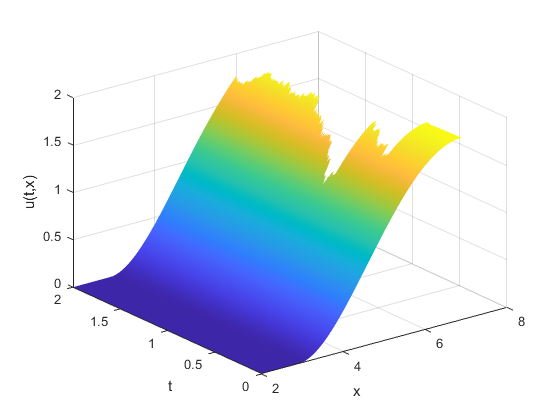}
    \includegraphics[scale = 0.5]{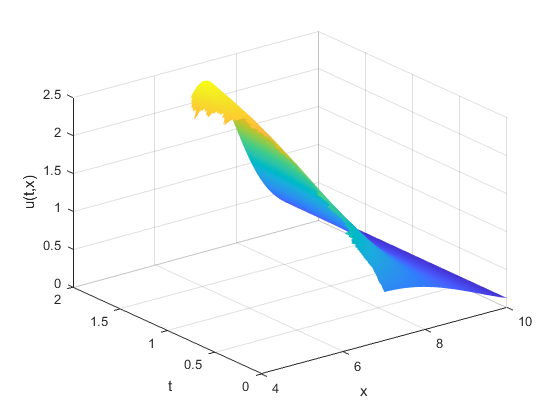}
    \caption{The evolution of the bid side (left) and ask side (right) volume density functions (in absolute coordinates; visible books only).}
    \label{fig:SimFullLOBAbVol}
\end{figure}

In a second simulation, starting from the same initial values, we allow jumps in all directions ($\lambda_b^+=\lambda_a^- = 0.15,$ $\lambda_b^-=\lambda_a^+ = 0.35$) and assume a rather strong dependence of the jump sizes on the external factor $Y^{(n)}$ by choosing 
$\eta_1 = 9$ and $\kappa = 10,$ i.e.~after $Y^{(n)}$ hits the threshold, the absolute value of the jump sizes increases by the factor $10$. We depict the corresponding bid and ask price evolution of two runs of our simulation in Figure \ref{fig:simExtInfluence}. In both runs $Y^{(n)}$ hits the threshold shortly after $t = 1$.

\begin{figure}[H]
    \centering
    \includegraphics[scale = 0.5]{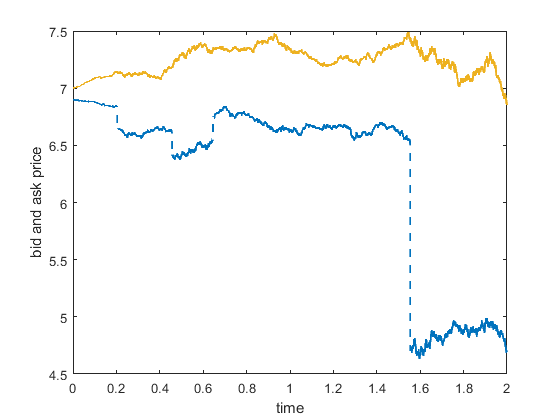} \quad
    \includegraphics[scale = 0.5]{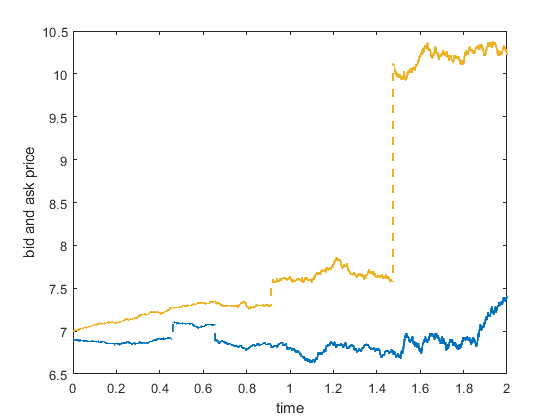}
    \caption{The evolution of the best bid price (blue) and the best ask price (yellow).}
    \label{fig:simExtInfluence}
\end{figure}

\section{Examples of large price jumps}\label{sec:jumpBe}

In this section we provide two examples of jump distributions that satisfy the rather technical Assumption \ref{ass:ExQ}. They are toy examples and not meant to mimic asset price jumps observed in real data, but rather to illustrate the range of jump distributions supported by our model.

\begin{Exp}\label{exp:normalJumps}

Let $M>0$ and suppose that $\delta_n=\xn$ and $(\xn)^{-1}\in\N$. Fix some continuous distribution function $F$ on $\mathcal{B}([-M,M])$ and two Lipschitz continuous, bounded functions $\mu,\sigma:E\rightarrow\R$ such that $\sigma(s)\geq 1$ for all $s\in E$. If $j\in\Z\backslash\{-1,0,1\}$ is such that there exists an $i\in\Z_M^{(n)}$ with
\[\sigma(s)x_i^{(n)} \leq x_j^{(n)}-\mu(s) < \sigma(s) x_i^{(n)} + \xn,\]
we set $$k_{a,j}^{(n)}(s)=F\left(\left[x_{i}^{(n)},x_{i+1}^{(n)}\right)\cap[-M,M]\right).$$ 
If none such $i\in \Z$ exists, we set $k_{a,j}^{(n)}(s)=0$. Note that all $k^{(n)}_{a,j}$s are well-defined as the intervals $[\sigma(s)x_i^{(n)},\sigma(s)x_i^{(n)} + \xn ),\ i\in\N,$ are non-overlapping due to $\sigma(s)\geq 1$. Then define for all $A\in\mathcal{B}(\R)$,
\[K_a(s,A):=F\left(\left\{\frac{x-\mu(s)}{\sigma(s)}:x\in A\right\}\cap[-M,M]\right),\qquad Q_a(A):=F(A\cap[-M,M])\]
and for all $x\in [-M,M]$,
\[\theta_a(s,x):=\sigma(s)x+\mu(s).\] 
Then Assumptions \ref{ass:ExQ} i), iv), and v) are satisfied by construction. Hence, we can apply Lemma \ref{lem:theta} and obtain for all $s\in E$ and $j\in Z$,
\begin{align*}
& K_a\Big(s,\theta_a\left(s,\left[x_{j}^{(n)},x_{j+1}^{(n)}\right)\right)\Big)-K^{(n)}_a\Big(s,\theta_a\left(s,\left[x_{j}^{(n)},x_{j+1}^{(n)}\right)\right)\Big)\\
&\qquad= F\Big(\left[x_{j}^{(n)},x_{j+1}^{(n)}\right) \cap[-M,M]\Big)-K_a^{(n)}\Big(s,\left[\sigma(s)x_j^{(n)}+\mu(s) ,\sigma(s)x_{j}^{(n)} + \mu(s) + \xn\right)\Big)=0,
\end{align*}
yielding the validity of Assumption \ref{ass:ExQ} iii). Moreover, we have for any $A\in\mathcal{B}([-M,M])$,
\[K_a\Big(s,\theta_a(s,A)\Big)=F(A)=Q_a(A).\]
Hence, Assumption \ref{ass:ExQ} ii) is also satisfied. Finally, note that
\begin{align*}
    &\sup_{s\in E}\sum_j\int_{[x_{j}^{(n)},x_{j+1}^{(n)})} \left|\1\Big(\theta_a^{(n)}(s,x_j^{(n)})=0\Big)-\1\Big(\theta_a(s,x)=0\Big)\right|Q_a(dx)\\
    &\qquad=\sup_{s\in E}\sum_j\1\Big(\sigma(s)x_j^{(n)}+\mu(s)\in \left(-\xn,0\right] \Big)Q_a\Big(\left[x_{j}^{(n)},x_{j+1}^{(n)}\right)\Big)
    \leq  \sup_{j} Q_a\Big(\left[x_{j}^{(n)},x_{j+1}^{(n)}\right)\Big)\longrightarrow 0.
\end{align*}
Therefore, Assumption \ref{ass:ExQ} vi) holds true as well. 
\end{Exp}

Example \ref{exp:normalJumps} shows that jump intensities can follow quite general distribution functions. While $\theta_a(s,\cdot)$ is linear for all $s\in E$ in the above example, we note that in general also non-linear transformations are possible. Moreover, observe that if $\mu(s)-M\sigma(s)> b-a$ for all $s=(b,v_b,a,v_a,t)\in E$, then the (negative) jumps of the ask price will never lead to a crossing of bid and ask prices. While in the above example the driving jump distribution in the pre-limit does not depend on $n$, the next example shows that this is generally possible. Indeed, even a slight dependence on $s\in E$ in the pre-limit is allowed as long as it vanishes for $n\rightarrow\infty$.

\begin{Exp}\label{exp:binJumps}

Suppose that $(\xn)^{-1}\in\N$ and $M\in\N$. Let $p_n:E\rightarrow(0,1)$ and $m_0,M_0:E\rightarrow\N$ satisfy 
\begin{align*}
\sup_{s\in E}|np_n(s)-\lambda|\rightarrow0,&\qquad M\geq M_0(s)\geq m_0(s)\geq1\quad\forall\ s\in E,\\ 
|m_0(s)-m_0(\tilde{s})|\leq L\|s-\tilde{s}\|_E,&\qquad|M_0(s)-M_0(\tilde{s})|\leq L\|s-\tilde{s}\|_E\quad\forall\ s,\tilde{s}\in E.
\end{align*} 
For $s\in E$ let
\[K^{(n)}_a(s,dx):=\sum_{k=m_0(s)}^{M_0(s)}{n\choose k+M-M_0(s)}(p_n(s))^{k+M-M_0(s)}(1-p_n(s))^{n-k-M+M_0(s)}\varepsilon_k(dx).\]
Then,
\[K_a(s,dx):=\sum_{k=m_0(s)}^{M_0(s)} e^{-\lambda}\frac{\lambda^{k+M-M_0(s)}}{(k+M-M_0(s))!}\varepsilon_{k}(dx),\qquad Q_a(dx):=\sum_{k=1}^M e^{-\lambda}\frac{\lambda^k}{k!}\varepsilon_k(dx)\]
and
\[\theta_a(s,x):=\begin{cases}x-M+M_0(s)&: x\in [M-M_0(s)+m_0(s),M]\\m_0(s)\cdot(x-M+M_0(s)-k_0(s)+1) &:x\in(M-M_0(s)+m_0(s)-1,M-M_0(s)+k_0(s))\\0&:x\in[-M, M-M_0(s)+m_0(s)-1]\end{cases}\]
satisfy for any $A\in\mathcal{B}([-M,M])$,
\[K_a(s,\theta_a(s,A))=K_a(s,\{x-M+M_0(s):x\in A\})=Q_a(A\cap[M-M_0(s)+m_0(s),M])\]
and 
\[Q_a(\{x\in A: \theta_a(s,x)=0\})=Q_a(A\cap[1,M-M_0(s)+m_0(s)-1])=Q_a(A)-Q_a(A\cap[M-M_0(s)+m_0(s),M]).\]
Hence, Assumption \ref{ass:ExQ} ii) is satisfied. By construction also Assumptions \ref{ass:ExQ} i) and \ref{ass:ExQ} v) hold. Moreover,
\begin{align*}
\sup_{s\in E}&\sum_j\Big| K_a\Big(s,\theta_a\left(s, \left[x_{j}^{(n)},x_{j+1}^{(n)}\right)\right)\Big)-K_a^{(n)}\Big(s,\theta_a\left(s,\left[x_{j}^{(n)},x_{j+1}^{(n)}\right)\right]\Big)\Big|\\
= &\sup_{s\in E}\sum_{k=m_0(s)}^{M_0(s)}\Big| K_a\Big(s,\{k\}\Big)-K^{(n)}_a\Big(s,\{k\}\Big)\Big|\\
=&\sup_{s\in E}\sum_{k=m_0(s)+M-M_0(s)}^M\left|e^{-\lambda}\frac{\lambda^{k}}{k!}- {n\choose k}(p_n(s))^{k}(1-p_n(s))^{n-k}\right|\leq \sup_{s\in E}2n(p_n(s))^2\longrightarrow 0,
\end{align*}
i.e.~Assumption \ref{ass:ExQ} iii) is satisfied. If $s,\tilde{s}\in E$ satisfy $M_0(s)-m_0(s)=M_0(\tilde{s})-m_0(\tilde{s})$, then 
\[|\theta_a(s,x)-\theta_a(\tilde{s},x)|\leq  |M_0(s)-M_0(\tilde{s})|\leq L\|s-\tilde{s}\|_E\quad\forall\ x\in[-M,M].\]
If $s,\tilde{s}\in E$ satisfy $M_0(s)-m_0(s)> M_0(\tilde{s})-m_0(\tilde{s})$, then we have for all $x\in[-M,M]$ the estimate
\begin{align*}
    |\theta_a(s,x)-\theta_a(\tilde{s},x)|&\leq |M_0(s)-M_0(\tilde{s})|+M\varepsilon_{\big(M-M_0(s)+m_0(s)-1,M-M_0(\tilde{s})+m_0(\tilde{s})\big)}(\{x\})\\
    &\leq |M_0(s)-M_0(\tilde{s})|+M|m_0(\tilde{s})-M_0(\tilde{s})-m_0(s)+M_0(s)+1|\\
    &\leq  (1+2M)|M_0(s)-M_0(\tilde{s})|+2M|m_0(\tilde{s})-m_0(s)|\\
    &\leq (1+4M)L\|s-\tilde{s}\|_E.
\end{align*}
Hence, Assumption \ref{ass:ExQ} iv) holds. 
Finally, we note that for all $x_j^{(n)}\in\N$ we have $\theta^{(n)}_a(s,x_j^{(n)})=\theta_a(s,x_j^{(n)})$. As $Q_a$ only charges $\N$, this shows that Assumption \ref{ass:ExQ} vi) is also satisfied.
\end{Exp}
Example \ref{exp:binJumps} illustrates very well the restrictions imposed on the limiting jump distribution through Assumption \ref{ass:ExQ}: while the range of jump sizes (parameterized through $m_0$ and $M_0$) can differ across states, the $\lambda$ determining the jump intensities has to be constant and cannot depend on the state $s\in E$. This is necessary to obtain jump diffusion dynamics - rather than more general (and even more complicated) semimartingale dynamics - in the high-frequency limit.

\section{Sketch of the proof of the main result (Theorem \ref{res:mainTheorem})}\label{sec:sketch}

In this section, we present a step-by-step proof sketch of our main result (Theorem \ref{res:mainTheorem}). The technical details can be found in Section \ref{sec:proofs}.

\subsubsection*{Step 1: State and time separation}

By making use of the time change theorem, we can simplify our subsequent analysis to equidistant,  deterministic order arrival times, where the time intervals between two consecutive order arrivals are of length $\tn.$ To this end, we set $\t_k := k\tn$ for $k \in\N_0$, $T_n:=\lfloor T/\tn\rfloor$, and define the \textit{state process} $\eta^{(n)}$ via
$$\eta^{(n)}(t) := \sum_{k=0}^{T_n}S^{(n)}_k\1_{\left[\t_k, \t_{k+1}\right)}(t),\quad t\in[0,T].$$  
Introducing the process 
$$\tau^{\eta,(n)}(u) := \sum_{k=0}^{T_n}\tau^{(n)}_k\1_{\left[\t_k, \t_{k+1}\right)}(u),\quad u\in[0,T],$$ 
we define the \textit{time process} $\zeta^{(n)}$ via
\[\zeta^{(n)}(t) := \inf\left\{u > 0: \tau^{\eta,(n)}(u) > t\right\}\wedge (T_n+1)\tn,\quad t\in[0,T].\]

 The advantage of the state and time separation is that we can focus on first analyzing the convergence of the state process $\eta^{(n)}$, for which we will prove the following convergence result (cf.~Steps 2-5 below).

\begin{Prop}\label{res:mainTheorem-detTimes}
   Let the assumptions of Theorem \ref{res:mainTheorem} be satisfied. Then, $\eta^{(n)}$ converges weakly in the Skorokhod topology to $\eta = (B^{\eta}, v_b^{\eta}, A^{\eta}, v^{\eta}_a, \tau^{\eta})$ being the unique strong solution of the coupled diffusion-fluid system  \eqref{eq:eta}.
\end{Prop}

Let us now define the composition of the state process $\eta^{(n)}$ with the time process $\zeta^{(n)}$ as
\[S^{(n),*}(t) := \eta^{(n)}\left(\zeta^{(n)}(t) - \tn\right),\quad t\in[0,T].\]

Relying on a time change argument for processes with discontinuities,  
our main result readily follows from statements ii) and iii) of the following corollary. 

\begin{Cor}\label{Cor:FromDetToRanTimes}
Let Assumptions \ref{ass:IV}--\ref{ass:Scaling} be satisfied. Then,
    \begin{enumerate}
        \item[i)] $\zeta^{(n)} \Rightarrow \zeta$ in the Skorokhod topology, where $\zeta^{-1}(t) = \tau^\eta(t)= \int^t_0 \varphi(\eta(u))du,$
        \item[ii)] $S^{(n),*} \Rightarrow \eta\circ \zeta =: S$ in the Skorokhod topology, and
        \item[iii)] as $n\rightarrow\infty$,
        $$\Pro\left[\sup_{t\in[0,T]}\big\|S^{(n),*}(t)-S^{(n)}(t)\big\|_E>0\right]\rightarrow0.$$
    \end{enumerate}
\end{Cor}

In the subsequent steps we present a sketch of the proof of Proposition \ref{res:mainTheorem-detTimes}. To this end, let us define $\eta^{(n)} := \left(B^{\eta,(n)}, v^{\eta,(n)}_b, A^{\eta, (n)}, v^{\eta,(n)}_a, \tau^{\eta,(n)}\right),$ where for $t\in \left[\t_k, \t_{k+1}\right)\cap[0,T],$
\begin{align}\label{def:detModelN}
\begin{aligned}
&B^{\eta,(n)}(t) = B^{(n)}_k, &\quad v^{\eta,(n)}_b(t,\cdot) = v^{(n)}_{b,k},\\
&A^{\eta,(n)}(t) = A^{(n)}_k, &\quad v^{\eta,(n)}_a(t,\cdot) = v^{(n)}_{a,k}, &\quad \tau^{\eta,(n)}(t) = \tau^{(n)}_k.
\end{aligned}
\end{align}

\subsubsection*{Step 2: Representation of the LOB-dynamics as a stochastic difference equation and convergence of its integrators}

To prove Proposition \ref{res:mainTheorem-detTimes}, we will apply results of Kurtz and Protter \cite{KurtzProtter2} about the convergence of stochastic differential equations in infinite dimension. To this end, we rewrite the discrete time dynamics of $\eta^{(n)}$ in the form of a proper stochastic difference equation, whose driving processes converge to limit processes which are independent of the order book sequence. First, we decompose
\[B^{\eta,(n)}(t) = B^{(n)}_0 + B^{\eta,s,(n)}(t) + B^{\eta,\ell,(n)}(t), \quad A^{\eta,(n)}(t) = A^{(n)}_0 + A^{\eta,s,(n)}(t) + A^{\eta, \ell,(n)}(t),\]
where
$B^{\eta,s,(n)}, A^{\eta, s,(n)}$ and $B^{\eta,\ell,(n)}, A^{\eta,\ell,(n)}$ describe the price dynamics of the small and large price changes, respectively, i.e.
\begin{align*}
    B^{\eta,s,(n)}(t) &:= \sum_{k=1}^{\ttn} \delta B^{(n)}_k \1\left(\left|\delta B^{(n)}_k\right| \leq \delta_n\right), &\, A^{\eta,s,(n)}(t) &:= \sum_{k=1}^{\ttn} \delta A^{(n)}_k\1\left(\left|\delta A^{(n)}_k\right| \leq \delta_n\right),\\
    B^{\eta,\ell,(n)}(t) &:= \sum_{k=1}^{\ttn}\delta B^{(n)}_k \1\left(\left|\delta B^{(n)}_k\right| > \delta_n\right), &\, A^{\eta,\ell,(n)}(t) &:= \sum_{k=1}^{\ttn}\delta A^{(n)}_k \1\left(\left|\delta A^{(n)}_k\right| > \delta_n\right).
\end{align*}
Next, observe that we can write the price dynamics of the small price changes $B^{\eta,s,(n)}$ and $A^{\eta,s,(n)}$ as
\begin{align}\label{eq:rewSP}
\begin{aligned}
    B^{\eta,s,(n)}(t) &:= \sum_{k=1}^{\ttn}\left\{p^{(n)}_b(S^{(n)}_{k-1})\tn + r^{(n)}_b(S^{(n)}_{k-1}) \delta Z^{(n)}_{b,k}\right\},\\
    A^{\eta,s,(n)}(t) &:= \sum_{k=1}^{\ttn}\left\{p^{(n)}_a(S^{(n)}_{k-1})\tn + r^{(n)}_a(S^{(n)}_{k-1}) \delta Z^{(n)}_{a,k}\right\},
\end{aligned}
\end{align}
where $Z^{(n)}_{I,k} := \sum_{j=1}^k\delta Z^{(n)}_{I,j}$ for $I = b,a$ and
\begin{align*}
    \delta Z^{(n)}_{b,j} &:= \frac{\delta B^{(n)}_j \1(0 < |\delta B^{(n)}_{j}| \leq \delta_n) - \tn p^{(n)}_b(S^{(n)}_{j-1})}{r^{(n)}_b(S^{(n)}_{j-1})},\\
    \delta Z^{(n)}_{a,j} &:= \frac{\delta A^{(n)}_j \1(0 < |\delta A^{(n)}_j| \leq \delta_n) - \tn p^{(n)}_a(S^{(n)}_{j-1})}{r^{(n)}_{a}(S^{(n)}_{j-1})}.
\end{align*}
Then, Proposition \ref{res:limitSJ} states that the processes 
\begin{equation}\label{def:intZ}
    Z^{(n)}_b(t) := \sum^{T_n}_{k = 1} Z^{(n)}_{b,k} \1_{\left[\t_k, \t_{k+1}\right)}(t), \quad Z^{(n)}_a(t) := \sum^{T_n}_{k=1} Z^{(n)}_{a,k} \1_{\left[\t_k, \t_{k+1}\right)}(t),\quad t\in[0,T],
\end{equation} 
converge to two independent Brownian motions and therefore the integrators in \eqref{eq:rewSP} converge to processes that are independent of the order book dynamics.

\begin{Prop}\label{res:limitSJ}
Let Assumptions \ref{ass:Prob} and \ref{ass:Scaling} be satisfied. Then, as $ n \rightarrow \infty,$ the process $(Z^{(n)}_{b}, Z^{(n)}_a)$ converges weakly in the Skorokhod topology to a standard planar Brownian motion $(Z_b, Z_a).$ In particular, $Z_b$ and $Z_a$ are independent.
\end{Prop}

 Let us now turn to the processes $B^{\eta,\ell,(n)}, A^{\eta,\ell,(n)}$ corresponding to the price dynamics of the large price changes. First, note that their joint jump measure is given by
 \[\mu^{\eta,(n)}([0,t], dx, dy):=\sum_{k=1}^{\ttn} \1\left(\left|\delta B^{(n)}_k\right| > \delta_n\right) \varepsilon_{(\delta B^{(n)}_k,0)}(dx,dy)+ \1\left(\left|\delta A^{(n)}_k\right| > \delta_n\right) \varepsilon_{(0,\delta A^{(n)}_k)}(dx,dy)\]
 for all $t\in[0,T]$, since $A$ and $C$ events do not occur simultaneously. Setting for all $t\in[0,T]$,
 \[\mu_b^{\eta,(n)}([0,t],dx):=\mu^{\eta,(n)}([0,t] , dx,\{0\}),\quad \mu_a^{\eta,(n)}([0,t],dy):=\mu^{\eta,(n)}([0,t] ,\{0\},dy),\]
we can now rewrite $B^{\eta,\ell,(n)}$ and $A^{\eta,\ell,(n)}$ as
\[B^{\eta,\ell,(n)}(t)= \int_{\R} x \ \mu_b^{\eta,(n)}([0,t] , dx) \quad \text{ and } \quad A^{\eta,\ell,(n)}(t)= \int_{\R} y\ \mu_a^{\eta,(n)}([0,t] , dy)\]
for all $t\in[0,T]$. Then the compensator $\nu^{\eta,(n)}$ of $\mu^{\eta,(n)}$ is given by
\begin{align*}
\nu^{\eta,(n)}([0,t],dx, dy) &:= \sum_{k=1}^{\ttn} \tn K_b^{(n)}\left(S^{(n)}_{k-1}, dx\right)\varepsilon_0(dy)+ \tn K_a^{(n)}\left(S^{(n)}_{k-1}, dy\right)\varepsilon_0(dx)
\end{align*}
for $t\in [0,T]$. Equation \eqref{eq:Q} in Assumption \ref{ass:ExQ} will be our starting point  to construct a sequence of discrete-time integrators, which converges to a limit that is independent of the order book dynamics. Given the finite measures $Q_b, Q_a$ introduced in Assumption \ref{ass:ExQ}, for technical convenience we will assume for all $n\in\N$ the existence of two independent, homogeneous Poisson random measures $\mu^{Q,(n)}_I,\ I=b,a,$ on $\mathcal{B}^{(n)}$ with intensity measures $Q_I,\ I=b,a$, respectively.

For all $n \in \N$ we now consider the random jump measure
$$\mu^{J^{(n)}}\left(dt, dx, dy\right) :=\mu_b^{J^{(n)}}\left(dt, dx\right)\varepsilon_0(dy)+\mu_a^{J^{(n)}}\left(dt, dy\right)\varepsilon_0(dx),$$
where for $I=b,a$, $t\in[0,T]$, and $A\in\mathcal{B}([-M,M])$ we define 
\begin{align}\label{def:measJn}
\begin{aligned}
\mu_I^{J^{(n)}}&\left([0,t], A\right) :=\sum_{k=1}^{\ttn}\sum_{j:\ x_j^{(n)}\in A}\mu^{\eta,(n)}_I\Big(\Big[\t_k, \t_{k+1}\Big), \left\{\theta^{(n)}_I\big(S^{(n)}_{k-1},x_j^{(n)}\big)\right\}\Big)\\
&\quad+\sum_{k=1}^{\ttn}\sum_{j:\ x_j^{(n)}\in A}
\mu^{Q,(n)}_I\left(\Big[\t_k, \t_{k+1}\Big),  \Big[x_{j}^{(n)},x_{j+1}^{(n)}\Big)\right)\1\Big(\theta^{(n)}_I\big(S_{k-1}^{(n)},x_j^{(n)}\big)=0\Big).
\end{aligned}
\end{align} 
Then the compensator of $\mu^{J^{(n)}}$ is given by
$$\nu^{J^{(n)}}\left(dt, dx, dy\right) :=\nu_b^{J^{(n)}}\left(dt, dx\right)\varepsilon_0(dy)+\nu_a^{J^{(n)}}\left(dt, dy\right)\varepsilon_0(dx),$$
with 
\begin{align}\label{def:compJn}
\begin{aligned}
 \nu_I^{J^{(n)}}([0,t], A)&:= \sum_{k=1}^{\ttn}\sum_{j:\ x_j^{(n)}\in A}\nu^{\eta,(n)}_I\Big(\left[\t_k, \t_{k+1}\right), \left\{\theta^{(n)}_I\Big(S^{(n)}_{k-1},x_j^{(n)}\Big)\right\}\Big)\\
&\qquad+\sum_{k=1}^{\ttn}\tn
\sum_j \1\Big(\theta^{(n)}_I(S^{(n)}_{k-1},x_j^{(n)})=0\Big)Q_I\Big(\left[x_{j}^{(n)},x_{j+1}^{(n)}\right)\Big)
\end{aligned}
\end{align}
being the compensator of $\eta^{J^{(n)}}_I$ for $I=b,a$ and $(t,A)\in[0,T]\times\mathcal{B}([-M,M])$. Now, as stated in Lemma \ref{lem:comLJ} below, we found a representation of $B^{\eta,\ell,(n)},$ $A^{\eta,\ell,(n)}$ in terms of the coefficient functions $\theta^{(n)}_b,$ $\theta^{(n)}_a$ introduced in Assumption \ref{ass:ExQ} and the random jump measures $\mu^{J^{(n)}}_b,$ $\mu^{J^{(n)}}_a$ introduced in \eqref{def:measJn}.

\begin{Lem}\label{lem:comLJ}
Let Assumption \ref{ass:Prob} and \ref{ass:ExQ} hold. Then, 
\begin{align}\label{eq:decLJ}
\begin{aligned}
B^{\eta,\ell,(n)}(t) &= \int_0^t \int_{[-M,M]} \theta^{(n)}_b(\eta^{(n)}(u-), y) \, \mu^{J^{(n)}}_b(du, dy)\quad \text{a.s.}\\
A^{\eta,\ell,(n)}(t) &= \int_0^t \int_{[-M,M]} \theta^{(n)}_a(\eta^{(n)}(u-),y )\, \mu^{J^{(n)}}_a(du, dy)\quad \text{a.s.}
\end{aligned}
\end{align}
\end{Lem}

 For all $n \in \N$, we construct a stochastic process $X^{(n)}$, indexed by $[0,T] \times C_b([-M,M]^2)$, in the following way: for all $t\in [0,T]$, $g \in C_b([-M,M]^2)$, $I = b,a,$ we set
\begin{equation}\label{def:Xn}
X^{(n)}(t, g) := \int_{[-M,M]^2} g(x,y)\, \mu^{J^{(n)}}([0,t], dx, dy).
\end{equation}
The following proposition proves the convergence of the $(X^{(n)})_{n\in\N}$ and thereby shows that the sequences $(\mu^{J^{(n)}}_I)_{n\in\N},$ $I=b,a$, converge to two independent, homogeneous Poisson random measures.

\newpage
 
\begin{Prop}\label{res:limitLJ}
	Assume that Assumptions \ref{ass:Prob} and \ref{ass:ExQ} are satisfied. Then for any $m\in\N$ and any $g_1, \cdots, g_m\in C_b([-M,M]^2)$ it holds that
	\begin{equation}\label{eq:convXn}
	\left(X^{(n)}(\cdot, g_1), \cdots, X^{(n)}(\cdot, g_m)\right) \Rightarrow \left(X(\cdot, g_1), \cdots, X(\cdot, g_m)\right)
	\end{equation}
	in $\mathcal{D}(\R^{m}; [0,T])$ with 
	\[X(t, g) := \int_{[-M,M]^2} g(x,0)\, \mu^{Q}_b([0,t], dy)+ g(0,y)\, \mu^{Q}_a([0,t], dy), \quad g \in C_b([-M,M]^2),\ t\in [0,T],\]
	where  $\mu^{Q}_I,\ I=b,a,$ are independent, homogeneous Poisson random measures with intensity measures given by $\lambda\otimes Q_I,\ I=b,a$, respectively.
\end{Prop}

Now, let us turn to the time and volume dynamics $\tau^{\eta,(n)}$, $v^{\eta,(n)}_b$, and $v^{\eta,(n)}_a.$ First note, that they can be rewritten for $k=0,\dots,T_n$ and $x\in\R$ as follows:
\begin{align*}
    \tau^{(n)}_k &= \tn \sum_{j=1}^k \varphi^{(n)}(S^{(n)}_{j-1}) + R^{(n)}_{\varphi,k},\\
    v^{(n)}_{b,k}(x) &= v^{(n)}_{b,0}\left(x-\left(B^{(n)}_{k} - B^{(n)}_0\right)\right) + \vn \sum_{j=1}^k f^{(n)}_b\left[S^{(n)}_{j-1}\right] \left(x - \left(B^{(n)}_k - B^{(n)}_{j-1}\right)\right) + R^{(n)}_{b,k}(x),\\
    v^{(n)}_{a,k}(x) &= v^{(n)}_{a,0}\left(x+ A^{(n)}_{k} - A^{(n)}_0\right) + \vn \sum_{j=1}^k f^{(n)}_a\left[S^{(n)}_{j-1}\right]\left(x + A^{(n)}_k - A^{(n)}_{j-1}\right) + R^{(n)}_{a,k}(x),
\end{align*}
where
\begin{align}\label{def:remainder}
\begin{aligned}
 R^{(n)}_{b,k}(x) &:= \vn \sum_{j=1}^{k} \left(M^{(n)}_{b,j}-f^{(n)}_b[S^{(n)}_{j-1}]\right)\left(x-\left(B^{(n)}_k - B^{(n)}_{j-1}\right)\right),\\
 R^{(n)}_{a,k}(x) &:= \vn \sum_{j=1}^{k} \left(M^{(n)}_{a,j}-f^{(n)}_a[S^{(n)}_{j-1}]\right)\left(x+A^{(n)}_k - A^{(n)}_{j-1}\right),
\end{aligned}
\end{align}
and 
\begin{align}\label{def:deltaVarphi}
    R^{(n)}_{\varphi,k} := \tn \sum_{j=1}^k \left(\varphi^{(n)}_j - \varphi^{(n)}(S^{(n)}_{j-1})\right).
\end{align}
Let us denote $R^{(n)}_I(t) := R^{(n)}_{I,k},$ $I = b,a,\varphi$, if $\t_k \leq t < \t_{k+1}$. According to the next proposition, the random fluctuations of the time and volume dynamics vanish in the high-frequency limit.

\begin{Prop}\label{res:conRemainder}
Under Assumptions \ref{ass:randTimes}, \ref{ass:probVol}, and \ref{ass:Scaling}, we have
\[\E\left[\sup_{k\leq T_n} \left|R^{(n)}_{\varphi,k}\right|^2\right]\rightarrow 0 \quad \text{ and } \quad \E \left[\sup_{k\leq T_n} \left\|R^{(n)}_{I,k}\right\|^2_{L^2}\right] \rightarrow 0 \quad  \text{ for } I = b,a.\]
In particular, the $L^2(\R)$--valued processes $R^{(n)}_I=(R^{(n)}_I(t))_{t\in[0,T]},\ I=b,a,$ and the $[0,T]$--valued process $R^{(n)}_\varphi=(R^{(n)}_\varphi(t))_{t\in[0,T]}$ converge weakly in the Skorokhod topology to the zero process.
\end{Prop}

\newpage

Altogether, we can write the microscopic LOB dynamics as a stochastic difference equation, i.e.~for all $x\in\R$ and $k=0,\dots,T_n$,
\begin{align*}
    B^{(n)}_k &= B^{(n)}_0 + \sum_{j=1}^{k} \Bigg(p^{(n)}_b(S^{(n)}_{j-1}) \tn + r^{(n)}_b(S^{(n)}_{j-1}) \delta Z^{(n)}_{b,j}\\
    &\hspace{3cm} + \int_{-M}^M \theta^{(n)}_b(S^{(n)}_{j-1},y) \mu^{J^{(n)}}_b\left(\left[\t_{j}, \t_{j+1}\right), dy\right)\Bigg),\\
    v^{(n)}_{b,k}(x) &= v^{(n)}_{b,0}\left(x-\left(B^{(n)}_{k} - B^{(n)}_0\right)\right) + \vn \sum_{j=1}^k f^{(n)}_b\left[S^{(n)}_{j-1}\right] \left(x - \left(B^{(n)}_k - B^{(n)}_{j-1}\right)\right) + R^{(n)}_{b,k}(x),\\
    A^{(n)}_k &= A^{(n)}_0 + \sum_{j=1}^{k} \Bigg(p^{(n)}_a(S^{(n)}_{j-1}) \tn + r^{(n)}_a(S^{(n)}_{j-1}) \delta Z^{(n)}_{a,j}\\
    &\hspace{3cm} + \int_{-M}^M \theta^{(n)}_a(S^{(n)}_{j-1},y) \mu^{J^{(n)}}_a\left(\left[\t_{j}, \t_{j+1}\right), dy\right)\Bigg),\\
    v^{(n)}_{a,k}(x) &= v^{(n)}_{a,0}\left(x+ A^{(n)}_{k} - A^{(n)}_0\right) + \vn \sum_{j=1}^k f^{(n)}_a\left[S^{(n)}_{j-1}\right] \left(x + A^{(n)}_k - A^{(n)}_{j-1}\right) + R^{(n)}_{a,k}(x),\\
\tau^{(n)}_k &= \tn \sum_{j=1}^k \varphi^{(n)}(S^{(n)}_{j-1}) + R^{(n)}_{\varphi,k},
\end{align*}
and we have shown in Proposition \ref{res:limitSJ}, Proposition \ref{res:limitLJ}, and Proposition \ref{res:conRemainder} that its integrators converge to limit processes that do not depend on the order book dynamics.

\subsubsection*{Step 3: A first approximation of the microscopic state process $\eta^{(n)}$}

Proposition \ref{res:conRemainder} suggests, that the fluctuations of the time and volume dynamics vanish in the high-frequency limit. Based on this observation, we introduce a new sequence $\tilde{S}^{(n)}$ of order book models in which the random innovations $M^{(n)}_{I,k},\ k\in\N$, are replaced by the approximations $f_I[\tilde{S}^{(n)}_{k-1}],\ k\in\N$, and $v^{(n)}_{I,0}$ is replaced by its limit $v_{I,0}$, for both $I = b,a$. Moreover, each $\varphi^{(n)}_k$ is replaced by the approximation $\varphi(\tilde{S}^{(n)}_{k-1})$ and we replace $\vn$ by $\tn$ as they are of the same order by Assumption \ref{ass:Scaling}. Therefore, we define $\tilde{S}^{(n)}_k = (\tilde{B}^{(n)}_k, \tilde{v}^{(n)}_{b,k}, \tilde{A}^{(n)}_k, \tilde{v}^{(n)}_{a,k}, \tilde{\tau}^{(n)}_k)$ for $k=,0\dots,T_n$ as
\begin{align*}
  \tilde{B}^{(n)}_k &= B_0^{(n)} + \sum_{j=1}^k \Bigg(p^{(n)}_b\left(\tilde{S}^{(n)}_{j-1}\right) \tn + r^{(n)}_b\left(\tilde{S}_{j-1}^{(n)}\right)\delta Z^{(n)}_{b,j}\\
  &\hspace{5.5cm} + \int_{-M}^M\theta^{(n)}_b\left(\tilde{S}^{(n)}_{j-1}, y\right)\mu^{J^{(n)}}_b\left(\left[\t_{j}, \t_{j+1}\right), dy\right)\Bigg),\\
   \tilde{v}^{(n)}_{b,k}(x) &= v_{b,0}\left(x - \left(\tilde{B}^{(n)}_k - B^{(n)}_0\right)\right) + \sum_{j=1}^k f_b\left[\tilde{S}^{(n)}_{j-1}\right]\left(x - \left(\tilde{B}^{(n)}_k - \tilde{B}^{(n)}_{j-1}\right)\right) \tn,
   \end{align*}
   \begin{align*}
   \tilde{A}^{(n)}_k &= A_0^{(n)} + \sum_{j=1}^k \Bigg(p^{(n)}_a\left(\tilde{S}^{(n)}_{j-1}\right) \tn + r^{(n)}_a\left(\tilde{S}_{j-1}^{(n)}\right)\delta Z^{(n)}_{a,j}\\
   &\hspace{5.5cm}  + \int_{-M}^M\theta^{(n)}_a\left(\tilde{S}^{(n)}_{j-1}, y\right)\mu^{J^{(n)}}_a\left(\left[\t_{j}, \t_{j+1}\right), dy\right)\Bigg),\\
  \tilde{v}^{(n)}_{a,k}(x) &= v_{a,0}\left(x + \tilde{A}^{(n)}_k - A^{(n)}_0\right) + \sum_{j=1}^k f_a\left[\tilde{S}^{(n)}_{j-1}\right]\left(x+ \tilde{A}^{(n)}_k - \tilde{A}^{(n)}_{j-1}\right) \tn,\\
   \tilde{\tau}^{(n)}_k &= \sum_{j=1}^k \varphi\left(\tilde{S}^{(n)}_{j-1}\right) \tn.
  \end{align*}
For all $n\in \N,$ we denote $\tilde{\eta}^{(n)}(t) :=\tilde{S}^{(n)}_k,$ if $\t_k \leq t < \t_{k+1}.$ The next proposition states that the interpolated state process $\eta^{(n)}$ is approximately equal to $\tilde{\eta}^{(n)}$ as $n \rightarrow \infty$. 

\begin{Prop}\label{res:convApprox}
If Assumptions \ref{ass:IV}--\ref{ass:Scaling} are satisfied, then
\[ \E \left[\sup_{k\leq T_n} \left\|S^{(n)}_k - \tilde{S}^{(n)}_k\right\|_{E}^2\right] \rightarrow 0 \quad \text{ as } n \rightarrow \infty.\]
In particular, the process $\eta^{(n)}- \tilde{\eta}^{(n)}$ converges weakly in the Skorokhod topology to the zero process.
\end{Prop}

\subsubsection*{Step 4: Limit theorem for the state process with respect to absolute volumes}

Note that the dynamics of $\tilde{v}^{\eta,(n)}_{b}$ and $\tilde{v}^{\eta,(n)}_a$ defined above are not given in standard semimartingale form due to the time dependent shift in the $x$-variable. This prevents us from directly applying convergence results for infinite dimensional semimartingales. To overcome this issue, we first proof an intermediate convergence result for the discrete-time order book sequence with respect to absolute volumes. To this end, we define the approximated \textit{absolute} volume dynamics by
\begin{equation}\label{def:absVol}
    \tilde{u}^{(n)}_{b,k}(x) := \tilde{v}^{(n)}_{b,k}\left(-x + \tilde{B}^{(n)}_k\right), \quad \tilde{u}^{(n)}_{a,k}(x) := \tilde{v}^{(n)}_{a,k}\left(x - \tilde{A}^{(n)}_k\right), \quad k = 0,\cdots,T_n.
\end{equation}
Then we introduce $\tilde{S}^{(n),abs}_k := (\tilde{B}^{(n)}_k, \tilde{u}^{(n)}_{b,k}, \tilde{A}^{(n)}_k, \tilde{u}^{(n)}_{a,k}, \tilde{\tau}^{(n)}_k),\ k=0,\dots,T_n$, and note that a priori the coefficients in the dynamics of $\tilde{S}^{(n),abs}$ are still functions of  $\tilde{S}^{(n)}$ and hence in particular of the approximated \textit{relative} volume dynamics $\tilde{v}^{(n)}_{b}$ and $\tilde{v}^{(n)}_{a}$. For this reason, we introduce a shift operator $\psi: E \rightarrow E$ such that for all $s = (b,v,a,w,t) \in E,$
\[\psi(s):= \left(b, v(- (\cdot - b)), a, w(\cdot - a), t\right).\]
Then $\psi(\tilde{S}^{(n),abs}_k)=\tilde{S}^{(n)}_k$ for all $k=0,\dots,T_n$ and we can rewrite the dynamics of $\tilde{S}^{(n),abs}$ in such a way that all coefficient functions directly depend on $\tilde{S}^{(n),abs}$:
\begin{align*}
    \tilde{B}^{(n)}_k &= B^{(n)}_0 + \sum_{j=1}^k\Bigg(p^{(n)}_b\left(\psi(\tilde{S}^{(n),abs}_{j-1})\right) \tn + r^{(n)}_b\left(\psi(\tilde{S}^{(n),abs}_{j-1})\right) \delta Z^{(n)}_{b,j}\\
    &\hspace{3cm} + \int_{[-M,M]} \theta^{(n)}_b\left(\psi(\tilde{S}^{(n),abs}_{j-1}),y\right) \mu^{J^{(n)}}_b\left(\left[\t_j, \t_{j+1}\right) \times dy\right)\Bigg),\\
    \tilde{u}^{(n)}_{b,k}(x) &= v_0(-x + B^{(n)}_0) + \sum_{j=1}^{k}f_b\left[\psi(\tilde{S}^{(n),abs}_{j-1})\right](-x+ \tilde{B}^{(n)}_{j-1}) \tn,
    \end{align*}
    \begin{align*}
    \tilde{A}^{(n)}_k &= A^{(n)}_0 + \sum_{j=1}^k\Bigg(p^{(n)}_a\left(\psi(\tilde{S}^{(n),abs}_{j-1})\right) \tn + r^{(n)}_a\left(\psi(\tilde{S}^{(n),abs}_{j-1})\right) \delta Z^{(n)}_{a,j}\\
    &\hspace{3cm} + \int_{[-M,M]} \theta^{(n)}_a\left(\psi(\tilde{S}^{(n),abs}_{j-1}),y\right) \mu^{J^{(n)}}_a\left(\left[\t_j, \t_{j+1}\right) \times dy\right)\Bigg),\\
    \tilde{u}^{(n)}_{a,k}(x) &= v_0(x - A^{(n)}_0) + \sum_{j=1}^{k}f_a\left[\psi(\tilde{S}^{(n),abs}_{j-1})\right](x- \tilde{A}^{(n)}_{j-1}) \tn,\\
    \tilde{\tau}^{(n)}_k &= \sum_{j=1}^k \varphi\left(\psi(\tilde{S}^{(n),abs}_{j-1})\right).
    \end{align*}

For $t\in[0,T]$, we define its piecewise constant interpolation as $\tilde{\eta}^{(n),abs}(t) := \tilde{S}^{(n),abs}_k,$ if $\t_k \leq t < \t_{k+1}$. Now the following theorem shows the
convergence of $\tilde{\eta}^{(n),abs}.$ Its proof is an application of convergence results for infinite dimensional SDEs in Kurtz and Protter \cite{KurtzProtter2}.

\begin{The}\label{res:conAV}
Let Assumptions \ref{ass:IV}-\ref{ass:Scaling} be satisfied. Then the interpolation of the approximated LOB dynamics with respect to the absolute volume density functions $\tilde{\eta}^{(n),abs}$ converges weakly in the Skorokhod topology to $\eta^{abs}=(B^{\eta},u^{\eta}_b, A^{\eta}, u^{\eta}_a, \tau^{\eta})$ being the unique strong solution to the coupled SDE system
\begin{align}\label{eq:SDEabs}
\begin{aligned}
B^{\eta}(t) &= B_0 + \int_0^t p_b\left(\psi(\eta^{abs}(u))\right) du + \int_0^t r_b\left(\psi(\eta^{abs}(u))\right) dZ_b(u)\\
&\hspace{4cm} + \int_0^t \int_{[-M,M]} \theta_b\left(\psi(\eta^{abs}(u-)), y\right) \mu^Q_b(du, dy),\\
u^{\eta}_b(t,x) &= v_{b,0}(-x + B_0) + \int_0^t f_b\left[\psi(\eta^{abs}(u))\right]\left(-x+B^{\eta}(u)\right) du,\\
A^{\eta}(t) &= A_0 + \int_0^t p_a\left(\psi(\eta^{abs}(u))\right) du + \int_0^t r_a\left(\psi(\eta^{abs}(u))\right) dZ_a(u)\\
&\hspace{4cm} + \int_0^t \int_{[-M,M]} \theta_a\left(\psi(\eta^{abs}(u-)), y\right) \mu^Q_a(du, dy),\\
u^{\eta}_a(t,x) &= v_{a,0}(x - A_0) + \int_0^t f_a\left[\psi(\eta^{abs}(u))\right]\left(x-A^{\eta}(u)\right) du,\\
\tau^{\eta}(t) &:= \int_0^t \varphi(\psi(\eta^{abs}(u)))du
\end{aligned}
\end{align}
for all $(t,x)\in [0,T]\times\R$, where $Z_b, Z_a$ are two independent, standard Brownian motions and $\mu^Q_b, \mu^{Q}_a$ are two independent, homogeneous Poisson random measures with intensity measures $\lambda\otimes Q_b$ and $\lambda\otimes Q_a$, respectively, independent of $Z_b$ and $Z_a$.
\end{The}

\subsubsection*{Step 5: End of the proof}

With a slight abuse of notation, by the Skorokhod representation theorem we may assume that $\tilde{\eta}^{(n),abs}$ converges almost surely in the Skorokhod topology to $\eta^{abs}$. 
Hence, there exists a sequence of continuous, strictly increasing functions $\gamma_n: [0,T] \rightarrow [0,T]$ with $\sup_{t \in [0,T]} |\gamma_n(t) - t| \rightarrow 0$ such that
\begin{equation}\label{eq:Skorokhodrepresentation}
    \sup_{t\in [0,T]}\left\|\tilde{\eta}^{(n),abs}(\gamma_n(t)) - \eta^{abs}(t)\right\|_E \rightarrow 0 \quad \text{a.s.}
\end{equation}
Then,we have
\begin{align}\label{eq:skB}
\begin{aligned}
&\sup_{t \in [0,T]} \left\|\psi\left(\tilde{\eta}^{(n),abs}(\gamma_n(t))\right) - \psi(\eta^{abs}(t))\right\|_E\\
&\quad\leq \sup_{t\in [0,T]} \Big\{ \left|\tilde{B}^{(n)}(\gamma_n(t)) - B^\eta(t)\right| + \left|\tilde{A}^{(n)}(\gamma_n(t)) - A^\eta(t)\right|\\
&\qquad + \left\|\tilde{u}_b^{(n)}\left(\gamma_n(t), -(\cdot - \tilde{B}^{(n)}(\gamma_n(t)))\right) - u^\eta_b(t, -(\cdot - B^\eta(t)))\right\|_{L^2}\\
&\qquad  + \left\|\tilde{u}_a^{(n)}\left(\gamma_n(t), \cdot - \tilde{A}^{(n)}(\gamma_n(t))\right) - u^\eta_a(t, \cdot - A^\eta(t))\right\|_{L^2} +  \left|\tilde{\tau}^{(n)}(\gamma_n(t)) - \tau^\eta(t)\right|\Big\}.
\end{aligned}
\end{align}
Applying Lemma \ref{lem:uLip}, we can bound
\begin{align*}
&\left\|\tilde{u}_b^{(n)}\left(\gamma_n(t),-( \cdot - \tilde{B}^{(n)}(\gamma_n(t)))\right) - u^\eta_b(t, -(\cdot - B^\eta(t)))\right\|_{L^2}\\
&\qquad \leq  \left\|\tilde{u}^{(n)}_b(\gamma_n(t), \cdot) - u^\eta_b(t, \cdot)\right\|_{L^2} + \left\|u^\eta_b\left(t,-( \cdot - \tilde{B}^{(n)}(\gamma_n(t)))\right) - u^\eta_b(t, -(\cdot - B^\eta(t))) \right\|_{L^2}\\
&\qquad \leq \left\|\tilde{u}_b^{(n)}(\gamma_n(t), \cdot)- u^\eta_b(t, \cdot)\right\|_{L^2} + L(1+T)\left|\tilde{B}^{(n)}(\gamma_n(t))-B^\eta(t)\right|.
\end{align*}
An analogous estimate holds for the ask side. 
Plugging these bounds into equation \eqref{eq:skB} and applying \eqref{eq:Skorokhodrepresentation}, it follows that 
\[\sup_{t \in [0,T]} \left\|\psi\left(\tilde{\eta}^{(n),abs}(\gamma_n(t))\right) - \psi\left(\eta^{abs}(t)\right)\right\|_E \rightarrow 0 \quad \text{a.s.}\]
Hence, $\tilde{\eta}^{(n)} = \psi(\tilde{\eta}^{(n),abs}) \Rightarrow \psi(\eta^{abs}) =: \eta$, which solves \eqref{eq:eta}. Therefore, Proposition \ref{res:mainTheorem-detTimes} holds true. Now, an application of Corollary \ref{Cor:FromDetToRanTimes} finishes the proof of Theorem \ref{res:mainTheorem}.

\section{Proofs}\label{sec:proofs}

\begin{proof}[\textbf{Proof of Corollary \ref{cor:mainTheorem}}]
 Thanks to Theorem \ref{res:mainTheorem}, $S^{(n)}$ converges weakly in the Skorokhod topology to $S = \eta \circ \zeta,$ where $\eta = (B^{\eta}, v_b^{\eta}, A^{\eta}, v_a^{\eta}, \tau^{\eta})$ is the unique strong solution to the system given in \eqref{eq:eta} and $\zeta(t) := \inf\{s> 0: \tau^{\eta}(s) > t\}.$\par 
First, we will show that $\eta$ solves the following coupled SDE-SPDE system: for $(t,x) \in [0,T]\times \R$,
 \begin{align}\label{eq:eta2}
 \begin{aligned}
 dB^{\eta}(t) &= p_b(\eta(t))dt + r_b(\eta(t))dZ_b(t) + \int_{[-M,M]} \theta_b(\eta(t-),y) \mu^{Q_b}(dt,dy),\\
 dv_b^{\eta}(t,x) &= \left(-\frac{\partial v^{\eta}_b}{\partial x}(t,x) p_b(\eta(t)) + \frac{1}{2}\frac{\partial^2 v^{\eta}_b}{\partial x^2}(t,x) \left(r_b(\eta(t))\right)^2 + f_b[\eta(t)](x)\right) dt \\
&\hspace{3.5cm} - \frac{\partial v^{\eta}_b}{\partial x}(t,x) r_b(\eta(t)) dZ_b(t) + \Big(v_b^{\eta}(t-,x-\Delta B^\eta_t)- v_b^{\eta}(t-,x)\Big),\\
dA^{\eta}(t) &= p_a(\eta(t))dt + r_a(\eta(t))dZ_a(t) + \int_{[-M,M]} \theta_a(\eta(t-),y) \mu^{Q_a}(dt,dy),\\
 dv_a^{\eta}(t,x) &= \left(\frac{\partial v_a^{\eta}}{\partial x}(t,x) p_a(\eta(t)) + \frac{1}{2}\frac{\partial^2 v_a^{\eta}}{\partial x^2}(t,x) \left(r_a(\eta(t))\right)^2 + f_a[\eta(t)](x)\right) dt \\
&\hspace{3.5cm} + \frac{\partial v_a^{\eta}}{\partial x}(t,x) r_a(\eta(t)) dZ_a(t) + \Big(v_a^{\eta}(t-,x+\Delta A^\eta_t)- v_a^{\eta}(t-,x)\Big),\\
d\tau^{\eta}(t) &= \varphi(\eta(t))dt.
 \end{aligned}
 \end{align}
 
\newpage

Since both, $v_{b,0}$ and $f_b[s]$, are twice continuously differentiable, we can apply Itô's formula for semimartingales with jumps (cf. Theorem II.7.32 in \cite{P04}) and obtain for any $x\in\R$,
\begin{align*}
&v_{b,0}(x - (B^{\eta}(t) - B_0))\\
&= v_{b,0}(x)-\int_0^t v_{b,0}'(x - (B^{\eta}(s-)-B_0)) dB^{\eta}_s + \frac{1}{2} \int_0^t v_{b,0}''(x-(B^{\eta}(s)-B_0)) \left(r_b(\eta(s))\right)^2 ds\\
&\hspace{0.3cm}+ \sum_{0 < s\leq t} \left\{v_{b,0}(x-(B^{\eta}(s) -B_0)) - v_{b,0}(x-(B^{\eta}(s-) -B_0)) + v_{b,0}'(x-(B^{\eta}(s-) - B_0)) \Delta B^{\eta}_s\right\}
\end{align*}
as well as
\begin{align*}
\int_0^t& f_b[\eta(u)] (x-(B^{\eta}(t) - B^{\eta}(u))) du\\
&= \int_0^t f_b[\eta(u)](x) du - \int_0^t \left(\int_0^s f_b'[\eta(u)](x-(B^{\eta}(s-)-B^{\eta}(u))) du\right) dB^{\eta}_s\\
&\qquad + \frac{1}{2} \int_0^t \left(\int_0^s f_b''[\eta(u)](x-(B^{\eta}(s)-B^{\eta}(u))) du\right) \left(r_b(\eta(s))\right)^2 ds\\
&\qquad + \sum_{0 < s \leq t} \left\{ \int_0^s \left(f_b[\eta(u)](x-(B^{\eta}(s) - B^{\eta}(u))) - f_b[\eta(u)](x-(B^{\eta}(s-) - B^{\eta}(u)))\right) du \right.\\
&\hspace{3cm} \left.+ \left(\int_0^s f_b'[\eta(u)](x-(B^{\eta}(s-)-B^{\eta}(u))) du\right)\Delta B^{\eta}_s\right\}.
\end{align*}
Combining both equations and using that $v_b^\eta(t,x)=v_b^\eta(t-,x-\Delta B_t^\eta)$ for all $t\in[0,T]$, we conclude that
 \begin{align*}
 v^{\eta}_b(t,x) &= v_{b,0}(x)-\int_0^t \frac{\partial v^{\eta}_b}{\partial x}(s-,x) dB^{\eta}_s + \frac{1}{2}\int_0^t \frac{\partial^2 v^{\eta}_b}{\partial x^2}(s,x) \left(r_b(\eta(s))\right)^2ds + \int_0^t f_b[\eta(s)](x) ds\\
 &\qquad + \sum_{0< s \leq t} \frac{\partial v^{\eta}_b}{\partial x}(s-,x) \Delta B^{\eta}_s + \sum_{0<s \leq t} \left(v_b^{\eta}(s,x) - v^{\eta}_b(s-,x)\right)\\
 &= v_{b,0}(x)+\int_0^t \left(-\frac{\partial v^{\eta}_b}{\partial x}(s,x) p_b(\eta(s)) + \frac{1}{2}\frac{\partial^2 v^{\eta}_b}{\partial x^2}(s,x) \left(r_b(\eta(s))\right)^2 + f_b[\eta(s)](x)\right) ds\\
 &\qquad - \int_0^t \frac{\partial v^{\eta}_b}{\partial x}(s,x) r_b(\eta(s)) dZ_b(s) + \sum_{0< s\leq t}\left(v^{\eta}_b(s-,x-\Delta B^\eta_s) - v^{\eta}_b(s-,x)\right).
 \end{align*}
 Similarly, we can show that $v^{\eta}_a$ solves the SPDE $dv^{\eta}_a.$ Hence, the solution $\eta$ of the system in \eqref{eq:eta} indeed solves the stated SDE-SPDE system in \eqref{eq:eta2}.\par 
By the definition of $\zeta$, observe that $\zeta = (\tau^{\eta})^{-1}.$ Hence,
\[\zeta'(t) = \left((\tau^{\eta})^{-1}\right)' = \frac{1}{(\tau^{\eta})'((\tau^{\eta})^{-1}(t))} = \frac{1}{\varphi(\eta \circ \zeta (t))} = \frac{1}{\varphi(S(t))}.\]
Since $\zeta$ is a continuous time change, the processes $S,$ $Z_I,$ $I = b,a,$ and $X(\cdot, g),$ for $g \in C_b([-M,M]^2)$ are adapted to $\zeta$ in the sense of Definition 10.13 in \cite{J79}. In particular, for $I = b,a,$ we have by Theorem 10.17 in \cite{J79} that
 \[\langle Z_I\circ \zeta, Z_I \circ \zeta\rangle_t = \langle Z_I, Z_I \rangle_{\zeta(t)} = \zeta(t),\]
 since $Z_I,$ $I = b,a,$ are standard Brownian motions. Moreover, by Theorem 10.27 in \cite{J79}, there exists an integer-valued random jump measure $\tilde{\mu}_I^Q$ such that for all $t \in [0,T],$
 \[X_I(\zeta(t), g) := \int_{[-M,M]} g(y) \mu^{Q}_I([0,\zeta(t)], dy) = \int_{[-M,M]} g(y) \tilde{\mu}^{Q}_I([0,t], dy),\]
 whose compensator is given by $\tilde{\nu}^Q_I(dt, dy) = (\varphi(S(t)))^{-1}dt \times Q_I(dy).$ Finally, we can apply Proposition 10.21 and Theorem 10.27 in \cite{J79} and conclude for all bounded, continuous functions $g_1: E \rightarrow \R,$ $g_2: E \times [-M,M] \rightarrow \R,$ and $t\in [0,T],$
 \[\int_0^{\zeta(t)}g_1(\eta(u))du = \int_0^t g_1(S(u))(\varphi(S(u)))^{-1} du, \quad \int_0^{\zeta(t)}g_1(\eta(u))dZ_I(u)= \int_0^t g_1(S(u)) \zeta^{1/2}(u)d\tilde{Z}_I(u)\]
 and 
 \[\int_0^{\zeta(t)} \int_{[-M,M]}g_2(\eta(u), y) \mu^{Q}_I(du, dy) = \int_0^t \int_{[-M,M]} g_2(S(u), y) \tilde{\mu}^{Q}_I(du, dy)\]
where $\tilde{Z}_{b}, \tilde{Z}_a$ are two independent Brownian motions, independent of $\tilde{\mu}^{Q}_b$ and $\tilde{\mu}^{Q}_a.$ Combining these observations with the fact that $\eta$ solves the system in \eqref{eq:eta2}, we conclude that $S$, starting in $S_0$, indeed solves the stated SDE-SPDE system in \eqref{eq:sde-spde}.
\end{proof}

\begin{proof}[\textbf{Proof of Corollary \ref{Cor:FromDetToRanTimes}}]
According to Proposition \ref{res:mainTheorem-detTimes}, $\eta^{(n)}$ 
converges weakly in the Skorokhod topology to $\eta := (B^{\eta}, v^{\eta}_b, A^{\eta}, v^{\eta}_a, \tau^{\eta})$ being the unique strong solution of the coupled diffusion-fluid system  \eqref{eq:eta}. 
Since $\varphi^{(n)},$ $\varphi$ are strictly positive, $\tau^{\eta,(n)},$ $\tau^{\eta}$ are increasing (resp.~strictly increasing) and therefore,
\[\zeta^{(n)}(t) = (\tau^{\eta,(n)})^{-1}(t) = \inf\{u > 0: \tau^{\eta,(n)}(u) > t\} \quad \text{ and } \quad \zeta(t) = (\tau^{\eta})^{-1}(t)\]
exist and are increasing. Moreover, $\zeta$ is continuous and even strictly increasing. By Corollary 13.6.4 in Whitt \cite{W02}, the inverse map is continuous at strictly increasing functions. By the continuous mapping theorem we therefore conclude that $(\zeta^{(n)},\eta^{(n)}) \Rightarrow (\zeta,\eta)$ in the Skorokhod topology. Especially, this proves i).

Since $\zeta$ is continuous and strictly increasing almost surely and since $\eta$ is continuous at time $\zeta(T)$ almost surely, Theorem 3.1 in Whitt \cite{W80} yields the continuity of the composition map in $(\zeta, \eta)$ in the Skorokhod topology. Hence, we can apply the continuous mapping theorem to conclude that
\[S^{(n),*} = \eta^{(n)}\circ (\zeta^{(n)}-\tn) \Rightarrow \eta \circ \zeta =: S\]
in the Skorokhod topology, which proves ii).

Finally, it follows from Theorem 12.5 in \cite{Billingsley} that $\tau^{\eta,(n)}(T)$ converges weakly to $\tau^{\eta}(T)$ and hence
\begin{equation*}
    \Pro\left[\sup_{t\in[0,T]}\big\|S^{(n),*}(t)-S^{(n)}(t)\big\|_E>0\right]=\Pro\left[\tau^{\eta,(n)}(T)> T\right]\rightarrow\Pro\left[\tau^\eta(T)> T\right]=0,
\end{equation*}
since by Assumption \ref{ass:randTimes} iii),
\begin{equation*}
\tau^\eta(T)=\int_0^T\varphi(\eta(t))dt\leq T.
\end{equation*}
This proves iii).
\end{proof}

\begin{proof}[\textbf{Proof of Proposition \ref{res:limitSJ}}]
First note that for all $n \in \N,\ k\leq T_n$,
 \begin{equation*}
 \frac{\xn \left| p_b^{(n)}(\s) \right|}{\left(r_b^{(n)}(\s)\right)^2} = \frac{\xn\left|\E\left[\dBn_k \1\left(0 < |\dBn_k|\leq \delta_n\right) \Big| \F_{k-1}\right]\right|}{\E\left[(\dBn_k)^2 \1\left(0 < |\dBn_k|\leq \delta_n\right) \Big| \F_{k-1}\right]}\leq 1\quad\text{a.s.}
 \end{equation*}
  Together with Assumption \ref{ass:Prob}, this gives the bound 
\begin{align*}
  \frac{\tn (p^{(n)}_b(\s))^2}{(r^{(n)}_b(\s))^2}\leq C\frac{\tn}{\xn},
 \end{align*}
which converges to zero by Assumption \ref{ass:Scaling}. A similar bound holds for $p^{(n)}_a$ and $r^{(n)}_a$. Hence, for all $t\in[0,T]$ and $I=b,a$, 
 \begin{equation*}
\sum_{k=1}^{\ttn} \E\left[\left(\delta Z_{k,I}^{(n)}\right)^2\big|\F_{k-1}\right] = \sum_{k=1}^{\ttn} \frac{\tn \left(r^{(n)}(\s)\right)^2- \left(\tn p^{(n)}(s)\right)^2}{\left(r^{(n)}(\s)\right)^2}\rightarrow t.
  \end{equation*}
Moreover, as $A$ and $C$ events do not occur simultaneously, we have for all $t\in[0,T]$, 
\begin{equation*}
\sum_{k=1}^{\ttn} \E\left[\delta Z^{(n)}_{b,k}, \delta Z^{(n)}_{a,k}|\F_{k-1}\right]=
- \sum_{k=1}^{\ttn} \frac{(\tn)^2p^{(n)}_b(S^{(n)}_{k-1})p^{(n)}_a(S^{(n)}_{k-1})}{r^{(n)}_b(S^{(n)}_{k-1}) r^{(n)}_a(S^{(n)}_{k-1})}\rightarrow 0.
\end{equation*}
Finally, we observe that for all $\varepsilon > 0$ and $I=b,a$,
\begin{align*}
\sum_{k=1}^{\ttn} \E\left[|\delta Z^{(n)}_{I,k}|^2 \1(|\delta Z^{(n)}_{I,k}|> \varepsilon) \right] \leq
\sum_{k=1}^{\ttn} \E\left[|\delta Z^{(n)}_{I,k}|^2 \1\left(\frac{\delta_n}{\eta_n}> \varepsilon\right) \right] \leq
\frac{\delta_n}{\varepsilon\eta_n}\sum_{k=1}^{\ttn} \E|\delta Z^{(n)}_{I,k}|^2
\leq \frac{t\delta_n}{\varepsilon\eta_n},
\end{align*}
which converges to zero by Assumption \ref{ass:probVol} i). As $(\delta Z^{(n)}_{b,k},\delta Z^{(n)}_{a,k})_{k,n}$ is a triangular martingale difference arrays, we may conclude by the functional limit theorem that $(Z^{(n)}_b,Z^{(n)}_a)\Rightarrow(Z_b,Z_a)$ in the Skorokhod topology, where $Z_b$ and $Z_a$ are independent Brownian motions.
\end{proof}

\begin{proof}[\textbf{Proof of Lemma \ref{lem:comLJ}}]
As both representations can be proven in the same way, we will only present the proof for $B^{\eta,\ell,(n)}.$ The statement follows directly from the definition of $\mu^{J^{(n)}}_b$:
\begin{align*}
\int_0^t &\int_{[-M,M]} \theta^{(n)}_b(\eta^{(n)}(u-), y)\mu^{J^{(n)}}_b(du, dy)\\
&=\sum_{k=1}^{\ttn} \int_{[-M,M]} \theta^{(n)}_b(\s, y)\ \mu^{J^{(n)}}_b\left(\left[\t_k, \t_{k+1}\right) , dy\right)\\
&\quad\stackrel{(1)}{=}\sum_{k=1}^{\ttn} \sum_{j\in\Z^{(n)}_M} \theta^{(n)}_b(\s, x_j^{(n)})\mu^{\eta,(n)}_b\left(\left[\t_k, \t_{k+1}\right) , \left\{\theta^{(n)}\big(\s,x_j^{(n)}\big)\right\}\right)\\
&\quad\stackrel{(2)}{=}
\sum_{k=1}^{\ttn} \sum_{j\in\Z}  x_j^{(n)}\mu^{\eta,(n)}_b\left(\left[\t_k, \t_{k+1}\right) , \left\{x_j^{(n)}\right\}\right) =
\int_{\R} y\, \mu^{\eta,(n)}_b\left([0,t] , dy\right) = B^{\eta,\ell,(n)}(t),
\end{align*}
where in (1) we used the definition of $\mu^{J^{(n)}}_b$ noting that it only charges the grid points $x_j^{(n)}$, $j \in \Z^{(n)}_M$, and in (2) we used Assumption \ref{ass:ExQ} v).
\end{proof}

\begin{proof}[\textbf{Proof of Theorem \ref{res:limitLJ}}]
From \eqref{def:compJn} we have 
for all $g \in C_b([-M,M])$, $t \in [0,T]$, and $I=b,a$,
\begin{align*}
\int_{[-M,M]} g(y)\nu_I^{J^{(n)}}&([0,t], dy)=
\sum_{k=1}^{\ttn}\tn\sum_{j\in \Z^{(n)}_M} g(x_j^{(n)})K_I^{(n)}\Big(S_{k-1}^{(n)},\left\{\theta^{(n)}_I\Big(S^{(n)}_{k-1},x_j^{(n)}\Big)\right\}\Big)\\
&+\sum_{k=1}^{\ttn}\tn\sum_{j\in \Z^{(n)}_M} g(x_j^{(n)})\1\Big(\theta^{(n)}_I(S^{(n)}_{k-1},x_j^{(n)})=0\Big)Q_I\Big(\left[x_{j}^{(n)},x_{j+1}^{(n)}\right)\Big).
\end{align*}
First, from Lemma \ref{lem:theta} we see that
\begin{align*}
&\sum_{k=1}^{\ttn} \tn\sum_{j\in \Z^{(n)}_M} g(x_j^{(n)})K_I^{(n)}\Big(S_{k-1}^{(n)},\left\{\theta_I^{(n)}\Big(S^{(n)}_{k-1},x_j^{(n)}\Big)\right\}\Big)\\
&\qquad=\sum_{k=1}^{\ttn} \tn\sum_{j} g(x_j^{(n)})K^{(n)}_I\left(S_{k-1}^{(n)},\theta_I\Big(S^{(n)}_{k-1},\left[x_{j}^{(n)},x_{j+1}^{(n)}\right)\Big)\right).
\end{align*}
 Next, we bound
\begin{align*}
I_n^{(1)}:=\sum_{k=1}^{\ttn}\tn\sum_{j} g(x_j^{(n)})&\left[K^{(n)}_I\left(S_{k-1}^{(n)},\theta_I\Big(S^{(n)}_{k-1},\left[x_{j}^{(n)},x_{j+1}^{(n)}\right)\Big)\right)\right.\\
&\left.\qquad\qquad\qquad-K_I\left(S_{k-1}^{(n)},\theta_I\Big(S^{(n)}_{k-1},\left[x_{j}^{(n)},x_{j+1}^{(n)}\right)\Big)\right)\right]
\end{align*}
by
\begin{align*}
    |I_n^{(1)}|\leq t\cdot\|g\|_\infty\cdot \sup_{s\in E}\sum_{j}\left|K^{(n)}_I\left(s,\theta_I\Big(s,\left[x_{j}^{(n)},x_{j+1}^{(n)}\right)\Big)\right)-K_I\left(s,\theta_I\Big(s,\left[x_{j-1}^{(n)},x_j^{(n)}\right)\Big)\right)\right|,
\end{align*}
which converges to zero by Assumption \ref{ass:ExQ} iii).
Moreover, we note that also
\begin{align*}
I_n^{(2)}:= \sum_{k=1}^{\ttn}\tn\sum_{j\in \Z^{(n)}_M} g(x_j^{(n)})\int_{[\x_j, \x_{j+1})}\Big[\1\Big(\theta_I^{(n)}&(S^{(n)}_{k-1},x_j^{(n)})=0\Big)\\
&-\1\Big(\theta_I(S^{(n)}_{k-1},x)=0\Big)\Big]Q_I(dx)
\end{align*}
goes to zero by Assumption \ref{ass:ExQ} vi). Therefore,
\begin{align*}
&\int_{[-M,M]} g(y)\nu^{J^{(n)}}_I([0,t]\times dy)=\sum_{k=1}^{\ttn}\tn\sum_{j} g(x_j^{(n)})\Big[K_I\Big(S_{k-1}^{(n)},\theta_I\Big(S^{(n)}_{k-1},\left[x_{j}^{(n)},x_{j+1}^{(n)}\right)\Big)\Big)\\
&\hspace{6cm} +
Q_I\left(\left\{x\in\left[x_{j}^{(n)},x_{j+1}^{(n)}\right):\theta_I\Big(S_{k-1}^{(n)},x\Big)=0\right\}\right)\Big] + I_n^{(1)}+I_n^{(2)}\\
&\stackrel{(1)}{=}\sum_{k=1}^{\ttn}\tn\sum_{j} g(x_j^{(n)})Q_I\left(\left[x_{j}^{(n)},x_{j+1}^{(n)}\right)\right) + I_n^{(1)}+I_n^{(2)}\rightarrow t\int_{[-M,M]} g(y)Q_I(dy),
\end{align*}
where in (1) we used Assumption \ref{ass:ExQ} ii). By the definition of the measure $\mu^{J^{(n)}}$ we then also have
\begin{align*}
&\int_{[-M,M]^2} g(x,y)\nu^{J^{(n)}}([0,t], dx, dy)
\rightarrow t\int_{[-M,M]^2}g(x,y)\left[Q_b(dx)\varepsilon_0(dy)+Q_a(dy)\varepsilon_0(dx)\right].
\end{align*}
As a Poisson random measure is uniquely determined by its intensity measure we conclude by Theorem 2.6 in \cite{KurtzProtter2} that for any $g_1, \cdots, g_m \in C_b([-M,M]^2)$,
\[\left(X^{(n)}(\cdot, g_1), \cdots, X^{(n)}(\cdot, g_m)\right) \Rightarrow (X(\cdot, g_1), \cdots, X(\cdot, g_m))\]
in $\mathcal{D}(\R^m; [0,T])$ with
\[X(t,g):=\int_{[-M,M]^2}g(x,y)\left[\mu^Q_b([0,t]\times dx)\varepsilon_0(dy)+\mu^Q_a([0,t]\times dy)\varepsilon_0(dx)\right],\]
where $\mu_I^Q,\ I=b,a$, are two independent Poisson random measures with intensity measure $\lambda\times Q_I$ for $I=b,a,$ respectively.
\end{proof}

\begin{Rmk}\label{rmk:jumpEstimate}
Carefully inspecting the above proof, we see that we have actually shown for $I=b,a$ and any $g\in C_b([-M,M])$ and $k=0,\dots,T_n$ the almost sure estimate
\[
\left|\int_{-M}^M g(y)\nu_I^{J^{(n)}}\Big(\left[t_k^{(n)},t_{k+1}^{(n)}\right)\times dy\Big)-\tn\sum_{j\in \Z^{(n)}_M} g(x_j^{(n)})Q_I\Big(\left[x_j^{(n)},x_{j+1}^{(n)}\right)\Big)\right|=o\left(\|g\|_\infty \tn\right).\]
Moreover, the estimate would still be true, if $g$ was $\F_{k-1}$-measurable.
This will be used later in Section \ref{proof:convApprox} and in the proof of Lemma \ref{res:UT}.
\end{Rmk}

\begin{proof}[\textbf{Proof of Proposition \ref{res:conRemainder}}]
First, we observe that
\[ \left\|R^n_{b,k}\right\|_{L^2} = \left\|\vn \sum_{j=1}^k\left(M^{(n)}_{b,j} - f^{(n)}_b[S^{(n)}_{j-1}]\right) (\cdot + B^{(n)}_{j-1})\right\|_{L^2} =: \left\|\tilde{R}^{(n)}_{b,k}\right\|_{L^2}.\]
Therefore, it is enough to prove the result for $\left(\tilde{R}^{(n)}_{b,k}\right)_{k\leq T_n}.$ By Assumption \ref{ass:probVol} ii), this sequence is a martingale with values in $L^2(\R)$. Applying Lemma \ref{techL:L2est} as well as Assumptions \ref{ass:probVol} i) and \ref{ass:Scaling} , there exists $C>0$ such that
\begin{align*}
    \E\left[\sup_{k\leq T_n} \|\tilde{R}^{(n)}_{b,k}\|^2_{L^2}\right] &\leq
    C \E\left[\sum_{k=1}^{T_n}\left\|\vn \left(M^{(n)}_{b,k}-f^{(n)}_b[\s]\right)\right\|^2_{L^2}\right]\\
    &\leq C T_n (\vn)^2 \sup_{k \leq T_n}\E\left[\left\|(M^{(n)}_{b,k})^2\right\|_{L^1}\right]\\
    &\leq CT\frac{(\vn)^2}{(\tn)^2} \frac{\tn}{\xn}\sup_{k\leq T_n} \E\left[(\omega_k^{(n)})^2\right]\rightarrow 0.
\end{align*}
Hence, $\E\left[\sup_{k\leq T_n} \|R^{(n)}_{b,k}\|^2_{L^2}\right] \rightarrow 0$ and analogously, one can show that $\E\left[\sup_{k\leq T_n} \|R^{(n)}_{a,k}\|^2_{L^2}\right] \rightarrow 0$. Next, observe that also
\[    R^{(n)}_{\varphi,k} := \tn \sum_{j=1}^k \left(\varphi^{(n)}_j - \varphi^{(n)}(S^{(n)}_{j-1})\right)\]
 defines a martingale and by Lemma \ref{techL:L2est} we have
 \begin{align*}
     \E\left[\sup_{k \leq T_n} \left| R^{(n)}_{\varphi,k}\right|^2\right] &\leq 
     C\E\left[\sum_{k=1}^{T_n}\left|\tn \left(\varphi^{(n)}_k - \varphi^{(n)}(S^{(n)}_{k-1})\right)\right|^2\right]\leq C T \tn \sup_{k\leq T_n} \E\left[\left(\varphi^{(n)}_k\right)^2\right] \rightarrow 0,
\end{align*}
where we applied Assumption \ref{ass:randTimes} i).
\end{proof}

\subsection{Proof of Proposition \ref{res:convApprox}}\label{proof:convApprox}

In order to prove the claim, we will introduce two further candidates for approximations of $S^{(n)}$ -- namely $\bar{S}^{(n)}$ and $\hat{S}^{(n)}$. We define $\bar{S}^{(n)}_k = (\bar{B}^{(n)}_k, \bar{v}^{(n)}_{b,k}, \bar{A}^{(n)}_k, \bar{v}^{(n)}_{a,k},\bar{\tau}^{(n)}_k)$ for $k=0,\dots,T_n$ as  
\begin{align*}
\bar{B}^{(n)}_k &:= B^{(n)}_k,\qquad \bar{A}^{(n)}_k := A^{(n)}_k,\qquad \bar{\tau}^{(n)}_k := \sum_{j=1}^k \varphi(S^{(n)}_{j-1})\tn,\\
\bar{v}^{(n)}_{b,k}(x) &:= v_{b,0}\left(x - \left(B^{(n)}_k - B^{(n)}_0\right)\right) + \sum_{j=1}^k f_b\left[S^{(n)}_{j-1}\right]\left(x - \left(B^{(n)}_k - B^{(n)}_{j-1}\right)\right) \tn,\\
\bar{v}^{(n)}_{a,k}(x) &:= v_{a,0}\left(x + \left(A^{(n)}_k - A^{(n)}_0\right)\right) + \sum_{j=1}^k f_a\left[S^{(n)}_{j-1}\right]\left(x + \left(A^{(n)}_k - A^{(n)}_{j-1}\right)\right) \tn.
\end{align*}
Note that the coefficient functions in the dynamics of $\bar{S}^{(n)}_k$ still depend on the original LOB sequence. In the second approximation $\hat{S}^{(n)}_k=(\hat{B}^{(n)}_k, \hat{v}^{(n)}_{b,k},\hat{A}^{(n)}_k,\hat{v}^{(n)}_{a,k}, \hat{\tau}^{(n)}_k)$ the diffusion coefficient and the coefficient of the compensated jumps depend on the approximated LOB dynamics $\tilde{S}^{(n)}$, whereas all other coefficients still depend on the original LOB sequence $S^{(n)}$, i.e.~for $k=0,\dots,T_n$, 
\begin{align*}
\begin{aligned}
	\hat{B}^{(n)}_k &:= B^{(n)}_0 + \sum_{j=1}^k \Bigg( p^{(n)}_b\left(S^{(n)}_{j-1}\right) \tn + r^{(n)}_b\left(\tilde{S}^{(n)}_{j-1}\right)\delta Z^{(n)}_{b,j}\\
	&\hspace{2.5cm}+ \int_{[-M,M]} \theta^{(n)}_b\left(\tilde{S}^{(n)}_{j-1}, y\right) \left(\mu^{J^{(n)}}_b-\nu^{J^{(n)}}_b\right)\left(\left[\t_j, \t_{j+1}\right), dy\right)\\
	&\hspace{2.5cm}+ \int_{[-M,M]} \theta^{(n)}_b\left(S^{(n)}_{j-1}, y\right) \nu^{J^{(n)}}_b\left(\left[\t_j, \t_{j+1}\right), dy\right)\Bigg),\\
	\hat{v}^{(n)}_{b,k}(x) & := v_{b,0}\left(x - \left(\hat{B}^{(n)}_k - B^{(n)}_0\right)\right) + \sum_{j=1}^k f_b\left[S^{(n)}_{j-1}\right]\left(x - \left(\hat{B}^{(n)}_k - \hat{B}^{(n)}_{j-1}\right)\right) \tn,\\
	\hat{A}^{(n)}_k &:= A^{(n)}_0 + \sum_{j=1}^k \Bigg( p^{(n)}_a\left(S^{(n)}_{j-1}\right) \tn + r^{(n)}_a\left(\tilde{S}^{(n)}_{j-1}\right)\delta Z^{(n)}_{a,j}\\
	&\hspace{2.5cm}+ \int_{[-M,M]} \theta^{(n)}_a\left(\tilde{S}^{(n)}_{j-1}, y\right) \left(\mu^{J^{(n)}}_a-\nu^{J^{(n)}}_a\right)\left(\left[\t_j, \t_{j+1}\right), dy\right)\\
	&\hspace{2.5cm}+ \int_{[-M,M]} \theta^{(n)}_a\left(S^{(n)}_{j-1}, y\right) \nu^{J^{(n)}}_a\left(\left[\t_j, \t_{j+1}\right), dy\right)\Bigg),\\
	\hat{v}^{(n)}_{a,k}(x) & := v_{a,0}\left(x - \left(\hat{A}^{(n)}_k - A^{(n)}_0\right)\right) + \sum_{j=1}^k f_a\left[S^{(n)}_{j-1}\right]\left(x - \left(\hat{A}^{(n)}_k - \hat{A}^{(n)}_{j-1}\right)\right) \tn,\\
	\hat{\tau}^{(n)}_k &:= \bar{\tau}^{(n)}_k.
\end{aligned}
\end{align*}

In what follows, let $(a_n)_{n\in \N}$ be a deterministic null sequence possibly changing from line to line. Further, we write $A \lesssim B$ if $A \leq CB$ for some deterministic constant $C > 0.$ 

\newpage 

\textbf{Step 1:} We prove that $\bar{S}^{(n)}$ is indeed an approximation for $S^{(n)}$.  Applying Proposition \ref{res:conRemainder} as well as Assumptions \ref{ass:IV}, \ref{ass:randTimes}, \ref{ass:probVol} iii), and \ref{ass:Scaling}, we see that 
\begin{align*}
 \E&\left[\sup_{k\leq T_n}\left\|S^{(n)}_k - \bar{S}^{(n)}_k\right\|_{E}^2\right]\\
 &\quad \lesssim   \|v_{b,0}^{(n)}-v_{b,0}\|^2_{L^2} + \|v_{a,0}^{(n)}-v_{a,0}\|^2_{L^2}+ \frac{(\vn)^2}{(\tn)^2}\left[\sup_{s\in E} \left\|f_b[s]-f^{(n)}_b[s]\right\|^2_{L^2}+\sup_{s\in E} \left\|f_a[s]-f^{(n)}_a[s]\right\|^2_{L^2} \right]\\
 &\qquad+ \E\left[\sup_{k\leq T_n} \left\|R^{(n)}_{b,k}\right\|^2_{L^2}\right]+ \E\left[\sup_{k\leq T_n} \left\| R^{(n)}_{a,k}\right\|^2_{L^2}\right]
 +  \sup_{s\in E}\left|\varphi(s) - \varphi^{(n)}(s)\right|^2 + \E\left[\sup_{k\leq T_n} \left|R^{(n)}_{\varphi, k}\right|^2\right] + a_n \rightarrow 0.
\end{align*}

 \textbf{Step 2:} We prove upper bounds for the pathwise $L^2(\R)$-errors of the approximated volume functions. For all $k\leq T_n$, applying equation \eqref{ass:Lcond} in Assumption \ref{ass:IV}, we have
\[ \left\|v_{b,0}\left(\cdot - \left(\Bn_k-\Bn_0\right)\right)-v_{b,0}\left(\cdot - \left(\tilde{B}^{(n)}_k - \Bn_0\right)\right)\right\|_{L^2} \leq L\left|B^{(n)}_k - \tilde{B}^{(n)}_k\right|.\]
 Further, using the Lipschitz-continuity of $f_b: E \times \R \rightarrow \R$ in both variables, cf.~Assumption \ref{ass:probVol} iii), we have
 \begin{align*}
& \left\|\sum_{j=1}^k \left(f_b[S^{(n)}_{j-1}](\cdot - (B^{(n)}_k - B^{(n)}_{j-1})) - f_b[\tilde{S}^{(n)}_{j-1}](\cdot - (\tilde{B}^{(n)}_k - \tilde{B}^{(n)}_{j-1}))\right) \tn \right\|_{L^2}\\
 & \leq \tn \sum_{j=1}^k \left\{\left\|f_b[S^{(n)}_{j-1}](\cdot - (B^{(n)}_k - B^{(n)}_{j-1})) - f_b[\tilde{S}^{(n)}_{j-1}](\cdot - (B^{(n)}_k - B^{(n)}_{j-1}))\right\|_{L^2}\right. \\
 & \hspace{3cm} \left. + \left\|f_b[\tilde{S}^{(n)}_{j-1}](\cdot - (B^{(n)}_k - B^{(n)}_{j-1})) - f_b[\tilde{S}^{(n)}_{j-1}](\cdot - (\tilde{B}^{(n)}_k - \tilde{B}^{(n)}_{j-1}))\right\|_{L^2}\right\}\\
 & \leq \tn \sum_{j=1}^k L \left\{\left\|S^{(n)}_{j-1} - \tilde{S}^{(n)}_{j-1}\right\|_{E} + \left|B^{(n)}_k - B^{(n)}_{j-1} - \left(\tilde{B}^{(n)}_k - \tilde{B}^{(n)}_{j-1}\right)\right|\right\}\\
 &\leq \tn \sum_{j=1}^k 2L \left\|S^{(n)}_{j-1} - \tilde{S}^{(n)}_{j-1}\right\|_{E} + LT \left|B^{(n)}_k - \tilde{B}^{(n)}_k\right|.
 \end{align*}
Combining both bounds, we conclude that
\begin{equation}\label{v-bar-tilde}
\left\|\bar{v}^{(n)}_{b,k} - \tilde{v}^{(n)}_{b,k}\right\|_{L^2} \lesssim \left|B^{(n)}_k - \tilde{B}^{(n)}_k\right| + \tn \sum_{j=1}^k\left\|S^{(n)}_{j-1} - \tilde{S}^{(n)}_{j-1}\right\|_{E}.
\end{equation}
Similarly, one can show that
\begin{equation}\label{eq:v-hat-tilde}
\left\|\hat{v}^{(n)}_{b,k} - \tilde{v}^{(n)}_{b,k}\right\|_{L^2} \lesssim \left|\hat{B}^{(n)}_k - \tilde{B}^{(n)}_k\right| +\tn \sum_{j=1}^k\left\{\left|\hat{B}^{(n)}_{j-1} - \tilde{B}^{(n)}_{j-1}\right|+ \left\|S^{(n)}_{j-1} - \tilde{S}^{(n)}_{j-1}\right\|_{E}\right\}
\end{equation}
and
\begin{equation}\label{eq:v-bar-hat}
\left\|\bar{v}^{(n)}_{b,k} - \hat{v}^{(n)}_{b,k}\right\|_{L^2} \lesssim \left|B^{(n)}_k - \hat{B}^{(n)}_k\right| + \tn \sum_{j=1}^k\left|B^{(n)}_{j-1} - \hat{B}^{(n)}_{j-1}\right|\lesssim \sup_{j\leq k}\left|B^{(n)}_j - \hat{B}^{(n)}_j\right|.
\end{equation}
Analogous estimates hold for the approximating processes on the ask side.

\newpage 

\textbf{Step 3:} We prove upper bounds for the differences of the price and time coefficient functions. Let $i \leq T_n$. Since $(\delta Z^{(n)}_{b,j})_{j \leq T_n}$ defines a martingale difference array, we can apply Doob's inequality for $p = 2$ and Assumption \ref{ass:limitFct} to conclude 
\begin{align}\label{eq:cR}
\E\left[\sup_{k\leq i} \left(\sum_{j=1}^k\left(r^{(n)}_b(S^{(n)}_{j-1}) - r^{(n)}_b(\tilde{S}^{(n)}_{j-1})\right)\delta Z^{(n)}_{b,j}\right)^2\right] &\lesssim \E\left[\sum_{j=1}^i \left(r^{(n)}_b(S^{(n)}_{j-1}) - r^{(n)}_b(\tilde{S}^{(n)}_{j-1})\right)^2 \tn\right]\notag\\
& \lesssim  \tn \sum_{j=1}^i \E\left\|S_{j-1}^{(n)}-\tilde{S}^{(n)}_{j-1}\right\|^2_{E}+ a_n.
\end{align}
For the compensated jumps, we can argue as follows:
\begin{align}\label{eq:cT1}
\E&\left[\sup_{k\leq i} \left(\sum_{j=1}^k \int_{[-M,M]}\left(\theta_b^{(n)}(S^{(n)}_{j-1},y) - \theta_b^{(n)}(\tilde{S}^{(n)}_{j-1},y)\right) \, \left(\mu^{J^{(n)}}_b-\nu^{J^{(n)}}_b\right)\left(\left[\t_j, \t_{j+1}\right), dy\right)\right)^2\right]\notag\\
& \overset{(1)}{\lesssim} \sum_{j=1}^i\E\left[ \left( \int_{[-M,M]} \left(\theta^{(n)}_b(S^{(n)}_{j-1},y) - \theta^{(n)}_b(\tilde{S}^{(n)}_{j-1},y)\right) \, \left(\mu^{J^{(n)}}_b-\nu^{J^{(n)}}_b\right)\left(\left[\t_j, \t_{j+1}\right), dy\right)\right)^2\right]\notag\\
&\lesssim \sum_{j=1}^i\E\left[ \left( \int_{[-M,M]}\left|\theta^{(n)}_b(S^{(n)}_{j-1},y) - \theta^{(n)}_b(\tilde{S}^{(n)}_{j-1},y)\right| \, \mu^{J^{(n)}}_b\left(\left[\t_j, \t_{j+1}\right) , dy\right)\right)^2\right]\notag\\
&\quad+\sum_{j=1}^i\E\left[ \left( \int_{[-M,M]}\left|\theta^{(n)}_b(S^{(n)}_{j-1},y) - \theta_b^{(n)}(\tilde{S}^{(n)}_{j-1},y)\right| \, \nu^{J^{(n)}}_b\left(\left[\t_j, \t_{j+1}\right) , dy\right)\right)^2\right]\notag\\
&\overset{(2)}{=} \sum_{j=1}^i\E\left[ \int_{[-M,M]}\left|\theta_b^{(n)}(S^{(n)}_{j-1},y) - \theta_b^{(n)}(\tilde{S}^{(n)}_{j-1},y)\right|^2 \, \nu_b^{J^{(n)}}\left(\left[\t_j, \t_{j+1}\right) , dy\right)\right]\notag\\
&\quad+\sum_{j=1}^i\E\left[ \left( \int_{[-M,M]}\left|\theta_b^{(n)}(S^{(n)}_{j-1},y) - \theta^{(n)}_b(\tilde{S}^{(n)}_{j-1},y)\right| \, \nu^{J^{(n)}}_b\left(\left[\t_j, \t_{j+1}\right) , dy\right)\right)^2\right]\notag\\
&\overset{(3)}{\lesssim} \sum_{j=1}^i\E\left[ \left(\left\|S^{(n)}_{j-1} - \tilde{S}^{(n)}_{j-1}\right\|^2_{E}+(\xn)^2 \right)\, \nu^{J^{(n)}}_b\left(\left[\t_j, \t_{j+1}\right) , [-M,M]\right)\right]\notag\\
&\quad+\sum_{j=1}^i\E\left[ \left( \Big(\left\|S^{(n)}_{j-1} - \tilde{S}^{(n)}_{j-1}\right\|_{E}+\xn\Big) \, \nu_b^{J^{(n)}}\left(\left[\t_j, \t_{j+1}\right) , [-M,M]\right)\right)^2\right]\notag\\
&\overset{(4)}{\lesssim} \tn(1+\tn)\sum_{j=1}^i\E\left[ \left(\left\|S^{(n)}_{j-1} - \tilde{S}^{(n)}_{j-1}\right\|^2_{E}+(\xn)^2 \right)\,\left( Q_b\Big([-M,M]\Big)+a_n\right)\right]\notag\\
&\lesssim \tn \sum_{j=1}^i \E\left\|S^{(n)}_{j-1} - \tilde{S}^{(n)}_{j-1}\right\|^2_{E} + a_n.
\end{align}
In (1) we applied the Burkholder-Davis-Gundy inequality for $p=2$ (cf.~Theorem IV.4.48 in \cite{P04}). In (2) we used the fact that $\mu^{J^{(n)}}_b$ is an integer-valued random measure, which only increases at times $t_k^{(n)},\ k=1,\dots,T_n$, and that $\nu^{J^{(n)}}$ is the compensator of $\mu^{J^{(n)}}_b$. Finally, in (3) we applied the uniform Lipschitz-estimate for $\theta$ (cf. Assumption \ref{ass:ExQ} iv)) and in (4) we used the estimate in Remark \ref{rmk:jumpEstimate} with $g\equiv 1$.\par
For the drift component, using the Lipschitz-continuity of $p_b$, we have
\begin{align}\label{eq:cP}
\sup_{k\leq i} \left|\sum_{j=1}^k \left(p_b^{(n)}(S^{(n)}_{j-1}) - p_b^{(n)}(\tilde{S}^{(n)}_{j-1})\right)\tn \right| &\leq \tn \sum_{j=1}^i  \left|p^{(n)}_b(S^{(n)}_{j-1})-p_b^{(n)}(\tilde{S}^{(n)}_{j-1}) \right|\notag\\
&\leq L \tn \sum_{j=1}^i \left\|S^{(n)}_{j-1}-\tilde{S}^{(n)}_{j-1}\right\|_{E} + a_n
\end{align}
and analogously for the time component, using the Lipschitz-continuity of $\varphi$, we conclude
\begin{align}\label{eq:cTime}
\sup_{k\leq i}\left|\sum_{j=1}^k \left(\varphi(S^{(n)}_{j-1}) - \varphi(\tilde{S}^{(n)}_{j-1})\right) \tn \right| \leq L\tn \sum_{j=1}^i \left\|S^{(n)}_{j-1}-\tilde{S}^{(n)}_{j-1}\right\|_{E} + a_n.
\end{align}
Finally, applying once more Assumption \ref{ass:ExQ} iv) and the estimate in Remark \eqref{rmk:jumpEstimate}, we have
\begin{align}\label{eq:cT2}
 \sup_{k\leq i} & \Bigg|\sum_{j=1}^k \int_{[-M,M]} \left(\theta^{(n)}_b(S^{(n)}_{j-1}, y)-\theta^{(n)}_b(\tilde{S}^{(n)}_{j-1},y)\right) \, \nu^{J^{(n)}}_b\left(\left[\t_j, \t_{j+1}\right) \times dy\right) \Bigg| \notag\\
& \lesssim \sum_{j=1}^i \left(\Big\|S^{(n)}_{j-1}- \tilde{S}^{(n)}_{j-1} \Big\|_{E}+\xn \right) \, \nu^{J^{(n)}}_b\left(\left[\t_j, \t_{j+1}\right)\times [-M,M] \right)\notag\\
& \lesssim \tn\sum_{j=1}^i \left(\Big\|S^{(n)}_{j-1}- \tilde{S}^{(n)}_{j-1} \Big\|_{E}+\xn \right) \, \left(Q_b\Big( [-M,M] \Big)+a_n\right)\notag\\
&\lesssim \tn \sum_{j=1}^i \left\|S^{(n)}_{j-1}-\tilde{S}^{(n)}_{j-1}\right\|_{E}+a_n.
\end{align}
Again, analogous estimates hold for the respective processes on the ask side.

\textbf{Step 4:} We prove that $\hat{S}^{(n)}$ is indeed an approximation for $S^{(n)}$. Applying equations \eqref{eq:v-hat-tilde}, \eqref{eq:cP}, \eqref{eq:cTime}, and \eqref{eq:cT2}, we conclude
\begin{align*}
&\left\|\hat{S}^{(n)}_k - \tilde{S}^{(n)}_k\right\|_{E} \leq \left|\hat{B}^{(n)}_k -\tilde{B}^{(n)}_k\right| + \left\|\hat{v}^{(n)}_{b,k} - \tilde{v}^{(n)}_{b,k}\right\|_{L^2} +
\left|\hat{A}^{(n)}_k -\tilde{A}^{(n)}_k\right| + \left\|\hat{v}^{(n)}_{a,k} - \tilde{v}^{(n)}_{a,k}\right\|_{L^2} +\left|\hat{\tau}^{(n)}_k - \tilde{\tau}^{(n)}_k\right|\\
&\qquad \lesssim  \left|\hat{B}^{(n)}_k -\tilde{B}^{(n)}_k\right|+\tn \sum_{j=1}^k\left|\hat{B}^{(n)}_{j-1} - \tilde{B}^{(n)}_{j-1}\right| + \left|\hat{A}^{(n)}_k -\tilde{A}^{(n)}_k\right|+\tn \sum_{j=1}^k\left|\hat{A}^{(n)}_{j-1} - \tilde{A}^{(n)}_{j-1}\right|\\ &\qquad\qquad+\tn \sum_{j=1}^k \left\|S^{(n)}_{j-1} - \tilde{S}^{(n)}_{j-1}\right\|_{E} + \left|\hat{\tau}^{(n)}_k - \tilde{\tau}^{(n)}_k\right|\\
&\qquad \lesssim  \sup_{j\leq k}\left|\hat{B}^{(n)}_j -\tilde{B}^{(n)}_j\right| + \sup_{j\leq k}\left|\hat{A}^{(n)}_j -\tilde{A}^{(n)}_j\right| + \tn \sum_{j=1}^k \left\|S^{(n)}_{j-1} - \tilde{S}^{(n)}_{j-1}\right\|_{E}+ \sup_{j\leq k}\left|\hat{\tau}^{(n)}_j - \tilde{\tau}^{(n)}_j\right|\\
& \qquad \lesssim \tn \sum_{j=1}^k\left\|S^{(n)}_{j-1} - \tilde{S}^{(n)}_{j-1}\right\|_{E}+a_n\\
&\qquad \lesssim \tn \sum_{j=1}^{k} \left(\left\|S^{(n)}_{j-1} - \hat{S}^{(n)}_{j-1}\right\|_{E}+ \left\|\hat{S}^{(n)}_{j-1} - \tilde{S}^{(n)}_{j-1}\right\|_{E}\right)+a_n.
\end{align*}
Applying the discrete version of the Gronwall Lemma, we have for all $k\leq T_n$,
\begin{equation}\label{eq:ineqGW}
\left\|\hat{S}^{(n)}_k - \tilde{S}^{(n)}_k\right\|_{E} \lesssim \tn \sum_{j=1}^k \left\|S^{(n)}_{j-1}-\hat{S}^{(n)}_{j-1}\right\|_{E} + a_n\lesssim T \sup_{j\leq k-1} \left\|S^{(n)}_{j}-\hat{S}^{(n)}_{j}\right\|_{E} + a_n.
\end{equation}
Now equations \eqref{eq:cR}, \eqref{eq:cT1}, and \eqref{eq:ineqGW} yield for $i\leq T_n$,
\begin{align}\label{eq:ineqPr}
\E\left[\sup_{k\leq i} \left(\hat{B}^{(n)}_k - B^{(n)}_k\right)^2\right] &\lesssim \tn \sum_{j=1}^i \E\left\|S^{(n)}_{j-1} - \tilde{S}^{(n)}_{j-1}\right\|^2_{E}+a_n\notag\\
&\lesssim \tn \sum_{j=1}^i \left( \E\left\|S^{(n)}_{j-1} - \hat{S}^{(n)}_{j-1}\right\|^2_{E} + \E\left\|\hat{S}^{(n)}_{j-1} - \tilde{S}^{(n)}_{j-1}\right\|^2_{E}\right)+a_n\notag\\
&\lesssim \tn \sum_{j=1}^i \E\left[\sup_{l \leq j-1} \left\|S^{(n)}_l - \hat{S}^{(n)}_l\right\|_{E}^2\right] + a_n
\end{align}
and a similar estimate holds for the best ask price and its approximation. 
Hence, we can conclude for all $i \leq T_n$ that
\begin{align*}
\E\left[\sup_{k\leq i} \left\|\hat{S}^{(n)}_k - S^{(n)}_k \right\|^2_{E}\right] &\leq \E\left[\sup_{k\leq i} \left(\hat{B}^{(n)}_k-B^{(n)}_k\right)^2\right] +\E\left[\sup_{k\leq i} \left(\hat{A}^{(n)}_k-A^{(n)}_k\right)^2\right]+ \E\left[\sup_{k\leq i} \left| \bar{\tau}^{(n)}_k - \tau^{(n)}_k\right|^2\right]\\
&\hspace{2cm} + 2\left(\E\left[\sup_{k\leq i} \left\|\hat{v}^{(n)}_{b,k} - \bar{v}^{(n)}_{b,k}\right\|_{L^2}^2\right] + \E\left[\sup_{k\leq i} \left\|\bar{v}^{(n)}_{b,k} - v^{(n)}_{b,k}\right\|_{L^2}^2\right]\right)\\
&\hspace{2cm}+2\left(\E\left[\sup_{k\leq i} \left\|\hat{v}^{(n)}_{a,k} - \bar{v}^{(n)}_{a,k}\right\|_{L^2}^2\right] + \E\left[\sup_{k\leq i} \left\|\bar{v}^{(n)}_{a,k} - v^{(n)}_{a,k}\right\|_{L^2}^2\right]\right)\\
& \overset{(1)}{\lesssim} \E\left[\sup_{k\leq i} \left(\hat{B}^{(n)}_k - B^{(n)}_k\right)^2\right] + \E\left[\sup_{k\leq i} \left(\hat{A}^{(n)}_k - A^{(n)}_k\right)^2\right] +a_n\\
& \overset{(2)}{\lesssim} \tn \sum_{j=1}^i \E\left[\sup_{l \leq j-1} \left\|S^{(n)}_l - \hat{S}^{(n)}_l\right\|^2_{E}\right] + a_n.
\end{align*}
In (1) we plugged in the estimate in \eqref{eq:v-bar-hat} and used step 1, while in (2) we used \eqref{eq:ineqPr}. With the discrete Gronwall Lemma we conclude that
\[\E\left[\sup_{k\leq i} \left\|S^{(n)}_k - \hat{S}^{(n)}_k\right\|^2_{E}\right] \lesssim a_n \left(1 +\tn \sum_{j=1}^i e^{\sum_{m=j+1}^i C \tn}\right) =o(1).\]

 \textbf{Step 5:} The result now follows from 
\begin{align*}
\E\left[\sup_{k\leq T_n} \left\|S^{(n)}_k - \tilde{S}^{(n)}_k\right\|^2_{E}\right]& \leq 2\E \left[\sup_{k\leq T_n} \left\|S^{(n)}_k - \hat{S}^{(n)}_k\right\|^2_{E}\right] + 2\E\left[\sup_{k\leq T_n} \left\|\hat{S}^{(n)}_k - \tilde{S}^{(n)}_k\right\|^2_{E}\right]\\
&\overset{(1)}{\lesssim} \E\left[\sup_{k\leq T_n}\left\|S^{(n)}_k - \hat{S}^{(n)}_k\right\|^2_{E}\right] + a_n \rightarrow 0,
\end{align*}
where we used in (1) the inequality in \eqref{eq:ineqGW} and the final convergence follows from step 4.

\subsection{Proof of Theorem \ref{res:conAV}}\label{proof:convSeq}

According to Lemma \ref{lem:uLip}, it is enough to prove the convergence result on the subspace $\tilde{E} \subset E,$ where
\[\tilde{E}:= \left\{(b,v_b, a, v_a, t) \in E:  \|v_I(\cdot + x)- v_I(\cdot + \tilde{x})\|_{L^2} \leq L(1+T)|x-\tilde{x}|  \quad \forall\, x, \tilde{x} \in \R,\ I = b,a \right\}.\]
The space $\tilde{E}$ endowed with the norm $\|\cdot\|_{E}$ is a closed subspace of $E$ and hence again a Banach space. Considering the coefficient functions only on the subspace $\tilde{E}$ will ensure that their limits are still Lipschitz-continuous as compositions of the Lipschitz-continuous functions $p_b,$ $p_a,$ $r_b,$ $r_a$, $\theta_b,$ $\theta_a,$ $\varphi$, and (shifted) $f_b,$ $f_a$ with $\psi$ (cf. Lemma \ref{lem:LG}).
 
 \begin{proof}[\textbf{Proof of Theorem \ref{res:conAV}}]
 We first note that $\tilde{\eta}^{(n), abs}$ takes values in the Banach space $(\tilde{E}, \|\cdot\|_{E})$ by Lemma \ref{lem:uLip}. Hence, we can restrict ourselves to this space and prove the statement of Theorem \ref{res:conAV} only in the closed subspace $\mathcal{D}(\tilde{E}; [0,T])$. We need to show three things: the sequence $(\tilde{\eta}^{(n),abs})_{n\in\N}$ is relative compact, any limit point satisfies \eqref{eq:SDEabs}, and there exists at most one solution $\eta^{abs}$ of \eqref{eq:SDEabs}.\par

In order to prove the first two things, we will apply Theorem 7.6 in \cite{KurtzProtter2}. Let us verify its conditions: 
Thanks to Proposition \ref{res:limitSJ} and Proposition \ref{res:limitLJ}, we have for any $m\in\N$ and $g_1, \cdots, g_m\in C_b([-M,M]^2)$,
\[ (Z_b^{(n)}, Z_a^{(n)}) \Rightarrow (Z_b, Z_a) \quad \text{ and } \quad \left(X^{(n)}(\cdot, g_1), \cdots, X^{(n)}(\cdot, g_m)\right) \Rightarrow \left(X(\cdot, g_1), \cdots, X(\cdot, g_m)\right)\]
 in $\mathcal{D}(\R^2; [0,T])$ and $\mathcal{D}(\R^{2m}; [0,T])$, respectively. Since $(Z_b, Z_a)$ is a standard planar Brownian motion, its paths are almost surely continuous. Hence, Corollary 3.33 in \cite{JacodShiryaev} implies the joint convergence of the integrators $(Z^{(n)}_b, Z^{(n)}_a)$ and  $(X^{(n)}(\cdot, g_1),\cdots, X^{(n)}(\cdot, g_m))$ for any $m\in\N$ and $g_1, \cdots, g_m\in C_b([-M,M]^2)$. By Theorem \ref{res:limitLJ}, $X(\cdot, g)$ is a pure jump L\'evy process for all $g\in C_b([-M,M]^2).$ Hence, the quadratic covariation of $X(\cdot, g)$ with $Z_I, I = b,a,$ is  equal to zero almost surely and $X$ is independent of $(Z_b,Z_a)$. Moreover, the sequence of integrators is uniformly tight by Lemma \ref{res:UT}. By Assumption \ref{ass:IV}, we have
\[S^{(n),abs}_0 \rightarrow \left(B_0, v_{b,0}(-(\cdot - B_0)), A_0, v_{a,0}(\cdot - A_0), 0\right) = S^{abs}_0.\]
Since $S^{abs}_0$ is deterministic, we conclude the joint convergence of $S^{(n),abs}_0,Z_b^{(n)},Z_a^{(n)}$, and $X^{(n)}$.

  As the shift operator $\psi$ maps elements of $\tilde{E}$ to $E$, we can apply Assumptions \ref{ass:limitFct} and \ref{ass:ExQ} to conclude the uniform convergence of the coefficient functions composed with $\psi$. Together with Lemma \ref{lem:LG} this implies that Condition C.2 of \cite{KurtzProtter2} is satisfied. Moreover, as the shift operator $\psi$ maps elements of $\tilde{E}$ to $E$, it is sufficient to prove the boundedness and compactness requirements with respect to the unshifted coefficient functions considered on the larger space $E$. To this end, Assumptions \ref{ass:randTimes} iii), \ref{ass:probVol} iii), \ref{ass:Prob}, and \ref{ass:ExQ} imply that for $I=b,a$,
\begin{equation*}
 \sup_{n\in\N} \sup_{s\in E} \left\{|p_I^{(n)}(s)|+|r_I^{(n)}(s)|+\big\|\theta_I^{(n)}(s,\cdot)\big\|_{L^\infty}+\big\|f_I^{(n)}[s]\big\|_{L^2}+|\varphi^{(n)}(s)| \right\}< \infty.
\end{equation*}
The required compactness condition now follows from the Heine-Borel theorem, the Fr\'echet-Kolmogorov theorem in combination with the uniform equicontinuity and equitightness of $f_I,\ I=b,a$, cf.~Assumption \ref{ass:probVol} iii), and the Arzel\`a-Ascoli theorem in combination with Assumption \ref{ass:ExQ} i). Hence, the requirements of Theorem 7.6 in \cite{KurtzProtter2} are satisfied and we may conclude that the sequence $(\tilde{\eta}^{(n),abs})_{n\in \N}$ is relatively compact and any limit point satisfies \eqref{eq:SDEabs}.\par

 Next, we will show uniqueness of a strong solution to \eqref{eq:SDEabs} by a standard Gronwall argument. Suppose $\eta = (B,v_b, A, v_a, \tau),\, \tilde{\eta} = (\tilde{B}, \tilde{v}_b, \tilde{A}, \tilde{v}_a, \tilde{\tau})$ are two strong solutions to \eqref{eq:SDEabs}.
Then the Lipschitz continuity of the coefficient functions (cf.~Lemma \ref{lem:LG}) implies:
\begin{align*}
    \E&\left[\sup_{s \leq t} \left|B(s)- \tilde{B}(s)\right|^2\right] \\
    &\leq 4 \E\left[\sup_{s \leq t}\left\{ \left|\int_0^s \left(p_b(\psi(\eta(u-))) - p_b(\psi(\tilde{\eta}(u-)))\right) du\right|^2 \right. \right.\\
    &\hspace{3cm} + \left| \int_0^s \left(r_b(\psi(\eta(u-))) - r(\psi(\tilde{\eta}(u-)))\right) dZ_b(u)\right|^2 \\
    &\hspace{3cm}  + \left| \int_0^s \int_{[-M,M]} \left(\theta_b(\psi(\eta(u-)), y) - \theta_b(\psi(\tilde{\eta}(u-)), y) \right) \left(\mu^{Q}_b- \nu^Q_b\right)(du, dy)\right|^2 \\
    &\hspace{3cm} \left.\left.+ \left| \int_0^s \int_{[-M,M]} \left(\theta_b(\psi(\eta(u-)),y ) - \theta_b(\psi(\tilde{\eta}(u-)), y)\right) \nu^Q_b(du, dy) \right|^2 \right\}\right]\\
    & \overset{(1)}{\lesssim}  
    \int_0^t \E\left|p_n(\psi(\eta(u-))) - p_b(\psi(\tilde{\eta}(u-)))\right|^2 du +  \int_0^t \E\left| r_b(\psi(\eta(u-))) - r_b(\psi(\tilde{\eta}(u-)))\right|^2 du\\
    &\qquad+ \int_0^t \int_{[-M,M]} \E \left|\theta_b(\psi(\eta(u-)),y) - \theta_b(\psi(\tilde{\eta}(u-)), y)\right|^2 Q_b(dy)\, du\\
    &\overset{(2)}{\lesssim} \int_0^t \E\left\|\eta(u-)- \tilde{\eta}(u-)\right\|^2_E du.
\end{align*}
Here, in (1) we applied the Burkholder-Davis-Gundy inequality for $p=2$ and in (2) we applied the Lipschitz continuity of the coefficient functions. Similar estimates were already done in the proof of Theorem \ref{res:convApprox} and are thus omitted for the other componentwise differences of $\eta$ and $\tilde{\eta}$, which (up to a multiplicative constant) can be bounded by the same term. Hence, an application of Gronwall's Lemma yields that $\eta = \tilde{\eta}$ almost surely. Now, Corollary 7.8 in \cite{KurtzProtter2} yields the existence of a unique strong solution of \eqref{eq:SDEabs} in the space $(\tilde{E}, \| \cdot\|_{E}).$\par
Hence, $\tilde{\eta}^{(n),abs} \Rightarrow \eta^{abs}$ in $\mathcal{D}(\tilde{E}; [0,T])$, where $\eta^{abs}$ is the unique solution to \eqref{eq:SDEabs}. Since $\tilde{E} \subset E$ is closed, we moreover deduce the weak convergence in $\mathcal{D}(E; [0,T])$.
\end{proof}

\appendix

\section{Auxiliary results}

\begin{Lem}\label{lem:theta}
Let Assumptions \ref{ass:ExQ} i) and v) be satisfied. Then for all $j\in \Z^{(n)}_M$, $I = b,a,$ and $s\in E$,
\[K^{(n)}_I\Big(s,\big\{\theta^{(n)}_I(s,x_j^{(n)})\big\}\Big)=K^{(n)}_I\left(s,\theta_I\Big(s,\left[x_{j}^{(n)},x_{j+1}^{(n)}\right)\Big)\right).\]
\end{Lem}

\begin{proof}
First suppose that $K^{(n)}_I\Big(s,\big\{\theta^{(n)}_I(s,x_j^{(n)})\big\}\Big)>0$. We want to show that in this case,
\begin{equation}\label{eq:thetaInt}
    \theta^{(n)}_I(s,x_j^{(n)})\in\theta_I\Big(s,\left[x_{j}^{(n)},x_{j+1}^{(n)}\right)\Big).
\end{equation}
In fact, by definition
\[\theta^{(n)}_I(s,x_j^{(n)})\in\Big[\theta_I(s,x_j^{(n)}), \theta_I(s,x_j^{(n)}) + \xn\Big).\]
Suppose that $\theta_I(s,x_{j+1}^{(n)})<\theta_I(s,x_j^{(n)})+\xn$ (otherwise, the relation in \eqref{eq:thetaInt} is already satisfied). Then,
\[\theta^{(n)}_I(s,x_{j+1}^{(n)})=\left\lceil\frac{\theta_I(s,x^{(n)}_{j+1})}{\xn}\right\rceil\cdot\xn\leq \left\lceil\frac{\theta_I(s,x^{(n)}_{j})+\xn}{\xn}\right\rceil\cdot\xn=\theta_I^{(n)}(s,x_j^{(n)}) + \xn.\]
By Assumption \ref{ass:ExQ} v), we must have equality in the above line, which implies that 
\[\theta_I(s,x_{j+1}^{(n)}) > \theta^{(n)}_I(s,x_j^{(n)}).\]
Hence, \eqref{eq:thetaInt} is satisfied.\par 
Second, suppose that there exists $x\in \left[x_{j}^{(n)},x_{j+1}^{(n)}\right)$ with $\theta_I(s,x)\in \xn \Z$ and $K^{(n)}_I(s,\theta(s,x))>0$. Then by Assumption \ref{ass:ExQ} v) there exists $i\in\N$ with $\theta_I(s,x)=\theta^{(n)}_I(s,x_i^{(n)})$. By strict  monotonicity of $\theta_I$ (cf.~Assumption \ref{ass:ExQ} i)) we have
\[\theta^{(n)}_I(s,x_j^{(n)})=\left\lceil\frac{\theta_I(s,x_j^{(n)})}{\xn}\right\rceil\cdot\xn\leq \left\lceil\frac{\theta_I(s,x)}{\xn}\right\rceil\cdot\xn=\theta_I(s,x)\]
and
\[\theta_I(s,x)=\left\lceil\frac{\theta_I(s,x)}{\xn}\right\rceil\cdot\xn < \left\lceil\frac{\theta_I(s,x_{j+1}^{(n)})}{\xn}\right\rceil\cdot\xn=\theta^{(n)}_I(s,x_{j+1}^{(n)}).\]
Hence, we must have $i=j$, i.e.~$\theta_I(s,x)=\theta^{(n)}_I(s,x_j^{(n)})$.
\end{proof}

\begin{Lem}\label{lem:uLip}
Under Assumptions \ref{ass:IV} and \ref{ass:probVol} the absolute volume dynamics introduced in \eqref{def:absVol} satisfy for all $n \in \N$, $k=0,\dots,T_n$, and $I=b,a$ the estimate
\[\Big\|\tilde{u}^{(n)}_{I,k}(\cdot + x) - \tilde{u}^{(n)}_{I,k}(\cdot + \tilde{x})\Big\|_{L^2} \leq L(1+T)|x - \tilde{x}|\quad \forall x,\tilde{x}\in\R.\]
\end{Lem}

\begin{proof}
By \eqref{ass:Lcond} and Assumption \ref{ass:probVol} iii), we derive for all $n \in \N,$ $k=0,\dots,T_n$, and $x, \tilde{x} \in \R,$
\begin{align*}
\Big\|\tilde{u}^{(n)}_{I,k}(\cdot + x) &- \tilde{u}^{(n)}_{I,k}(\cdot + \tilde{x})\Big\|_{L^2}\\
&\leq \left\|v_{I,0}(\cdot + x) - v_{I,0}(\cdot + \tilde{x}) \right\|_{L^2} + \tn \sum_{j=1}^k\left\|f_I[\tilde{S}^{(n)}_{j-1}](\cdot + x) - f_I[\tilde{S}^{(n)}_{j-1}](\cdot + \tilde{x})\right\|_{L^2}\\
&\leq L|x-\tilde{x}| + T\sup_{s \in E} \|f_I[s](\cdot +x) - f_I[s](\cdot + \tilde{x})\|_{L^2}\\
&\leq L(1+T) |x - \tilde{x}|.
\end{align*}
\end{proof}
  
\begin{Lem} \label{lem:LG}
Let Assumptions \ref{ass:randTimes}, \ref{ass:probVol}, \ref{ass:limitFct}, and \ref{ass:ExQ} be satisfied. Then the composed coefficient functions $p_b\circ\psi,p_a\circ\psi,r_b\circ\psi,r_a\circ\psi,f_b\circ\psi,f_a\circ\psi,\theta_b\circ\psi,\theta_a\circ\psi$, and $\varphi\circ\psi$ restricted to the subspace $\tilde{E}$ are Lipschitz continuous.
\end{Lem}

\newpage 
  
\begin{proof}
First note that for any $s=(b,v_b, a, v_a, t),\ \tilde{s}=(\tilde{b},\tilde{v}_b, \tilde{a}, \tilde{v}_a, \tilde{t})\in \tilde{E}$,
\begin{align*}
\|\psi(\tilde{s})- \psi(s)\|^2_{E}&=\Big\{ |\tilde{b} - b|^2 +\left\|\tilde{v}_b(-(\cdot - \tilde{b})) - v_b(-(\cdot - b))\right\|^2_{L^2}\\
&\hspace{2cm} + |\tilde{a}-a|^2 + \left\|\tilde{v}_a(\cdot-\tilde{a}) - v_a(\cdot - a)\right\|^2_{L^2} + |\tilde{t} - t|^2\Big\}\\
&\leq \left\{\|\tilde{s}-s\|_{E}^2 +  \left\|v_b(-(\cdot - \tilde{b})) - v_b( -(\cdot - b))\right\|^2_{L^2} + \left\|v_a(\cdot -\tilde{a}) - v_a(\cdot - a)\right\|^2_{L^2}\right\}\\
&\leq \left\{\|\tilde{s}-s\|_{E}^2 + (L^2(1+T)^2)\left(|\tilde{b}-b|^2 +|\tilde{a} - a|^2\right)\right\}\\
& \leq (1+L^2(1+T)^2)\left\|\tilde{s}-s\right\|^2_{E}.
\end{align*}
Now the claim follows from the Lipschitz continuity of $p_b, p_a, r_b, r_a, f_b, f_a, \theta_b,\theta_a$, and $\varphi$ on the larger space $E$, cf.~Assumptions \ref{ass:randTimes} iii), \ref{ass:probVol} iii), \ref{ass:limitFct}, and \ref{ass:ExQ} v).
\end{proof}

 \begin{Lem}\label{res:UT}
Suppose Assumptions \ref{ass:Prob}, \ref{ass:ExQ}, and \ref{ass:Scaling} are satisfied.
Then the sequence $(Z_b^{(n)},Z_a^{(n)},X^{(n)})_{n\in\N}$ is uniformly tight in the sense of \cite{KurtzProtter2}.
\end{Lem}

\begin{proof}
First, consider an arbitrary $(\F_k)$-adapted, real-valued process $(g^{(n)}_k)_{k=0,\dots,T_n}$ satisfying the bound $\sup_{k\leq T_n} |g^{(n)}_k|\leq 1$ a.s. Then for $I=b,a$ and any $t\in[0,T]$,
\[\E\left(\sum_{k=1}^{\ttn} g^{(n)}_{k-1}\delta Z^{(n)}_ {I,k}\right)^2 = \tn \sum_{k=1}^{\ttn} \E\left(g^{(n)}_{k-1}\right)^2 \leq t.\]
Next, consider an arbitrary $(\F_k)$-adapted, $C_b([-M,M])$-valued process $(h^{(n)}_k(\cdot))_{k=0,\dots,T_n}$ satisfying the bound $\sup_{k\leq T_n} \|h^{(n)}_k(\cdot)\|_{L^\infty}\leq 1$ a.s. Applying the estimate in Remark \ref{rmk:jumpEstimate}, we have for $I=b,a$ and any $t\in[0,T]$,
\begin{align*}
 \E\left|\sum_{k=1}^{\ttn} \int_{[-M,M]} h^{(n)}_{k-1}(y)  \mu^{J^{(n)}}_I\left(\left[\t_k, \t_{k+1}\right) \times dy\right)\right|
&\leq \E\left[\sum_{k=1}^{\ttn} \int_{[-M,M]}  \nu_I^{J^{(n)}}\left(\left[\t_k, \t_{k+1}\right) \times dy\right)\right]\\
 &\leq \sum_{k=1}^{\ttn}\tn \left[Q_I([-M,M])+ a_n\right]\\
 &\leq \left(Q_I([-M,M])+ a_n\right) \, t \leq Ct
\end{align*}
for some $C>0$, independent of $n$, since $Q_I$ is a finite measure by Assumption \ref{ass:ExQ} and $(a_n)_{n\in \N}$ is a deterministic null-sequence. Therefore, all four integrals considered above are 
bounded in probability for all $t \in [0,T]$, uniformly in $n\in\N$. As $g^{(n)}$ and $h^{(n)}$ were arbitrary, this establishes uniform tightness of $(Z_b^{(n)},Z_a^{(n)},X^{(n)})_{n\in\N}$.
 \end{proof}
 
 The following technical lemma is used to prove Proposition \ref{res:conRemainder} and provides an estimate for Hilbert space valued, discrete time martingales. It is a direct consequence of Theorem 6.1 in \cite{Pisier86}:

\begin{Lem}\label{techL:L2est}
Let $H$ be a Hilbert space. Then there exists a constant $C>0$ such that for every $H$-valued, discrete time martingale $X$ with $X_0=0$, we have
\[    \E\left(\sup_{i\geq 1}\|X_i\|^2_H\right) \leq C\E\left[\sum_{i=1}^{\infty} \|X_i - X_{i-1}\|^2_H\right].\]
\end{Lem}

\section*{Acknowledgement}
Financial support by MATH+ through project funding AA4-4 ‘‘Stochastic modeling of intraday electricity markets’’ is gratefully acknowledged.

\bibliographystyle{plain}
\bibliography{lit}
\end{document}